\documentclass[twoside]{article}

\usepackage[preprint]{aistats2027}

\usepackage[round]{natbib}

\usepackage[utf8]{inputenc}
\usepackage{microtype}
\usepackage{amssymb}
\usepackage{mathtools}
\usepackage{amsthm}
\usepackage{graphicx}
\usepackage{booktabs}
\usepackage{float}
\usepackage{xcolor}
\usepackage{enumitem}
\usepackage{algorithm}
\usepackage[noend]{algorithmic}
\usepackage{tikz}
\usetikzlibrary{arrows.meta}
\usepackage[colorlinks=true,linkcolor=blue!60!black,citecolor=blue!60!black,urlcolor=blue!60!black]{hyperref}

\theoremstyle{plain}
\newtheorem{theorem}{Theorem}[section]
\newtheorem{proposition}[theorem]{Proposition}
\newtheorem{lemma}[theorem]{Lemma}
\newtheorem{corollary}[theorem]{Corollary}
\theoremstyle{definition}
\newtheorem{definition}[theorem]{Definition}
\newtheorem{assumption}[theorem]{Assumption}
\theoremstyle{remark}
\newtheorem{remark}[theorem]{Remark}

\newcommand{\R}{\mathbb{R}}
\newcommand{\E}{\mathbb{E}}

\DeclareMathOperator*{\argmin}{argmin}
\DeclareMathOperator{\Beta}{Beta}
\DeclareMathOperator{\Exp}{Exp}

\setlist{itemsep=1pt,topsep=2pt,leftmargin=1.2em}

\begin{document}
\runningtitle{RanSOM: Second-Order Momentum with Randomized Scaling}

\twocolumn[
\aistatstitle{RanSOM: Second-Order Momentum with Randomized Scaling\\for Constrained and Unconstrained Optimization}
\aistatsauthor{El Mahdi Chayti \And Mohamad Fakhouri }
\aistatsaddress{Machine Learning and Optimization Laboratory (MLO), EPFL \And Independent Researcher}
]

\begin{abstract}
Second-order momentum corrects the lag of momentum behind the current gradient with a Hessian-vector product (HVP), but existing corrections need either a Lipschitz Hessian or extra forward and backward passes at an auxiliary point. We show that randomizing the step length removes this trade-off: if the step is exponentially distributed with mean $\eta$, a Stein-type identity makes $\eta$ times one HVP at the new iterate an estimate of the gradient change that is unbiased in expectation over the step, and this HVP shares its forward pass with the next gradient. The resulting method, RanSOM-E, works with normalized, sign, and Muon-type steps, and a Beta-distributed variant, RanSOM-B, keeps Frank-Wolfe iterates feasible. For objectives with Lipschitz gradient and heavy-tailed gradient and Hessian noise with tail indices $p, q \in (1,2]$, RanSOM-E attains the optimal oracle-complexity exponents for every pair $(p,q)$, including $p \neq q$, without a Lipschitz Hessian or clipping ($\mathcal{O}(\varepsilon^{-3})$ for $p = q = 2$). On CIFAR-10 with ResNet-20, we detect no accuracy difference from existing second-order momentum methods, and its mean update time is lower than that of the auxiliary-point method.
\end{abstract}

\section{INTRODUCTION}
\label{sec:intro}

We study stochastic optimization problems
\begin{equation*}
  \min_{x \in \mathcal{C}} \; f(x) = \E_\xi[f_\xi(x)],
\end{equation*}
where $f$ is nonconvex, $\mathcal{C}$ is either $\R^d$ or a compact convex set, and the algorithm has access to stochastic gradients $\nabla f_\xi(x)$ and stochastic Hessian-vector products (HVPs) $\nabla^2 f_\xi(x)\,v$.

Momentum methods \citep{polyak1964, sutskever2013importance} maintain a weighted average $m_t$ of past stochastic gradients. Averaging reduces the variance of the gradient estimate, but the averaged gradients were evaluated at earlier iterates, so along a curved trajectory $m_t$ lags behind $\nabla f(x_t)$ (Figure~\ref{fig:bias_diagram}). For one step the lag is the change of the gradient,
\begin{equation}
  \label{eq:bias_def}
  \begin{aligned}
  &\nabla f(x_{t+1}) - \nabla f(x_t) \\
  &\quad= \int_0^1 \nabla^2 f\bigl(x_t + \tau(x_{t+1} - x_t)\bigr)(x_{t+1}-x_t)\, d\tau .
  \end{aligned}
\end{equation}
Without a correction, the guarantees for stochastic momentum match those of SGD, $\mathcal{O}(\varepsilon^{-4})$ stochastic oracle calls to reach an $\varepsilon$-stationary point \citep{liu2020improved}. Variance-reduced momentum adds an estimate $\delta_{t+1}$ of~\eqref{eq:bias_def} to the buffer,
\begin{equation}
  \label{eq:mom_template}
  m_{t+1} = (1-\beta)\bigl(m_t + \delta_{t+1}\bigr) + \beta\, g_{t+1},
\end{equation}
where $g_{t+1}$ is a stochastic gradient at $x_{t+1}$ and $\beta \in (0,1]$ is the averaging weight. The momentum error $e_t = m_t - \nabla f(x_t)$ then satisfies
\begin{multline*}
  e_{t+1} = (1-\beta)\bigl(e_t + \delta_{t+1} - [\nabla f(x_{t+1}) - \nabla f(x_t)]\bigr) \\
  + \beta\bigl(g_{t+1} - \nabla f(x_{t+1})\bigr).
\end{multline*}
If $\delta_{t+1}$ is an unbiased estimate of the gradient change given the past, both brackets are zero-mean noise and averaging over many steps reduces them; with suitable moment bounds, the oracle complexity improves to $\mathcal{O}(\varepsilon^{-3})$ \citep{cutkosky2019storm, tran2022better}. This complexity is optimal for algorithms with gradient and HVP oracles \citep{arjevani2020second}.

\begin{figure}[t]
  \centering
  \begin{tikzpicture}[scale=1.55, >=Stealth]
    \draw[thick, gray, ->] (0,1.0) to[out=-60, in=170] (2,0.35);
    \fill (0,1.0) circle (0.03) node[left] {$x_t$};
    \fill (2,0.35) circle (0.03) node[below] {$x_{t+1}$};
    \draw[->, thick, red!70!black, dashed] (2,0.35) -- (2.75,0.85) node[above] {\small $m_t \approx \nabla f(x_t)$};
    \draw[->, thick, blue!70!black] (2,0.35) -- (2.8,0.05) node[right] {\small $\nabla f(x_{t+1})$};
    \draw[->, thick, orange!90!black] (2.75,0.85) -- (2.8,0.08);
    \node[orange!90!black, right] at (2.78,0.5) {\small lag~\eqref{eq:bias_def}};
  \end{tikzpicture}
  \caption{The momentum $m_t$ averages gradients from earlier iterates and approximates $\nabla f(x_t)$. After the step to $x_{t+1}$, it differs from the new gradient by the change~\eqref{eq:bias_def}, which the correction $\delta_{t+1}$ estimates.}
  \label{fig:bias_diagram}
\end{figure}

Three estimators of~\eqref{eq:bias_def} are in use (Table~\ref{tab:comparison}). STORM \citep{cutkosky2019storm} takes the difference of two stochastic gradients on the same sample, $\nabla f_\xi(x_{t+1}) - \nabla f_\xi(x_t)$; its error is controlled through mean-squared smoothness of the sample functions $f_\xi$. Second-order momentum \citep{tran2022better, khirirat2025better}, which we call SOM, takes the first-order Taylor term $\nabla^2 f_\xi(x_{t+1})(x_{t+1}-x_t)$. Its HVP is evaluated where the next gradient is needed and shares its forward computation, but the Taylor error is controlled through a Lipschitz Hessian. SOM-Unif \citep{zhang2020one, salehkaleybar2022momentum, sadiev2025second, khirirat2025better} evaluates the HVP at an auxiliary point $\hat x_t$ drawn uniformly from the segment $[x_t, x_{t+1}]$. By~\eqref{eq:bias_def} this is unbiased without a Lipschitz Hessian, at the price of a second forward and backward pass at $\hat x_t$, a point that is not on the trajectory.

\paragraph{This work.} We randomize the step length instead of the evaluation point. Write the update as $x_{t+1} = x_t + s_t d_t$, where $d_t$ is a direction computed from $m_t$ and $s_t \ge 0$ is the step length. If $s_t$ is exponentially distributed with mean $\eta_t$ and independent of the past, a Stein-type identity (Lemma~\ref{lemma:stein}) gives
\begin{equation*}
  \E_{s_t}\bigl[\nabla f(x_{t+1}) - \nabla f(x_t)\bigr]
  = \eta_t\, \E_{s_t}\bigl[\nabla^2 f(x_{t+1})\, d_t\bigr].
\end{equation*}
The correction $\delta_{t+1} = \eta_t \nabla^2 f_\xi(x_{t+1}) d_t$ therefore has the cost of SOM and the unbiasedness of SOM-Unif. Our contributions are:
\begin{itemize}
  \item \textbf{Randomized-step corrections} (Section~\ref{sec:method}). With an exponential or Beta step length, one HVP at the new iterate is an unbiased estimate of the gradient change. The exponential law is the only law with a deterministic weight; the Beta law keeps Frank-Wolfe iterates feasible.
  \item \textbf{Algorithms.} RanSOM-E (Algorithm~\ref{alg:expsom}) uses linear minimization oracle (LMO) steps over a norm ball, which cover normalized, sign, and Muon-type updates; RanSOM-B (Algorithm~\ref{alg:betasom}) is a stochastic Frank-Wolfe method.
  \item \textbf{Optimal exponents} (Section~\ref{sec:theory}). Under a Lipschitz gradient and noise with bounded $p$-th (gradient) and $q$-th (HVP) moments, the oracle complexity of RanSOM-E has the optimal exponents for every $(p,q) \in (1,2]^2$ (Corollary~\ref{cor:oracle_complexity}), without a Lipschitz Hessian or clipping. RanSOM-B attains the same exponent for the Frank-Wolfe gap.
  \item \textbf{Experiments} (Section~\ref{sec:experiments}). On CIFAR-10 with ResNet-20, we detect no accuracy difference between RanSOM-E and SOM or SOM-Unif in either geometry, and RanSOM-E has a lower mean update time than SOM-Unif.

  \item \textbf{General setting} (Appendix~\ref{app:theory}). The analysis extends to $(L_0,L_1)$-smoothness, affine noise, and minibatches; the part of the correction error caused by the random step does not decrease with the batch size.
\end{itemize}

\begin{table*}[t]
\centering
\small
\caption{Estimators of the gradient change~\eqref{eq:bias_def} used in~\eqref{eq:mom_template}. For SOM and SOM-Unif the step is deterministic, $x_{t+1} - x_t = \eta_t d_t$; for RanSOM-E it is $s_t d_t$ with $s_t \sim \mathrm{Exp}$ of mean $\eta_t$. The published analyses of the four methods give $\mathcal{O}(\varepsilon^{-3})$ under bounded variance and their respective assumptions. The last column counts automatic-differentiation passes per step for the gradient and the correction in a reverse-over-reverse implementation: $F$ is a forward pass, $B_g$ a backward pass that builds the gradient graph, and $B_h$ the backward pass through that graph that returns the HVP \citep{pearlmutter1994fast}. SOM-Unif needs a second forward and gradient pass at $\hat x_t$. Measured times are in Figure~\ref{fig:cifar10-update-time} and Table~\ref{tab:movielens_timing}.}
\label{tab:comparison}
\begin{tabular}{@{}lllll@{}}
\toprule
Method & Correction $\delta_{t+1}$ & Evaluated at & Smoothness beyond $\nabla f$ Lipschitz & AD passes \\
\midrule
STORM & $\nabla f_\xi(x_{t+1}) - \nabla f_\xi(x_t)$ & $x_{t+1}$ and $x_t$ & mean-squared smoothness of $f_\xi$ & $2F + 2B_g$ \\
SOM & $\eta_t\,\nabla^2 f_\xi(x_{t+1})\, d_t$ & $x_{t+1}$ & Lipschitz $\nabla^2 f$ & $F + B_g + B_h$ \\
SOM-Unif & $\eta_t\,\nabla^2 f_\xi(\hat x_t)\, d_t$, $\hat x_t \in [x_t, x_{t+1}]$ & $\hat x_t$ & none & $2F + 2B_g + B_h$ \\
\textbf{RanSOM-E} & $\eta_t\,\nabla^2 f_\xi(x_{t+1})\, d_t$ & $x_{t+1}$ & none & $F + B_g + B_h$ \\
\bottomrule
\end{tabular}
\end{table*}

\section{RELATED WORK}
\label{sec:related_work}

\textbf{Variance-reduced momentum.} SPIDER \citep{fang2018spider}, built on the recursive estimator of SARAH \citep{nguyen2017sarah}, attains $\mathcal{O}(\varepsilon^{-3})$ with periodic large batches; STORM \citep{cutkosky2019storm} attains it with a single sample per step. Normalized momentum with an extrapolation step reaches $\mathcal{O}(\varepsilon^{-3.5})$ under second-order smoothness, without HVPs \citep{cutkosky2020momentum}. The analyses of STORM-type estimators require mean-squared smoothness of the sample functions, or its $(L_0,L_1)$ generalization \citep{chen2023generalized}.

\textbf{Second-order momentum.} SGDHess \citep{tran2022better} corrects momentum with an HVP at the current iterate under a Lipschitz Hessian. Evaluating the HVP at a uniformly drawn auxiliary point was introduced for stochastic Frank-Wolfe by \citet{zhang2020one} and used in reinforcement learning \citep{salehkaleybar2022momentum} and under heavy-tailed noise \citep{sadiev2025second}. \citet{khirirat2025better} analyze SOM and SOM-Unif with LMO steps under $(L_0,L_1)$-smoothness, SOM with a Lipschitz-type Hessian condition. \citet{fradin2025tight} prove lower bounds for heavy-tailed gradient and Hessian noise with separate tail indices, and a near-optimal upper bound with the auxiliary-point estimator.

\textbf{LMO-based methods.} Normalized SGD, signSGD \citep{bernstein2018signsgd}, and Muon \citep{jordan2024muon} are steepest-descent or LMO steps in different norms \citep{pethick2025training, kovalev2025non}. Without a correction, momentum in these geometries has $\mathcal{O}(\varepsilon^{-4})$ guarantees.

\textbf{Heavy tails and generalized smoothness.} Gradient noise in deep learning can be heavy-tailed \citep{simsekli2019tail}, and normalization handles bounded $p$-th moments without clipping \citep{hubler2025gradient, liu2023nonconvex}. $(L_0,L_1)$-smoothness \citep{zhang2019gradient, chen2023generalized} lets the Hessian norm grow with the gradient norm.

\textbf{Random step lengths.} \citet{zhang2024random} use exponentially distributed scaling to convert online learning into nonsmooth nonconvex optimization. We use random step lengths for a different purpose: to make a single HVP at the new iterate an unbiased estimate of the gradient change. Stein-type identities are classical \citep{stein1972bound}; ours are integration-by-parts formulas for the survival function of the step. This line of work is distinct from momentum for second-order methods such as stochastic cubic Newton \citep{chayti2024improving}.

\section{THE RANSOM FRAMEWORK}
\label{sec:method}

\subsection{Randomized Integration by Parts}

Fix $x_t$ and a direction $d_t$, and let $g(s) = \nabla f(x_t + s d_t)$, so that $g'(s) = \nabla^2 f(x_t + s d_t)\, d_t$. For a random step $s$, the gradient change to estimate is $\E[g(s) - g(0)]$.

\begin{lemma}[Stein-type identities]
\label{lemma:stein}
Let $g: \R \to \R^d$ be continuously differentiable and let $s \ge 0$ be a random variable with survival function $\bar F_s(z) = \mathbb{P}(s > z)$ such that $\int_0^\infty \|g'(z)\|_2\, \bar F_s(z)\, dz < \infty$.
\begin{enumerate}
  \item If $s \sim \Exp(\lambda)$, then
  \[ \E[g(s) - g(0)] = \tfrac{1}{\lambda}\,\E[g'(s)]. \]
  \item If $s \sim \Beta(1, K)$, then
  \[ \E[g(s) - g(0)] = \E\bigl[\tfrac{1-s}{K}\, g'(s)\bigr]. \]
\end{enumerate}
\end{lemma}

\begin{proof}[Proof sketch]
Writing $g(s) - g(0) = \int_0^s g'(z)\,dz$ and applying Tonelli's theorem gives $\E[g(s) - g(0)] = \int_0^\infty g'(z)\,\bar F_s(z)\,dz$. For the exponential law $\bar F_s = f_s/\lambda$, where $f_s$ is the density; for $\Beta(1,K)$, $\bar F_s(z) = (1-z)^K = \tfrac{1-z}{K} f_s(z)$ on $[0,1]$. The integrability condition holds automatically for the Beta law and, for the exponential law, under the step-size condition of the analysis (Appendix~\ref{app:stein}).
\end{proof}

With $\lambda = 1/\eta_t$, the first identity states that $\eta_t\, \nabla^2 f(x_{t+1})\, d_t$ is unbiased for $\nabla f(x_{t+1}) - \nabla f(x_t)$ when $x_{t+1} = x_t + s_t d_t$, in expectation over $s_t$ given $x_t$ and $d_t$; it is not unbiased conditionally on the realized step. The weight is the mean $\eta_t$, not the realized $s_t$: replacing it by $s_t$ gives a biased estimate.

\paragraph{Why the deterministic step is biased.} With $s \equiv \eta$, as in SOM, the correction $\eta\,\nabla^2 f(x_t + \eta d)\,d$ differs from~\eqref{eq:bias_def} by
\begin{equation*}
  \int_0^\eta \bigl[\nabla^2 f(x_t + \eta d) - \nabla^2 f(x_t + z d)\bigr] d\, dz ,
\end{equation*}
which is controlled only by the variation of the Hessian along the step; this is where SOM needs a Lipschitz Hessian. A random step weights each point $x_t + zd$ of the ray by $\mathbb{P}(s > z)$, and Lemma~\ref{lemma:stein} shows that this weighting is reproduced exactly by an HVP at the random endpoint multiplied by a suitable weight.

\paragraph{General step laws.} For any step law with a density $f_s$ that is positive wherever the survival function is positive, the same computation gives $\E[g(s) - g(0)] = \E[w(s)\,g'(s)]$ with the Stein weight $w(s) = \bar F_s(s)/f_s(s)$, and $\E[w(s)] = \int_0^\infty \bar F_s = \E[s]$. The weight is deterministic exactly when the hazard rate $f_s/\bar F_s$ is constant, that is, for the exponential law; RanSOM-E uses this case, so the realized step only moves the evaluation point. Laws with bounded support have random weights: $w(s) = 2\eta - s$ for $s$ uniform on $[0, 2\eta]$, and $w(s) = (1-s)/K$ for $\Beta(1,K)$, whose mean is $\eta = 1/(1+K)$. RanSOM-B uses the Beta law because a step in $[0,1]$ keeps Frank-Wolfe iterates feasible.

\subsection{RanSOM-E: Unconstrained Problems}

RanSOM-E (Algorithm~\ref{alg:expsom}) chooses the direction with an LMO over the ball $\mathcal{B}_\rho = \{v : \|v\| \le \rho\}$ of a norm $\|\cdot\|$, $d_t \in \argmin_{v \in \mathcal{B}_\rho}\langle m_t, v\rangle$, so that $\|d_t\| \le \rho$ whatever the size of $m_t$. The choice of norm gives familiar methods:
\begin{itemize}
  \item $\ell_2$: $d_t = -\rho\, m_t/\|m_t\|_2$, normalized momentum \citep{cutkosky2020momentum};
  \item $\ell_\infty$: $d_t = -\rho\,\mathrm{sign}(m_t)$, sign momentum \citep{bernstein2018signsgd};
  \item spectral norm of each weight matrix: with $m_t = U\Sigma V^\top$, $d_t = -\rho\, U V^\top$, a Muon-type step \citep{jordan2024muon}, computed in practice by Newton-Schulz iterations. Our analysis assumes an exact LMO.
\end{itemize}
For a network with parameter blocks $u = (u_1, \ldots, u_L)$, such as the layers, one uses a norm per block and the product ball $\{v : \|v_l\|_{(l)} \le \rho \text{ for all } l\}$, which is the ball of $\|u\| = \max_l \|u_l\|_{(l)}$. The LMO then acts block by block, so normalized, sign, and spectral steps can be mixed across layers. When $m_t = 0$ we set $d_t = 0$. With $\delta_{t+1} = 0$ and $s_t = \eta_t$, RanSOM-E reduces to these uncorrected methods. RanSOM-E changes only the correction and the step length, so each geometry has a direct uncorrected counterpart.

\begin{algorithm}[t]
\caption{RanSOM-E}
\label{alg:expsom}
\small
\begin{algorithmic}[1]
\STATE {\bfseries Input:} $x_0$, mean steps $\eta_t$, weight $\beta$, radius $\rho$, initial batch size $B_{\mathrm{init}}$
\STATE $m_0 = \frac{1}{B_{\mathrm{init}}}\sum_{\xi \in \mathcal{B}_0}\nabla f_\xi(x_0)$
\FOR{$t = 0$ {\bfseries to} $T-1$}
\STATE $d_t \in \argmin_{\|v\| \le \rho}\langle m_t, v\rangle$
\STATE Sample $s_t \sim \Exp$ with mean $\eta_t$
\STATE $x_{t+1} = x_t + s_t d_t$
\STATE Sample $\xi_{t+1}$; with one forward pass at $x_{t+1}$, compute $g_{t+1} = \nabla f_{\xi_{t+1}}(x_{t+1})$
\STATE \quad and $h_{t+1} = \nabla^2 f_{\xi_{t+1}}(x_{t+1})\, d_t$
\STATE $\delta_{t+1} = \eta_t\, h_{t+1}$
\STATE $m_{t+1} = (1-\beta)(m_t + \delta_{t+1}) + \beta\, g_{t+1}$
\ENDFOR
\end{algorithmic}
\end{algorithm}

\begin{algorithm}[t]
\caption{RanSOM-B}
\label{alg:betasom}
\small
\begin{algorithmic}[1]
\STATE {\bfseries Input:} $x_0 \in \mathcal{C}$, mean steps $\eta_t \in (0, \frac12]$, weight $\beta$, initial batch size $B_{\mathrm{init}}$
\STATE $m_0 = \frac{1}{B_{\mathrm{init}}}\sum_{\xi \in \mathcal{B}_0}\nabla f_\xi(x_0)$
\FOR{$t = 0$ {\bfseries to} $T-1$}
\STATE $v_t \in \argmin_{v \in \mathcal{C}}\langle m_t, v\rangle$; $d_t = v_t - x_t$
\STATE $K_t = 1/\eta_t - 1$; sample $s_t \sim \Beta(1, K_t)$
\STATE $x_{t+1} = x_t + s_t d_t$
\STATE Sample $\xi_{t+1}$; compute $g_{t+1}$ and $h_{t+1}$ at $x_{t+1}$ as in Algorithm~\ref{alg:expsom}
\STATE $\delta_{t+1} = \frac{1-s_t}{K_t}\, h_{t+1}$
\STATE $m_{t+1} = (1-\beta)(m_t + \delta_{t+1}) + \beta\, g_{t+1}$
\ENDFOR
\end{algorithmic}
\end{algorithm}

\subsection{RanSOM-B: Constrained Problems}

For a compact convex set $\mathcal{C}$, RanSOM-B (Algorithm~\ref{alg:betasom}) takes a Frank-Wolfe direction $d_t = v_t - x_t$ with $v_t$ given by the LMO of $\mathcal{C}$, which only requires minimizing a linear function over $\mathcal{C}$ (for a nuclear-norm ball, one top singular pair). Since $s_t \in [0,1]$, the update $x_{t+1} = (1-s_t)x_t + s_t v_t$ is a convex combination and stays in $\mathcal{C}$; an exponential step could leave $\mathcal{C}$. The mean step $\eta_t = 1/(1+K_t)$ plays the role of the Frank-Wolfe step size. Compared with the one-sample stochastic Frank-Wolfe method of \citet{zhang2020one}, which evaluates its HVP at an auxiliary point of $[x_t, x_{t+1}]$, RanSOM-B evaluates it at $x_{t+1}$.

\subsection{Computing the Correction}
\label{sec:cost}

The HVP is obtained by differentiating the inner product of the gradient with the fixed direction $d_t$,
\begin{equation*}
  h_{t+1} = \nabla_x \bigl\langle \nabla f_{\xi_{t+1}}(x),\, d_t\bigr\rangle\big|_{x = x_{t+1}},
\end{equation*}
without forming the Hessian \citep{pearlmutter1994fast}. In a reverse-over-reverse implementation, one forward pass at $x_{t+1}$ and one backward pass with the graph retained give $g_{t+1}$, and a second backward pass gives $h_{t+1}$. The per-step oracle cost is one gradient and one HVP, as for SOM. SOM-Unif has the same oracle count but evaluates its HVP at $\hat x_t \ne x_{t+1}$, which requires an additional forward pass and gradient backward pass at $\hat x_t$.

\paragraph{Arithmetic and wall-clock cost.} In arithmetic, Pearlmutter's method makes an HVP a small constant multiple of a gradient. Wall-clock time also depends on memory traffic: the pass $B_h$ traverses the graph built by $B_g$, so activations and intermediate derivatives must be stored and read again, and automatic-differentiation frameworks do not fuse it with the gradient computation. The overhead therefore depends on the architecture and the implementation; in our ResNet-20 runs, an update of an HVP-based method took more than three times as long as a first-order update (Section~\ref{sec:cifar10}). This overhead is shared by SOM, SOM-Unif and RanSOM; SOM-Unif additionally pays the passes at $\hat x_t$.

\paragraph{Implementation.} Two points matter in an implementation: the realized $s_t$ must move the iterate while the correction uses the mean $\eta_t$, and $d_t$ is computed from $m_t$ before $\xi_{t+1}$ is drawn, so it is independent of the oracle noise at step $t+1$. Drawing $s_t$ costs one scalar sample per step.

\paragraph{Unbounded steps.} The exponential law has unbounded support, but long steps are rare: $\mathbb{P}(s_t > k\eta_t) = e^{-k}$, about $0.7\%$ for $k = 5$, and a step has length $s_t\|d_t\| \le s_t\rho$. The analysis uses only moments of $s_t$, $\E[s_t^q] = \Gamma(q+1)\eta_t^q$ with $\Gamma$ the Gamma function, and, when $L_1 > 0$, exponential moments of $s_t$, which are finite under the step-size condition of Section~\ref{sec:theory}. We do not truncate $s_t$: truncation would bias the correction.

\section{CONVERGENCE ANALYSIS}
\label{sec:theory}

\subsection{Setting and Assumptions}

We work on $\mathcal{E} = \R^d$ with a norm $\|\cdot\|$, its dual $\|g\|_* = \sup_{\|x\| \le 1}\langle g, x\rangle$, and the operator norm $\|H\|_{\mathrm{op}} = \sup_{\|u\| \le 1}\|Hu\|_*$. Let $\mathcal{F}_t$ be the $\sigma$-field generated by everything up to step $t$, so that $x_t$, $m_t$ and $d_t$ are $\mathcal{F}_t$-measurable, while $s_t$ and the fresh sample $\xi_{t+1}$ are drawn independently of $\mathcal{F}_t$ and of each other. A point $x$ is $\varepsilon$-stationary if $\|\nabla f(x)\|_* \le \varepsilon$. We write $a \lesssim b$ when $a \le c\, b$ for a constant $c$ that depends only on $p$ and $q$, and $a \asymp b$ when $a \lesssim b \lesssim a$.

\begin{assumption}[Lower bound]
\label{ass:well_defined}
$f_* = \inf f > -\infty$, and $\Delta_0 \triangleq f(x_0) - f_*$.
\end{assumption}

\begin{assumption}[Pointwise $(L_0,L_1)$-smoothness]
\label{ass:geometric_smoothness}
$f$ is twice continuously differentiable and $\|\nabla^2 f(x)\|_{\mathrm{op}} \le L_0 + L_1\|\nabla f(x)\|_*$ for all $x$.
\end{assumption}

\begin{assumption}[Norm compatibility]
\label{ass:norm_compatibility}
There is $\kappa \ge 1$ with $\|u\|_* \le \kappa\|u\|_2$ and $\|u\|_2 \le \kappa\|u\|_*$ for all $u$.
\end{assumption}

\begin{assumption}[Heavy-tailed noise]
\label{ass:noise}
The oracles $\nabla f_\xi(x)$ and $\nabla^2 f_\xi(x) w$ are unbiased, and for some $p, q \in (1,2]$ and $\sigma_g, \alpha_g, \sigma_h, \alpha_h \ge 0$, for every $x$ and every $w$ independent of $\xi$,
\begin{align*}
  &\E\|\nabla f_\xi(x) - \nabla f(x)\|_2^p \le \sigma_g^p + \alpha_g^p\|\nabla f(x)\|_*^p,\\
  &\E\|(\nabla^2 f_\xi(x) - \nabla^2 f(x))w\|_2^q \\
  &\qquad\le (\sigma_h^q + \alpha_h^q\|\nabla f(x)\|_*^q)\,\|w\|^q.
\end{align*}
\end{assumption}

\paragraph{Discussion.} Assumption~\ref{ass:geometric_smoothness} with $L_1 = 0$ is $L_0$-smoothness in the chosen norm, with no condition on third derivatives; with $L_1 > 0$ it is the $(L_0,L_1)$-smoothness of \citet{zhang2019gradient}, under which the curvature may grow with the gradient norm. The constant $\kappa$ links the chosen norm to the Euclidean norm in which the noise is measured, which is the natural norm for averaging samples: $\kappa = 1$ for $\ell_2$, $\kappa = \sqrt{d}$ for $\ell_\infty$, and $\kappa = \sqrt{\min(m,n)}$ for the spectral norm of an $m \times n$ matrix. For the per-block norm $\max_l \|u_l\|_{(l)}$ with block constants $\kappa_l$, one can take $\kappa = \sqrt{L}\max_l \kappa_l$. Assumption~\ref{ass:noise} allows infinite variance when $p < 2$ or $q < 2$, separate tail indices for gradient and HVP noise as in \citet{fradin2025tight} (\citet{sadiev2025second} treat the case $p = q$), and moments that grow with the gradient norm. The HVP bound is required only for directions $w$ independent of the sample, which holds for $w = d_t$ by the order of operations in Algorithms~\ref{alg:expsom}--\ref{alg:betasom}. The analysis uses only unbiasedness and these moment bounds: $\nabla f_\xi(x)$ and $\nabla^2 f_\xi(x)w$ denote stochastic estimators and need not be derivatives of a common sample function. Differentiating a sampled loss, as in our implementation, is one instance.

\paragraph{Main-text setting.} The theorems below take $L_1 = 0$, $\alpha_g = \alpha_h = 0$, one sample per step, and constant $\eta_t = \eta$; Appendix~\ref{app:theory} treats the general case. We write
\begin{equation*}
  A = \frac{p-1}{p}, \quad K = \frac1q, \quad r = \frac{A}{2A+K} = \frac{q(p-1)}{2q(p-1)+p},
\end{equation*}
and $h = \sigma_h + \kappa L_0$ for the effective scale of the correction noise. When $L_0 > 0$, the constant $h$ is positive even with an exact Hessian oracle ($\sigma_h = 0$), because the random step makes the correction fluctuate around its mean.

\subsection{Unconstrained Problems}

The analysis bounds the momentum error $e_t = m_t - \nabla f(x_t)$ and combines it with a descent inequality.

\begin{lemma}[Descent]
\label{lemma:main_descent}
For RanSOM-E in the main-text setting, for every $\eta_t > 0$, with $C_s = \E[s_t^2]/\eta_t^2 = 2$,
\begin{multline*}
  \E[f(x_{t+1}) \mid \mathcal{F}_t] - f(x_t) \le -\rho\eta_t\|\nabla f(x_t)\|_* \\
  + 2\rho\kappa\eta_t\|e_t\|_2 + \frac{C_s\rho^2 L_0}{2}\eta_t^2 .
\end{multline*}
\end{lemma}

The LMO step gives $\langle \nabla f(x_t), d_t\rangle \le -\rho\|\nabla f(x_t)\|_* + 2\rho\|e_t\|_*$, and $\kappa$ converts $\|e_t\|_*$ into $\|e_t\|_2$. The quadratic term uses $\E[s_t^2] = 2\eta_t^2$; this is the only place where the descent pays for the randomness of the step, through a factor $2$.

\begin{lemma}[Momentum error]
\label{lemma:main_error}
For RanSOM-E in the main-text setting, with $\beta \in (0, \frac12]$,
\begin{multline*}
  \frac1T\sum_{t=0}^{T-1}\E\|e_t\|_2 \lesssim \E\|e_0\|_2 \min\Bigl\{1, \frac{1}{\beta T}\Bigr\} \\
  + \beta^{1-1/p}\sigma_g + \eta\,\beta^{-1/q}\rho\, h .
\end{multline*}
\end{lemma}

The gradient-noise term shrinks as $\beta$ decreases, and the correction-noise term grows, because a smaller $\beta$ keeps more corrections in the buffer. Since each correction is conditionally unbiased given $\mathcal{F}_t$, the corrections enter as martingale differences, and the von Bahr-Esseen inequality \citep{vonbahr1965, pinelis2022multidimensional} bounds their sum with the tail indices $p$ and $q$. The initialization error is carried, not amplified: its factor is at most $1$. If $m_0$ averages $B_{\mathrm{init}}$ independent stochastic gradients, the von Bahr-Esseen inequality gives $\E\|e_0\|_2 \le 2^{1/p}\sigma_g B_{\mathrm{init}}^{-(p-1)/p}$.

\begin{theorem}[RanSOM-E]
\label{thm:expsom_convergence}
Initialize $m_0$ with $B_{\mathrm{init}} \ge 1$ samples and set $G_T = \frac1T\sum_{t=0}^{T-1}\E\|\nabla f(x_t)\|_*$. In the main-text setting, for any $\eta > 0$ and $\beta \in (0,\frac12]$,
\begin{multline*}
  G_T \lesssim \frac{\Delta_0}{\rho\eta T} + \kappa\E\|e_0\|_2\min\Bigl\{1, \frac{1}{\beta T}\Bigr\} \\
  + \kappa\sigma_g\beta^{A} + \kappa\rho h\,\eta\beta^{-K} + \rho L_0 \eta ,
\end{multline*}
with constants depending only on $p$ and $q$. With $\eta = \eta_T$ and $\beta = \beta_T$ chosen as in Lemma~\ref{lem:app:tuning}, which for $\sigma_g > 0$ and large $T$ gives $\eta_T = c_\eta T^{-\frac{A+K}{2A+K}}$ and $\beta_T = c_\beta T^{-\frac{1}{2A+K}}$ with $c_\eta, c_\beta$ depending on $\Delta_0$, $\sigma_g$, $h$, $\rho$, $\kappa$ but not on $T$ (and $\beta_T = \frac12$ if $\sigma_g = 0$), for every $T \ge 1$,
\begin{multline*}
  G_T \lesssim \kappa\,(\Delta_0 h)^{r}\sigma_g^{\frac{K}{2A+K}}\, T^{-r} + \sqrt{\frac{\Delta_0(L_0 + \kappa h)}{T}} \\
  + \kappa\E\|e_0\|_2\min\{1, 1/(\beta_T T)\}.
\end{multline*}
The output $x_R$, with $R$ uniform on $\{0, \ldots, T-1\}$, satisfies $\E\|\nabla f(x_R)\|_* = G_T$.
\end{theorem}

\paragraph{Reading the theorem.} Minimizing $\sigma_g\beta^{A} + \rho h\,\eta\beta^{-K}$ over $\beta$ gives $\beta \propto (\rho h\eta/\sigma_g)^{1/(A+K)}$ and a value of order $\eta^{A/(A+K)}$; balancing this with $\Delta_0/(\rho\eta T)$ gives the exponents above and the leading term $T^{-r}$. The radius $\rho$ cancels in the tuned bound: rescaling the LMO ball is equivalent to rescaling $\eta$. The $T^{-1/2}$ term contains the noise scale $\kappa h$ because $\beta$ is capped at $1/2$ for small $T$, which makes the bound hold for all $T \ge 1$. For $p = q = 2$ the leading term is $\kappa(\Delta_0\sigma_g h)^{1/3}T^{-1/3}$ with $\beta_T \propto T^{-2/3}$, the momentum scaling of STORM. For $\sigma_g > 0$, $\beta_T T \ge \min\{\frac12, c_\beta\}\,T^{\theta}$ with
\begin{equation}
  \label{eq:init-exponent}
  \theta \triangleq 1 - \frac{1}{2A+K} = \frac{pq - 2q + p}{2q(p-1)+p},
\end{equation}
the initialization term decays as $T^{-\theta}$, at the same rate as the leading term when $p = q$; when $\theta \le 0$ (small $p$ and large $q$) it does not decay and must be controlled by $B_{\mathrm{init}}$. The horizon need not be known in advance: restarts with doubling horizons lose at most a factor $4$ (Appendix~\ref{app:anytime}).

\begin{corollary}[Oracle complexity]
\label{cor:oracle_complexity}
Let $\varepsilon > 0$ and $B_{\mathrm{init}} \asymp \max\{1, (\kappa\sigma_g/\varepsilon)^{p/(p-1)}\}$, so that $\kappa\E\|e_0\|_2 \lesssim \varepsilon$. Then $G_T \le \varepsilon$ after a number of oracle queries
\begin{multline}
  \label{eq:sample-complexity}
  N = B_{\mathrm{init}} + T = \mathcal{O}\biggl(1 + \Bigl(\frac{\kappa\sigma_g}{\varepsilon}\Bigr)^{\frac{p}{p-1}} \\
  + \frac{\Delta_0(L_0 + \kappa h)}{\varepsilon^2}\Bigl[1 + \Bigl(\frac{\kappa\sigma_g}{\varepsilon}\Bigr)^{\frac{p}{q(p-1)}}\Bigr]\biggr).
\end{multline}
\end{corollary}

Here the initial batch uses $B_{\mathrm{init}}$ gradients and each step queries one gradient and one HVP at the same point. The term $(\kappa\sigma_g/\varepsilon)^{p/(p-1)}$ is the cost of the initial batch, the rest the cost of the steps. The sum selects the dominant term, so no case distinction between $p \le q$ and $p > q$ is needed.

\paragraph{Optimality.} \citet{fradin2025tight} prove a lower bound for zero-respecting algorithms with unbiased gradient and Hessian estimators whose noise has bounded $p$-th and $q$-th central moments, on functions with Lipschitz gradient and Lipschitz Hessian. It is a minimum of three branches. Since our class fixes no Hessian-Lipschitz constant, for small $\varepsilon$ the minimum is the branch whose terms are those of~\eqref{eq:sample-complexity} in the Euclidean geometry, with the same exponents of $\varepsilon$ (Appendix~\ref{app:lower-bound}). RanSOM-E with $\ell_2$ or $\ell_\infty$ steps is zero-respecting, so it attains the optimal exponents for every $(p,q) \in (1,2]^2$, including $p \ne q$; for $p = q = 2$ this is the $\varepsilon^{-3}$ bound of \citet{arjevani2020second}. The comparison is made in the oracle model of the lower bound, whose estimators need not be derivatives of a common sample loss, and it concerns exponents for fixed positive noise scales: $h$ includes the fluctuation of the random step and is positive when $\sigma_h = 0$. The near-optimal upper bound of \citet{fradin2025tight} evaluates the HVP at an auxiliary point and has an extra term that depends on the Hessian Lipschitz constant.

\begin{figure*}[t]
  \centering
  \includegraphics[width=0.75\textwidth]{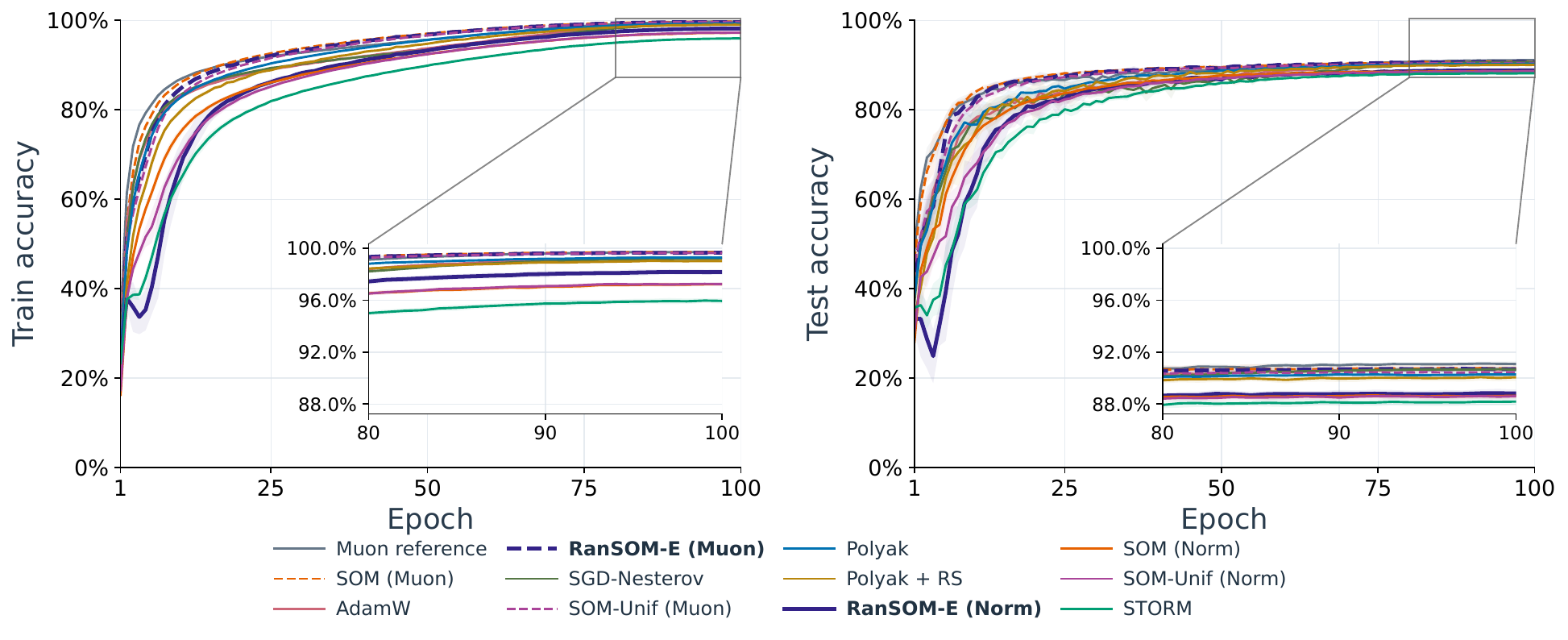}
  \caption{CIFAR-10 train accuracy (left) and test accuracy (right) for ResNet-20. Curves show the mean over ten seeds; shading shows one sample standard deviation. Colors identify algorithm families; Muon-geometry curves are dashed, normalized and reference curves are solid. Insets show epochs 80--100. RanSOM-E labels are bold. Settings are selected on validation data; Table~\ref{tab:cifar10-accuracy} reports epoch-100 results.}
  \label{fig:cifar10-accuracy}
\end{figure*}

\subsection{Proof Outline}

Let $\Psi_{t+1} = \delta_{t+1} - (\nabla f(x_{t+1}) - \nabla f(x_t))$ be the correction error and $\bar\xi_{t+1} = g_{t+1} - \nabla f(x_{t+1})$ the gradient noise, so that
\begin{equation*}
  e_{t+1} = (1-\beta)\, e_t + \beta\,\bar\xi_{t+1} + (1-\beta)\,\Psi_{t+1}.
\end{equation*}
By Lemma~\ref{lemma:stein} and the unbiasedness of the HVP oracle, $\E[\Psi_{t+1} \mid \mathcal{F}_t] = 0$, so after unrolling, $e_t$ is $(1-\beta)^t e_0$ plus two geometrically weighted sums of martingale differences. The gradient noise has $p$-th moment at most $\sigma_g^p$. The correction error splits into HVP oracle noise, of size $\eta\rho\sigma_h$, and the path residual $\eta\nabla^2 f(x_{t+1})d_t - \int_0^{s_t}\nabla^2 f(x_t + z d_t)d_t\,dz$, of size at most $\kappa\rho L_0(\eta + s_t)$, whose $q$-th moment is $\mathcal{O}(\eta^q)$ for the exponential law; hence $\E[\|\Psi_{t+1}\|_2^q \mid \mathcal{F}_t] \lesssim (\eta\rho h)^q$. The von Bahr-Esseen inequality bounds each weighted sum by its weighted moments, and $\sum_{k\ge0}(1-\beta)^{ak} \le 2/(a\beta)$ for $a \in \{p, q\}$ gives Lemma~\ref{lemma:main_error}. Telescoping Lemma~\ref{lemma:main_descent} and inserting Lemma~\ref{lemma:main_error} gives the first bound of Theorem~\ref{thm:expsom_convergence}; choosing $\eta$ and $\beta$ gives the second. The only curvature quantity used is the bound $\|\nabla^2 f\|_{\mathrm{op}} \le L_0$ along the segment $[x_t, x_{t+1}]$; the Hessian is never compared at two points, which is where a Lipschitz Hessian would enter.

\subsection{Constrained Problems}

Let $D = \sup_{x,y \in \mathcal{C}}\|x - y\|$ and let $\mathcal{G}(x) = \max_{v \in \mathcal{C}}\langle\nabla f(x), x - v\rangle$ be the Frank-Wolfe gap, which is nonnegative and vanishes exactly at first-order stationary points of $f$ on $\mathcal{C}$. For the Beta step with $\eta_t \le \frac12$, the descent inequality reads
\begin{multline*}
  \E[f(x_{t+1}) \mid \mathcal{F}_t] - f(x_t) \le -\eta_t\,\mathcal{G}(x_t) \\
  + 2\kappa D\eta_t\|e_t\|_2 + L_0 D^2\eta_t^2 .
\end{multline*}
The momentum error is bounded as in Lemma~\ref{lemma:main_error}, with the step length $\|d_t\| \le D$ in place of $\rho$ and the random Beta weight, whose moments are bounded by absolute constants when $\eta_t \le \frac12$.

\begin{theorem}[RanSOM-B]
\label{thm:betasom_convergence}
In the setting of Theorem~\ref{thm:expsom_convergence}, with $\eta_T, \beta_T \in (0,\frac12]$ chosen as in Lemma~\ref{lem:app:tuning} with the cap $\eta \le \frac12$, RanSOM-B satisfies, for every $T \ge 1$,
\begin{multline*}
  \frac1T\sum_{t=0}^{T-1}\E[\mathcal{G}(x_t)] \lesssim \kappa\Delta_0^{r}(D\sigma_g)^{\frac{K}{2A+K}}(D^2 h)^{r}\,T^{-r} \\
  + D\sqrt{\frac{\Delta_0(L_0 + \kappa h)}{T}} + \frac{\Delta_0}{T} \\
  + \kappa D\,\E\|e_0\|_2\min\{1, 1/(\beta_T T)\}.
\end{multline*}
\end{theorem}

The diameter enters the leading term twice: through the descent inequality, which multiplies every momentum-error term by $D$, and through the length $\|d_t\| \le D$ of the HVP direction. The cap $\eta \le \frac12$ keeps $K_t \ge 1$; at $\eta = 1$ the law $\Beta(1,0)$ is not defined. For bounded variance the exponent is $1/3$, as for the one-sample stochastic Frank-Wolfe method \citep{zhang2020one}. We do not know a matching lower bound for the Frank-Wolfe gap, so Theorem~\ref{thm:betasom_convergence} is stated as an upper bound.

\subsection{General Setting}

Appendix~\ref{app:theory} proves the general versions for $L_1 > 0$, $\alpha_g, \alpha_h > 0$, and minibatches of size $B$. Three features differ from the main-text setting. First, with $L_1 > 0$ the gradient norm can grow exponentially along the ray $x_t + z d_t$ (by a Gr\"onwall argument), so the analysis requires $q L_1 \rho \eta_t \le \frac12$, which keeps the exponential moments of $s_t$ finite. Second, the correction error splits into oracle noise, which decreases with $B$, and the path residual caused by the random step, which is common to the batch and does not decrease with $B$. Third, when $L_1 > 0$ or $\alpha_h > 0$, a coupling condition $\eta/\beta \lesssim 1$ appears; it is compatible with the tuning above for $q \le p$, while for $q > p$ the rate we establish is $T^{-(p-1)/(2p-1)}$. In the constrained case the gradient is bounded on $\mathcal{C}$, so the affine terms are absorbed into the constants and no coupling condition arises.

\section{EXPERIMENTS}
\label{sec:experiments}

We compare RanSOM-E with the other second-order momentum methods in accuracy and cost on image classification (Section~\ref{sec:cifar10}), and test RanSOM-B on a constrained problem (Section~\ref{sec:movielens}). ResNet-20 uses ReLU activations, which are not twice differentiable, so the first experiment lies outside the setting of the theorems.

\subsection{Image Classification on CIFAR-10}
\label{sec:cifar10}

\paragraph{Setup.} We train ResNet-20 \citep{he2016deep} on CIFAR-10 \citep{krizhevsky2009learning} with the cross-entropy loss, batch size 256, and single-precision arithmetic. We compare twelve methods: Polyak momentum, Polyak momentum with randomized steps (Polyak + RS), and STORM, in the normalized geometry; SGD with Nesterov momentum \citep{nesterov1983}, AdamW \citep{loshchilov2019decoupled}, and the reference implementation of Muon \citep{jordan2024muon}; and SOM, SOM-Unif, and RanSOM-E, each in the normalized and in the Muon geometry; the Muon geometry approximates the spectral LMO by five Newton-Schulz steps, whereas the theorems assume an exact LMO. Each method runs for 100 epochs with seeds 0 to 9 and its own selected hyperparameters. Within each seed, methods use the same initialization, image order, and augmented batches. The methods of the RanSOM family use no weight decay; the reference methods use the weight decay selected in their grids.

\paragraph{Parameter selection.} We hold out a fixed split of 5{,}000 of the 50{,}000 training images for validation. Each setting is first screened for 25 epochs; the two best settings per method are then continued to 100 epochs from their saved optimizer and random-number-generator states, and we select by epoch-100 validation accuracy. The learning-rate schedule (5\% linear warmup, then cosine decay) keeps its 100-epoch horizon during screening. For the momentum families, we search the step parameter $\gamma \in \{0.003, 0.01, 0.03, 0.1, 0.3\}$, the product $\rho\eta$ of the LMO radius and the mean step, and $\beta \in \{0.03, 0.1\}$. The final runs train on all 50{,}000 training images and evaluate on the 10{,}000 test images. We report epoch-100 results without selecting a checkpoint by test accuracy. Appendix~\ref{sec:cifar10-details} gives the reference grids and the selected settings. 

\begin{table}[t]
\centering
\small
\caption{CIFAR-10 accuracy (\%) with ResNet-20 at epoch 100: mean $\pm$ sample standard deviation over ten seeds. Bold values mark the highest mean in each column. Both columns evaluate the final model on fixed normalized images without crop or flip. Training curves use augmented batches during updates.}
\label{tab:cifar10-accuracy}
\begin{tabular}{@{}lcc@{}}
\toprule
Method & Train $\uparrow$ & Test $\uparrow$ \\
\midrule
Muon reference & $\mathbf{99.87 \pm 0.03}$ & $\mathbf{91.09 \pm 0.21}$ \\
SOM (Muon) & $99.82 \pm 0.02$ & $90.74 \pm 0.33$ \\
AdamW & $99.58 \pm 0.04$ & $90.71 \pm 0.17$ \\
\textbf{RanSOM-E (Muon)} & $99.79 \pm 0.01$ & $90.69 \pm 0.17$ \\
SGD-Nesterov & $99.54 \pm 0.03$ & $90.66 \pm 0.22$ \\
SOM-Unif (Muon) & $99.76 \pm 0.03$ & $90.44 \pm 0.26$ \\
Polyak & $99.56 \pm 0.03$ & $90.29 \pm 0.24$ \\
Polyak + RS & $99.32 \pm 0.06$ & $90.07 \pm 0.27$ \\
\textbf{RanSOM-E (Norm)} & $98.52 \pm 0.18$ & $88.82 \pm 0.30$ \\
SOM (Norm) & $97.69 \pm 0.16$ & $88.64 \pm 0.20$ \\
SOM-Unif (Norm) & $97.80 \pm 0.13$ & $88.60 \pm 0.23$ \\
STORM & $96.47 \pm 0.19$ & $88.19 \pm 0.28$ \\
\bottomrule
\end{tabular}

\end{table}

\paragraph{Accuracy.} Table~\ref{tab:cifar10-accuracy} reports final test accuracy; reference Muon has the highest mean, $91.09\%$. In each geometry, RanSOM-E, SOM, and SOM-Unif are within $0.3$ percentage points of each other. Over ten paired seeds, the mean difference between RanSOM-E and SOM is $-0.06$ points with Muon updates and $+0.18$ with normalized updates, and against SOM-Unif it is $+0.24$ and $+0.22$; all four 95\% intervals contain zero (Table~\ref{tab:cifar10-paired} in the appendix). The data therefore do not separate RanSOM-E from SOM and SOM-Unif.

\paragraph{Measured update cost.} CUDA events measure the time per update, excluding data loading and evaluation (Figure~\ref{fig:cifar10-update-time} in the appendix). RanSOM-E has a lower mean update time than SOM-Unif in both geometries, $125.8$ against $150.5$ ms with normalized updates and $171.7$ against $184.3$ ms with Muon updates, where the difference is within one standard deviation (about $30$ ms). This is a descriptive comparison, consistent with the extra passes of SOM-Unif. RanSOM-E and SOM perform the same computation per step, and their times differ by less than one standard deviation. All HVP-based methods take more than three times as long per update as Polyak momentum (Section~\ref{sec:cost}).

\subsection{Constrained Matrix Completion}
\label{sec:movielens}

We compare RanSOM-B with stochastic Frank-Wolfe with momentum (Polyak), with the SOM correction, and with the SOM-Unif correction on the full $943 \times 1682$ MovieLens 100K rating matrix \citep{movielens}. The objective is the mean log-cosh loss over observed training ratings, subject to the nuclear-norm constraint $\|X\|_* \le 1600$. The loss is smooth and convex, and its curvature $\operatorname{sech}^2$ depends on the residual, so the Hessian varies along the trajectory; Appendix~\ref{sec:constrained-details} gives the full problem~\eqref{eq:movielens-logcosh-problem}. A validation search selected the mean step $0.0003$ and $\beta = 0.03$ for all four methods, which we then ran for 100 epochs on ten new paired seeds.

Polyak has the lowest mean test loss, $0.37148$, against $0.37208$ for SOM, $0.37212$ for SOM-Unif, and $0.37229$ for RanSOM-B (Table~\ref{tab:constrained-logcosh-test} in the appendix). RanSOM-B is higher than SOM and SOM-Unif by $2.0 \times 10^{-4}$ and $1.7 \times 10^{-4}$, and both paired 95\% intervals contain zero. On this problem, run on a CPU, the linear minimization oracle dominates the cost of a step: RanSOM-B takes $70.3$ ms per update, close to the $69.8$ ms of Polyak, while SOM takes $73.4$ ms and SOM-Unif $81.4$ ms.

\section{CONCLUSION}
\label{sec:discussion}

Drawing the step length at random lets a single HVP at the new iterate estimate the change of the gradient without bias. The resulting methods need neither a Lipschitz Hessian nor an auxiliary point. For objectives with Lipschitz gradient and heavy-tailed gradient and Hessian noise, the oracle complexity of RanSOM-E has the optimal exponents for all tail indices.

\textbf{Limitations.} For ReLU networks, automatic differentiation returns an almost-everywhere second derivative, which need not satisfy the identity~\eqref{eq:bias_def} or our assumptions, so our CIFAR-10 experiment is outside the setting of the theorems. In our ResNet-20 implementation, an HVP-based update took more than three times as long as a first-order update. We detect no accuracy difference between RanSOM and SOM or SOM-Unif, and on MovieLens first-order Frank-Wolfe has the lowest test loss; the advantages of the random step are an analysis without a Lipschitz Hessian and the absence of an auxiliary point. For $L_1 > 0$ or affine Hessian noise, the rate we establish for $q > p$ is slower.

\subsection*{AI Use Statement}

In this work, we used generative AI tools to edit the text, to assist in checking and revising proofs, and to draft experiment code. We have reviewed all AI-assisted work, and we take responsibility for the final content of this work, including text, claims, or artifacts produced with the aid of generative AI.

\bibliography{ransom}

\clearpage
\section*{Checklist}

\begin{enumerate}

  \item For all models and algorithms presented, check if you include:
  \begin{enumerate}
    \item A clear description of the mathematical setting, assumptions, algorithm, and/or model. [Yes] Sections~\ref{sec:method} and~\ref{sec:theory}.
    \item An analysis of the properties and complexity (time, space, sample size) of any algorithm. [Yes] Theorems~\ref{thm:expsom_convergence} and~\ref{thm:betasom_convergence}, Corollary~\ref{cor:oracle_complexity}, and the per-step cost in Section~\ref{sec:method}.
    \item (Optional) Anonymized source code, with specification of all dependencies, including external libraries. [No] The code will be released with the final version.
  \end{enumerate}

  \item For any theoretical claim, check if you include:
  \begin{enumerate}
    \item Statements of the full set of assumptions of all theoretical results. [Yes] Section~\ref{sec:theory} and Appendix~\ref{app:theory}.
    \item Complete proofs of all theoretical results. [Yes] Appendix~\ref{app:theory}.
    \item Clear explanations of any assumptions. [Yes] Section~\ref{sec:theory}.
  \end{enumerate}

  \item For all figures and tables that present empirical results, check if you include:
  \begin{enumerate}
    \item The code, data, and instructions needed to reproduce the main experimental results (either in the supplemental material or as a URL). [No] The datasets are public and the settings are in Appendix~\ref{sec:AppExp}; the code will be released with the final version.
    \item All the training details (e.g., data splits, hyperparameters, how they were chosen). [Yes] Section~\ref{sec:cifar10} and Appendix~\ref{sec:AppExp}.
    \item A clear definition of the specific measure or statistics and error bars (e.g., with respect to the random seed after running experiments multiple times). [Yes] Ten seeds with sample standard deviations and paired Student-$t$ intervals (Section~\ref{sec:cifar10}, Table~\ref{tab:cifar10-accuracy}, Figures~\ref{fig:cifar10-accuracy} and~\ref{fig:cifar10-update-time}); ten fresh paired seeds for log-cosh MovieLens (Appendix~\ref{sec:constrained-details}).
    \item A description of the computing infrastructure used. (e.g., type of GPUs, internal cluster, or cloud provider). [Yes] The CIFAR-10 runs used one NVIDIA Tesla V100-SXM2 GPU with 32 GB of memory, and the MovieLens runs used a CPU (Appendices~\ref{sec:cifar10-details} and~\ref{sec:constrained-details}).
  \end{enumerate}

  \item If you are using existing assets (e.g., code, data, models) or curating/releasing new assets, check if you include:
  \begin{enumerate}
    \item Citations of the creator If your work uses existing assets. [Yes] CIFAR-10 and MovieLens are cited.
    \item The license information of the assets, if applicable. [Yes] CIFAR-10 has no formal license and asks for citation of its technical report; MovieLens 100K is licensed for research use under the GroupLens terms (Appendices~\ref{sec:cifar10-details} and~\ref{sec:constrained-details}).
    \item New assets either in the supplemental material or as a URL, if applicable. [Not Applicable]
    \item Information about consent from data providers/curators. [Not Applicable]
    \item Discussion of sensible content if applicable, e.g., personally identifiable information or offensive content. [Not Applicable]
  \end{enumerate}

  \item If you used crowdsourcing or conducted research with human subjects, check if you include:
  \begin{enumerate}
    \item The full text of instructions given to participants and screenshots. [Not Applicable]
    \item Descriptions of potential participant risks, with links to Institutional Review Board (IRB) approvals if applicable. [Not Applicable]
    \item The estimated hourly wage paid to participants and the total amount spent on participant compensation. [Not Applicable]
  \end{enumerate}

\end{enumerate}

\clearpage
\appendix
\thispagestyle{empty}
\onecolumn
\aistatstitle{RanSOM: Second-Order Momentum with Randomized Scaling\\for Constrained and Unconstrained Optimization:\\Supplementary Materials}


\section{Additional Experiments}\label{sec:AppExp}
\subsection{CIFAR-10 Protocol and Selected Settings}
\label{sec:cifar10-details}
\raggedbottom

\paragraph{Data and schedule.}
Images use random crops with four-pixel padding and random horizontal flips.
Channel means are $(0.4914,0.4822,0.4465)$ and standard deviations are
$(0.247,0.2435,0.2616)$. The validation split uses seed 20260925.
Linear warmup occupies the first 5\% of the 100-epoch schedule, followed
by cosine decay. Training batches contain 256 images. The family methods initialize the momentum with the average gradient over ten batches of 128 training images (1{,}280 images). All CIFAR-10 runs used one NVIDIA Tesla V100-SXM2 GPU with 32 GB of memory. The CIFAR-10 distribution page states no formal license and asks users to cite the accompanying technical report \citep{krizhevsky2009learning}, which we do. Normalized variants normalize each parameter tensor;
Muon variants use five Newton-Schulz steps on matrix-shaped updates.

\paragraph{Reference grids.}
SGD-Nesterov uses learning rates $\{0.01,0.03,0.1,0.3,1\}$.
AdamW uses learning rates $\{0.0003,0.001,0.003\}$ and weight decay
$\{0.0005,0.05\}$.
Reference Muon uses learning rates $\{0.005,0.01,0.02,0.05,0.1\}$.
The final reference weight decays are 0.0005 for SGD-Nesterov,
0.05 for AdamW, and 0.01 for Muon. Family methods use zero weight decay.
The coefficient $\beta$ weights the new gradient, so the retained momentum
coefficient is $1-\beta$. Table~\ref{tab:cifar10-settings} gives the
settings held fixed across the ten final seeds.

\begin{table}[H]
\centering
\small
\caption{Selected CIFAR-10 settings, held fixed across the ten final seeds. Corrections are applied without additional scaling.}
\label{tab:cifar10-settings}
\begin{tabular}{@{}lc@{}}
\toprule
Method & Selected settings \\
\midrule
Muon reference & $\mathrm{lr}=0.02$, $\mathrm{wd}=0.01$ \\
SOM (Muon) & $\gamma=0.01$, $\beta=0.1$ \\
AdamW & $\mathrm{lr}=0.003$, $\mathrm{wd}=0.05$ \\
\textbf{RanSOM-E (Muon)} & $\gamma=0.01$, $\beta=0.1$ \\
SGD-Nesterov & $\mathrm{lr}=0.03$, $\mathrm{wd}=0.0005$ \\
SOM-Unif (Muon) & $\gamma=0.01$, $\beta=0.1$ \\
Polyak & $\gamma=0.03$, $\beta=0.1$ \\
Polyak + RS & $\gamma=0.03$, $\beta=0.03$ \\
\textbf{RanSOM-E (Norm)} & $\gamma=0.03$, $\beta=0.1$ \\
SOM (Norm) & $\gamma=0.01$, $\beta=0.03$ \\
SOM-Unif (Norm) & $\gamma=0.01$, $\beta=0.1$ \\
STORM & $\gamma=0.03$, $\beta=0.1$ \\
\bottomrule
\end{tabular}

\end{table}

\paragraph{Final losses.}
Table~\ref{tab:cifar10-loss} reports epoch-100 cross-entropy with the same
seeds and selection as Table~\ref{tab:cifar10-accuracy}.
Both table columns evaluate the final model on fixed normalized images without crop or flip. Training curves use augmented batches during updates.

\begin{table}[H]
\centering
\small
\caption{CIFAR-10 cross-entropy at epoch 100, mean $\pm$ sample standard deviation over ten seeds. Bold values mark the lowest mean in each column.}
\label{tab:cifar10-loss}
\begin{tabular}{@{}lcc@{}}
\toprule
Method & Train cross-entropy & Test cross-entropy \\
\midrule
Muon reference & $\mathbf{0.0045 \pm 0.0005}$ & $0.4887 \pm 0.0102$ \\
SOM (Muon) & $0.0061 \pm 0.0004$ & $0.5616 \pm 0.0181$ \\
AdamW & $0.0147 \pm 0.0007$ & $0.4111 \pm 0.0074$ \\
\textbf{RanSOM-E (Muon)} & $0.0069 \pm 0.0003$ & $0.5490 \pm 0.0151$ \\
SGD-Nesterov & $0.0192 \pm 0.0006$ & $\mathbf{0.3615 \pm 0.0092}$ \\
SOM-Unif (Muon) & $0.0078 \pm 0.0006$ & $0.5615 \pm 0.0169$ \\
Polyak & $0.0132 \pm 0.0008$ & $0.5407 \pm 0.0098$ \\
Polyak + RS & $0.0197 \pm 0.0013$ & $0.5524 \pm 0.0217$ \\
\textbf{RanSOM-E (Norm)} & $0.0420 \pm 0.0053$ & $0.5265 \pm 0.0110$ \\
SOM (Norm) & $0.0654 \pm 0.0042$ & $0.4903 \pm 0.0126$ \\
SOM-Unif (Norm) & $0.0630 \pm 0.0036$ & $0.4758 \pm 0.0085$ \\
STORM & $0.0990 \pm 0.0050$ & $0.4503 \pm 0.0093$ \\
\bottomrule
\end{tabular}

\end{table}

\begin{figure}[H]
    \centering
    \includegraphics[width=\linewidth]{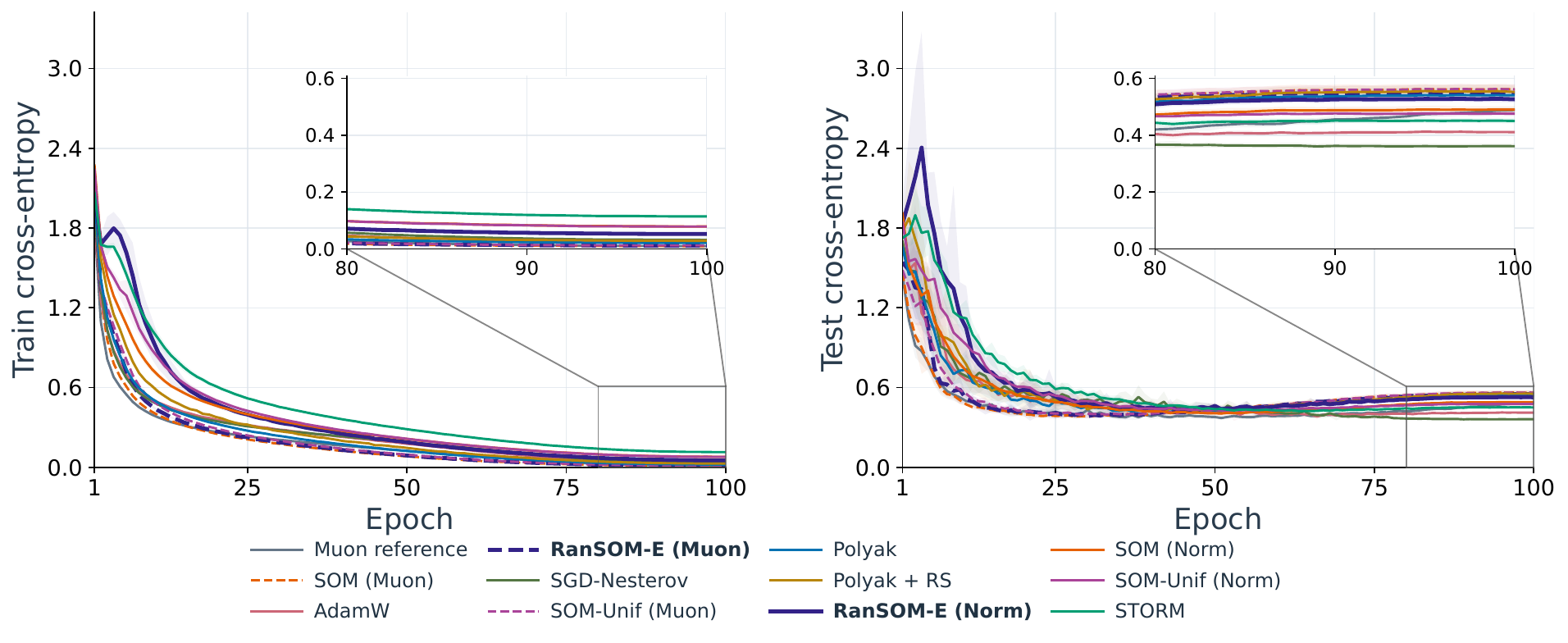}
    \caption{CIFAR-10 train cross-entropy (left) and test cross-entropy
    (right). Curves show ten-seed means, with one sample standard
    deviation as shading. Both panels use the same linear loss scale.
    Insets show epochs 80--100. Colors and line styles match Figure~\ref{fig:cifar10-accuracy}.
    Training metrics use augmented batches during updates; test metrics
    use evaluation after each epoch.}
    \label{fig:cifar10-loss}
\end{figure}

\begin{figure}[H]
  \centering
  \includegraphics[width=0.6\linewidth]{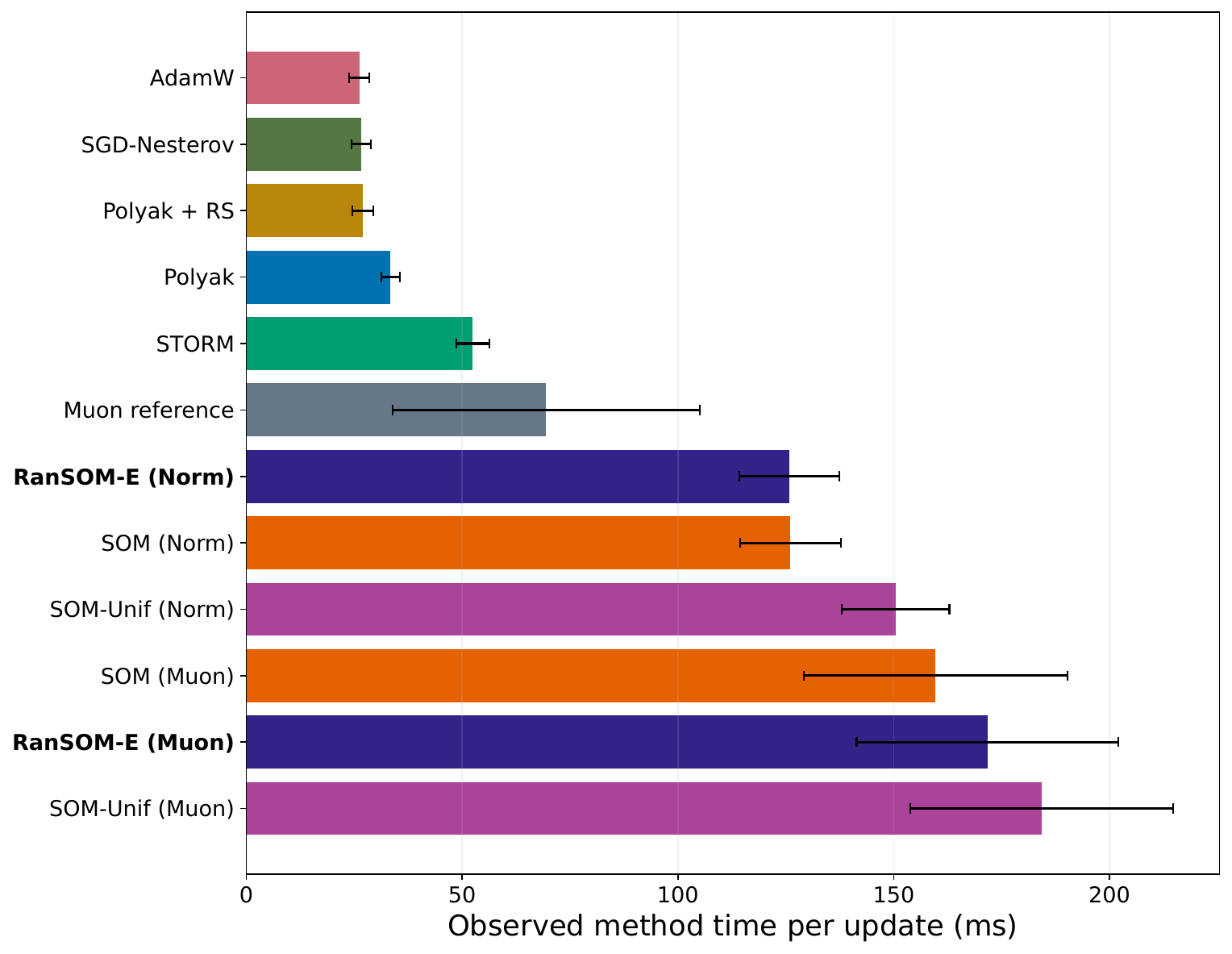}
  \caption{Average CUDA-event time per update on CIFAR-10, in milliseconds. Each seed averages all updates over 100 epochs; bars and error bars show the mean and one sample standard deviation across ten seeds. Data loading and evaluation are excluded.}
  \label{fig:cifar10-update-time}
\end{figure}

\paragraph{Paired accuracy comparisons.}
Table~\ref{tab:cifar10-paired} reports RanSOM-E minus each reference within the same geometry.
For each comparison, we subtract epoch-100 test accuracy within each seed, then average the ten differences.
The 95\% Student-$t$ intervals use the sample standard deviation of the paired differences and nine degrees of freedom.
\begin{table}[H]
\centering\small
\setlength{\tabcolsep}{4pt}
\caption{Paired final test accuracy differences for RanSOM-E versus SOM and SOM-Unif over ten seeds. Values are in percentage points. Positive differences favor RanSOM-E. Intervals are 95\% Student-$t$ intervals with nine degrees of freedom.}
\label{tab:cifar10-paired}
\begin{tabular}{@{}lccc@{}}
\toprule
Geometry & Reference & Mean $\Delta$ (pp) & 95\% interval (pp) \\
\midrule
Normalized & SOM & $+0.181$ & $[-0.058, +0.420]$ \\
Normalized & SOM-Unif & $+0.216$ & $[-0.091, +0.523]$ \\
Muon & SOM & $-0.057$ & $[-0.275, +0.161]$ \\
Muon & SOM-Unif & $+0.243$ & $[-0.001, +0.487]$ \\
\bottomrule
\end{tabular}

\end{table}

\clearpage
\subsection{Constrained Matrix Completion: MovieLens 100K}
\label{sec:constrained-details}
\label{appendix:experiment_setups}

\paragraph{Problem definition.} Let $\Omega_{\mathrm{tr}}$ be the set of observed training user-movie pairs, and let $y_{ij}$ be the rating for pair $(i,j)$. Define the training mean $\bar y=|\Omega_{\mathrm{tr}}|^{-1}\sum_{(i,j)\in\Omega_{\mathrm{tr}}}y_{ij}$. We optimize the matrix $X\in\mathbb{R}^{943\times1682}$ directly and predict the rating $\widehat y_{ij}=\bar y+X_{ij}$. The constrained problem is
\begin{equation}
\label{eq:movielens-logcosh-problem}
\min_{X\in\mathbb{R}^{943\times1682}}
\; F(X):=\frac{1}{|\Omega_{\mathrm{tr}}|}
\sum_{(i,j)\in\Omega_{\mathrm{tr}}}
\log\cosh(\bar y+X_{ij}-y_{ij})
\quad\text{subject to}\quad \|X\|_*\leq R,\qquad R=1600.
\end{equation}

Here $\|X\|_*=\sum_k\sigma_k(X)$ is the nuclear norm, the sum of the singular values. The constraint couples the user and movie predictions and limits the matrix complexity. The objective and feasible set are convex. Unlike squared loss, log-cosh has residual-dependent curvature $\operatorname{sech}^2(r)$. The Frank-Wolfe linear oracle minimizes $\langle G,S\rangle$ over this ball: for a leading singular-vector pair $(u_1,v_1)$ of the current gradient or momentum matrix $G$, it returns $S=-R u_1v_1^\top$. Updates combine this vertex with the current feasible matrix. We compute $(u_1, v_1)$ with an ARPACK truncated singular value decomposition: one singular triplet of largest magnitude, tolerance $0$ (machine precision), and a fixed random state.

\paragraph{Protocol.} We use all 100,000 MovieLens 100K ratings on a $943\times1682$ matrix. The fixed split contains 80,000 training, 10,000 validation, and 10,000 test ratings, with split seed 20260927. We initialize $X$ at zero and use the problem in~\eqref{eq:movielens-logcosh-problem}. We use FP32 arithmetic and batches of 512 ratings sampled without replacement, so each epoch makes one pass over the 80,000 training ratings in 157 updates (the last batch has 128 ratings). As for CIFAR-10, the momentum is initialized with the average gradient over ten batches of 128 training ratings (1{,}280 ratings). All MovieLens runs used one core of an Intel Xeon Gold 6240 CPU at 2.6 GHz, with one PyTorch and BLAS thread per job, so the times in Table~\ref{tab:movielens_timing} are single-thread CPU times. MovieLens 100K is distributed by GroupLens for research use, on the conditions that publications cite \citet{movielens}, that the data are not redistributed, that no endorsement by the University of Minnesota is implied, and that commercial use requires separate permission; our use meets these conditions. The four methods use separate states and the same batches within each seed.

\paragraph{Parameter selection.} The validation search ran each setting with seeds 0 and 1 and selected by the mean final validation log-cosh over the two seeds. The final grid contains the mean steps $\{0.0001, 0.0003, 0.001, 0.003, 0.01, 0.03, 0.1\}$ and $\beta \in \{0.01, 0.03, 0.1, 0.3\}$; we extended the grid until the selected value was not at its boundary. The search selected the mean step $0.0003$ and $\beta = 0.03$ for all four methods, and we held these values fixed for the ten final paired seeds.

\begin{figure}[!htbp]
\centering
\includegraphics[width=\linewidth]{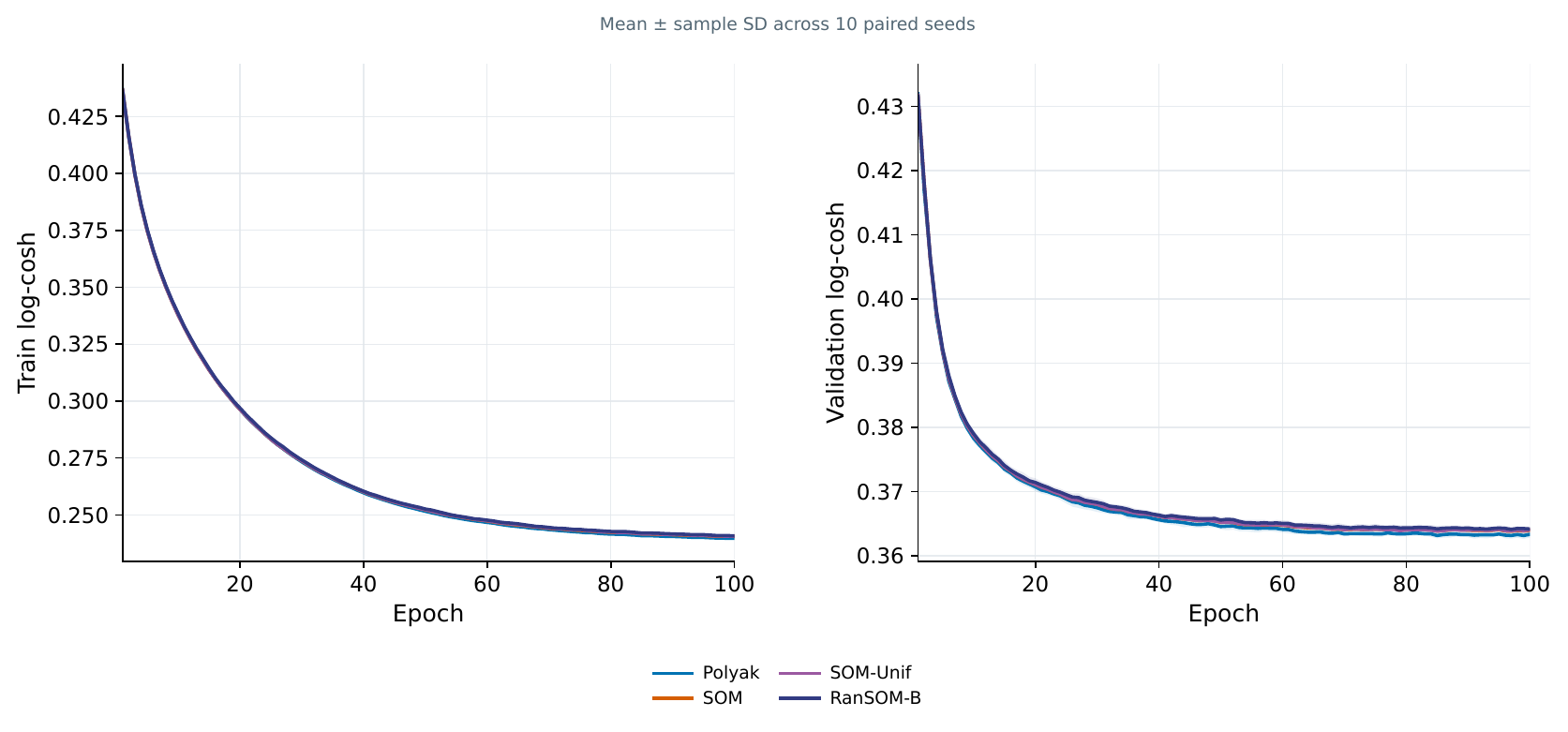}
\caption{Full MovieLens log-cosh replication. Solid curves show means and bands show sample SD over ten fresh paired seeds. The panels show training and validation losses.}
\label{fig:movielens_results}
\end{figure}

\begin{table}[!htbp]
\centering\footnotesize
\caption{Constrained MovieLens with log-cosh loss: final test loss, mean $\pm$ sample standard deviation over ten paired seeds. Lower is better; bold marks the lowest mean.}
\label{tab:constrained-logcosh-test}
\begin{tabular}{@{}lrr@{}}
\toprule
Method & Test log-cosh \\
\midrule
Polyak & $\mathbf{0.371480 \pm 0.000640}$ \\
SOM & $0.372084 \pm 0.000367$ \\
SOM-Unif & $0.372117 \pm 0.000363$  \\
\textit{RanSOM-B} & $0.372287 \pm 0.000369$  \\
\bottomrule
\end{tabular}

\end{table}

\begin{table}[!htbp]
\centering\small
\caption{Epoch-100 training and validation log-cosh, mean $\pm$ sample SD over ten fresh paired seeds.}
\begin{tabular}{@{}lrr@{}}
\toprule
Method & Train log-cosh & Validation log-cosh \\
\midrule
Polyak & $0.239764 \pm 0.000313$ & $0.363331 \pm 0.000507$ \\
SOM & $0.240586 \pm 0.000197$ & $0.363965 \pm 0.000477$ \\
SOM-Unif & $0.240613 \pm 0.000240$ & $0.364014 \pm 0.000475$ \\
\textit{RanSOM-B} & $0.240795 \pm 0.000328$ & $0.364172 \pm 0.000251$ \\
\bottomrule
\end{tabular}

\end{table}

\begin{table}[!htbp]
\centering\small
\caption{Paired final test differences, RanSOM-B minus comparator. All numbers use units of $10^{-4}$. Positive values favor the comparator. Intervals use Student-$t$ with nine degrees of freedom and ten pairs. They are exploratory and have no multiple-comparison adjustment.}
\label{tab:constrained-paired}
\begin{tabular}{@{}lrrr@{}}
\toprule
Comparator & Metric & Difference & 95\% interval \\
\midrule
SOM & Log-cosh & 2.027 & $[-0.453, 4.508]$ \\
SOM-Unif & Log-cosh & 1.691 & $[-0.236, 3.618]$ \\
\bottomrule
\end{tabular}

\end{table}

The training and validation losses decrease throughout the 100 epochs for all four methods. Polyak has the lowest final mean test loss. RanSOM-B has a slightly higher mean test loss than SOM and SOM-Unif, and both paired intervals contain zero (Table~\ref{tab:constrained-paired}).

\begin{table}[!htbp]
\centering\small
\caption{MovieLens log-cosh update and oracle times in milliseconds. Update and oracle columns give means $\pm$ sample SD across ten seeds. Step SD uses the pooled individual update duration moments. The oracle cost is part of the update.}
\label{tab:movielens_timing}
\begin{tabular}{@{}lrrr@{}}
\toprule
Method & Update (ms) & Oracle (ms) & Step SD (ms) \\
\midrule
Polyak & $69.77 \pm 25.53$ & $50.70 \pm 19.05$ & 28.73 \\
SOM & $73.44 \pm 26.94$ & $52.55 \pm 19.83$ & 30.76 \\
SOM-Unif & $81.39 \pm 29.64$ & $52.54 \pm 19.84$ & 33.51 \\
\textit{RanSOM-B} & $70.29 \pm 26.01$ & $52.45 \pm 19.79$ & 29.25 \\
\bottomrule
\end{tabular}

\end{table}

All final iterates satisfied the nuclear-norm constraint. The maximum recorded ratio $\|X\|_*/1600$ was below $0.856$.

\section{Theoretical Analysis and Proofs}
\label{app:theory}

This appendix presents the full convergence analysis of RanSOM in three stages. First, we restate the assumptions and derive the Stein-type identities that make the RanSOM correction unbiased. Second, we prove the key technical lemma bounding the moment of the correction error. Third, we use this bound to establish the convergence rates stated in the main text, for both the unconstrained (RanSOM-E) and constrained (RanSOM-B) settings. Throughout, we work under the pointwise $(L_0, L_1)$-smoothness condition of Assumption~\ref{ass:geometric_smoothness} used in the main text.

\subsection{Notation and Setting}
\label{app:notation}

Recall the notation from Section~\ref{sec:theory}. The space $\mathcal{E}
= \mathbb{R}^d$ is equipped with a primal norm $\|\cdot\|$, dual norm
$\|\cdot\|_*$, and Euclidean norm $\|\cdot\|_2$. Since $\mathcal{E}$ is
finite-dimensional, all norms are equivalent; Assumption~\ref{ass:app:kappa}
below fixes the relevant compatibility constant $\kappa \geq 1$
satisfying $\|u\|_* \leq \kappa \|u\|_2$ and $\|u\|_2 \leq \kappa \|u\|_*$
for all $u \in \mathcal{E}^* \cup \mathcal{E}$. (A single symmetric
constant can always be chosen by taking the maximum of the two
one-sided constants.)

The filtration $\{\mathcal{F}_t\}_{t \geq 0}$ records the history of all
random variables up to and including step $t$: the iterate $x_t$, the
momentum $m_t$, and all previously drawn step sizes, Stein weights, and
batches. In particular $x_t$ and $m_t$ are $\mathcal{F}_t$-measurable.
At step $t$, the step size $s_t$, Stein weight $w_t = w(s_t)$, batch
$\mathcal{B}_{t+1}$, and stochastic oracles $g_{t+1}, h_{t+1}$ are drawn
conditionally on $\mathcal{F}_t$; the step $s_t$ (and hence $w_t$) is drawn independently of $\mathcal{F}_t$, and the batch $\mathcal{B}_{t+1}$ consists of fresh samples, i.i.d.\ within the batch and independent of $\mathcal{F}_t$ and $s_t$.

Because the step $s_t$ is drawn after $\mathcal{F}_t$, the next iterate
$x_{t+1} = x_t + s_t d_t$ is \emph{not} $\mathcal{F}_t$-measurable. It
is convenient to introduce the intermediate $\sigma$-field
\[
  \mathcal{H}_t \;\triangleq\; \sigma(\mathcal{F}_t,\, s_t),
\]
which records the history together with the fresh step but not the fresh
oracle sample. Then $x_{t+1}$, $w_t$ and $d_t$ are
$\mathcal{H}_t$-measurable, and the oracle noise at step $t+1$ (both the
gradient noise at $x_{t+1}$ and the Hessian-vector noise at $x_{t+1}$)
is centered conditional on $\mathcal{H}_t$. The path residual of the
Stein correction, in contrast, is centered conditional on
$\mathcal{F}_t$ (Section~\ref{app:stein}) but not conditional on
$\mathcal{H}_t$. The two kinds of randomness are therefore controlled at
different conditioning levels, and we keep them separate throughout.
Across iterations, the resulting error sequences are martingale
differences with respect to $\{\mathcal{F}_t\}$; they are not
independent.

\subsection{Assumptions}
\label{app:assumptions}

The assumptions below match those stated in the main text
(Assumptions~\ref{ass:well_defined}--\ref{ass:noise}); we restate them
here for self-containedness.

\begin{assumption}[Well-Defined Problem]
\label{ass:app:wd}
$f : \mathcal{E} \to \mathbb{R}$ is bounded below by $f_* > -\infty$; let
$\Delta_0 \triangleq f(x_0) - f_*$.
\end{assumption}

\begin{assumption}[Pointwise $(L_0, L_1)$-Smoothness]
\label{ass:app:smooth}
$f$ is twice continuously differentiable and there exist constants
$L_0, L_1 \geq 0$ such that
\begin{equation}
  \|\nabla^2 f(x)\|_{\mathrm{op}}
  \;\leq\; L_0 + L_1 \|\nabla f(x)\|_*
  \qquad \text{for every } x \in \mathcal{E}.
  \label{eq:app:L0L1}
\end{equation}
\end{assumption}

\begin{assumption}[Norm Compatibility]
\label{ass:app:kappa}
There exists $\kappa \geq 1$ such that $\|u\|_* \leq \kappa \|u\|_2$ and
$\|u\|_2 \leq \kappa \|u\|_*$ for every $u \in \mathcal{E}^* \cup
\mathcal{E}$.
\end{assumption}

\begin{assumption}[Affine Heavy-Tailed Noise]
\label{ass:app:noise}
The stochastic gradient $g(x) = \nabla f_\xi(x)$ and the Hessian-vector
product oracle $H(x)w = \nabla^2 f_\xi(x) w$ are unbiased, and there
exist constants $\sigma_g, \alpha_g, \sigma_h, \alpha_h \geq 0$ and
exponents $p, q \in (1, 2]$ such that, for every $x \in \mathcal{E}$ and
every deterministic $w \in \mathcal{E}$ independent of the oracle noise
$\xi$,
\begin{align}
  \mathbb{E}_\xi\bigl[\|g(x) - \nabla f(x)\|_2^p \,\big|\, x\bigr]
    &\leq \sigma_g^p + \alpha_g^p \|\nabla f(x)\|_*^p,
  \label{eq:app:noise-g} \\
  \mathbb{E}_\xi\bigl[\|(H(x) - \nabla^2 f(x))w\|_2^q \,\big|\, x, w\bigr]
    &\leq \bigl(\sigma_h^q + \alpha_h^q \|\nabla f(x)\|_*^q\bigr) \|w\|^q.
  \label{eq:app:noise-h}
\end{align}
\end{assumption}

\paragraph{Oracle model.} The analysis uses only the unbiasedness of the two oracle outputs and the moment bounds of Assumption~\ref{ass:app:noise}. The symbols $\nabla f_\xi(x)$ and $\nabla^2 f_\xi(x)w$ denote stochastic estimators of $\nabla f(x)$ and $\nabla^2 f(x)w$; they need not be derivatives of a common sample function $f_\xi$, and the gradient and HVP outputs at the same point may come from the same sample. In particular, the path residual of the correction involves only the true Hessian $\nabla^2 f$, and the oracle noise is bounded separately (Section~\ref{app:stein}). Differentiating a sampled loss by automatic differentiation, which lets the gradient and HVP share a forward pass, is one instance of this model. Appendix~\ref{app:lower-bound} uses the general model for the comparison with lower bounds.

\smallskip\noindent\textbf{Stepsize regime (unconstrained only).}
For the unconstrained algorithm RanSOM-E, we impose the stepsize
condition
\begin{equation}
  q L_1 \rho\, \eta_t \;\leq\; \tfrac12,
  \label{eq:app:stepsize-mgf}
\end{equation}
which ensures the moment generating function of $s_t \sim
\mathrm{Exp}(1/\eta_t)$ is finite at rate $qL_1\rho$. No such condition
is needed for RanSOM-B, whose step $s_t \in [0,1]$ is bounded.

\subsection{Stein-Type Identities}
\label{app:stein}

We begin by deriving the two Stein identities used to construct the
unbiased Hessian-vector-product correction. Both follow from a single
integration-by-parts formula.

\begin{lemma}[Master identity]
\label{lem:app:master}
Let $g : \mathbb{R} \to \mathbb{R}^d$ be continuously differentiable, and
let $s$ be a non-negative random variable with density $f_s$ and
survival function $\bar F_s(z) = \mathbb{P}(s > z)$ satisfying
\begin{equation}
  \int_0^\infty \|g'(z)\|_2\, \bar F_s(z)\, dz < \infty.
  \label{eq:app:master-cond}
\end{equation}
Then
\begin{equation}
  \mathbb{E}[g(s) - g(0)]
  \;=\; \int_0^\infty g'(z)\, \bar F_s(z)\, dz.
  \label{eq:app:master}
\end{equation}
\end{lemma}

\begin{proof}
Since $g$ is continuously differentiable, $g'$ is bounded on each
compact interval, so the fundamental theorem of calculus applies on the
finite segment $[0,s]$ and gives $g(s) - g(0) = \int_0^s g'(z)\, dz$ for
every realization of $s$; no growth condition on $g'$ is required for
this step. Writing it as an indicator integral and taking expectation,
\begin{equation*}
  \mathbb{E}[g(s) - g(0)]
  = \mathbb{E}\!\left[\int_0^\infty g'(z)\, \mathbf{1}_{\{z \leq s\}}\, dz\right]
  = \int_0^\infty g'(z)\, \mathbb{P}(s \geq z)\, dz,
\end{equation*}
where the interchange of expectation and integral is Tonelli's theorem
applied to the non-negative integrand $\|g'(z)\|_2\,
\mathbf{1}_{\{z \leq s\}}$, whose iterated integral is exactly
\eqref{eq:app:master-cond} and hence finite. Since $\mathbb{P}(s = z) =
0$ for a continuous distribution, $\mathbb{P}(s \geq z) = \bar F_s(z)$.
\end{proof}

\begin{remark}[When the condition holds]
\label{rem:app:master-cond}
Condition~\eqref{eq:app:master-cond} is what the two instantiations
below require. We state it explicitly for each.

\emph{Constrained (Beta) case.} Here $s \sim \mathrm{Beta}(1,K)$ has
survival function $\bar F_s(z) = (1-z)^K$ supported on $[0,1]$ and
bounded by $1$, so
\[
  \int_0^\infty \|g'(z)\|_2\, \bar F_s(z)\, dz
  = \int_0^1 \|g'(z)\|_2\, (1-z)^K\, dz
  \le \int_0^1 \|g'(z)\|_2\, dz
  < \infty,
\]
the last integral being finite because $g'$ is continuous on the
compact interval $[0,1]$. Condition~\eqref{eq:app:master-cond} is
therefore \emph{automatic} for any continuously differentiable $g$, with
no restriction on the stepsize or the growth of $g'$.

\emph{Unconstrained (exponential) case.} Here $s \sim
\mathrm{Exp}(\lambda)$ has survival function $\bar F_s(z) = e^{-\lambda
z}$. With $g'(z) = \nabla^2 f(x_t + z d_t)\, d_t$ and $\|d_t\| \le \rho$,
the pointwise $(L_0,L_1)$-smoothness of
Assumption~\ref{ass:geometric_smoothness} and the norm conversion
$\|\cdot\|_2 \le \kappa\|\cdot\|_*$ give $\|g'(z)\|_2 \le \kappa\rho(L_0
+ L_1\|\nabla f(x_t + z d_t)\|_*)$. Along the ray parameter $z$ (which is
arclength only when $\|d_t\| = 1$), the gradient norm $\varphi(z) =
\|\nabla f(x_t + z d_t)\|_*$ is absolutely continuous and satisfies
$\varphi'(z) \le \rho(L_0 + L_1\varphi(z))$ almost everywhere (for a
general norm, $\varphi$ need not be differentiable everywhere; the
inequality holds for its upper Dini derivative and integrates), so by
Gr\"onwall
(as in the proof of Lemma~\ref{lem:app:hess-noise}) $\varphi(z) \le
e^{L_1\rho z}(\|\nabla f(x_t)\|_* + L_0\rho z)$. Hence
\[
  \|g'(z)\|_2
  \;\le\; \kappa\rho L_0 + \kappa\rho L_1 e^{L_1\rho z}\bigl(\|\nabla f(x_t)\|_* + L_0\rho z\bigr),
\]
and substituting into~\eqref{eq:app:master-cond},
\[
  \int_0^\infty \|g'(z)\|_2\, e^{-\lambda z}\, dz
  \;\le\; \kappa\int_0^\infty
    \Bigl[\rho L_0 + \rho L_1 e^{L_1\rho z}\bigl(\|\nabla f(x_t)\|_* + L_0\rho z\bigr)\Bigr]
    e^{-\lambda z}\, dz,
\]
which is finite whenever $\lambda > L_1\rho$, i.e. $L_1 \rho \eta_t < 1$
with $\lambda = 1/\eta_t$. The stepsize regime $q L_1 \rho \eta_t \le
\tfrac12$ used throughout the unconstrained analysis implies
$\lambda = 1/\eta_t \ge 2 q L_1 \rho > L_1 \rho$, so the condition holds.
Global integrability of $g'$ on $[0,\infty)$ is not needed, which
matters precisely because the Gr\"onwall bound permits $\|\nabla f\|$,
and hence $\|g'\|_2$, to grow along the ray.
\end{remark}

\begin{lemma}[Stein-type identities]
\label{lem:app:stein}
Let $g : \mathbb{R} \to \mathbb{R}^d$ be as in
Lemma~\ref{lem:app:master}.
\begin{enumerate}[topsep=2pt,itemsep=0pt]
  \item \emph{(Exponential.)} If $s \sim \mathrm{Exp}(\lambda)$ with
  $\lambda > 0$, then
  \begin{equation}
    \mathbb{E}[g(s) - g(0)]
    \;=\; \tfrac{1}{\lambda}\, \mathbb{E}[g'(s)].
    \label{eq:app:stein-exp}
  \end{equation}
  \item \emph{(Beta.)} If $s \sim \mathrm{Beta}(1, K)$ with $K > 0$,
  then
  \begin{equation}
    \mathbb{E}[g(s) - g(0)]
    \;=\; \mathbb{E}\!\left[\tfrac{1 - s}{K}\, g'(s)\right].
    \label{eq:app:stein-beta}
  \end{equation}
\end{enumerate}
\end{lemma}

\begin{proof}
\emph{Exponential.} The density is $f_s(z) = \lambda e^{-\lambda z}$ for
$z \geq 0$; integrating gives the survival function $\bar F_s(z) =
e^{-\lambda z} = \tfrac{1}{\lambda} f_s(z)$. Substituting into
\eqref{eq:app:master},
\[
  \mathbb{E}[g(s) - g(0)]
  = \int_0^\infty g'(z) \cdot \tfrac{1}{\lambda} f_s(z)\, dz
  = \tfrac{1}{\lambda}\, \mathbb{E}[g'(s)].
\]

\emph{Beta.} The density on $[0,1]$ is $f_s(z) = K(1-z)^{K-1}$ (since
$\mathrm{B}(1,K) = 1/K$). Integrating, $\bar F_s(z) = (1-z)^K =
\tfrac{1-z}{K} f_s(z)$ for $z \in [0,1]$. Substituting,
\[
  \mathbb{E}[g(s) - g(0)]
  = \int_0^1 g'(z)\, \tfrac{1-z}{K} f_s(z)\, dz
  = \mathbb{E}\!\left[\tfrac{1-s}{K}\, g'(s)\right]. \qedhere
\]
\end{proof}

\begin{remark}[Unified Stein weight]
\label{rem:app:stein-weight}
Both identities take the form $\mathbb{E}[g(s) - g(0)] =
\mathbb{E}[w(s)\, g'(s)]$ with Stein weight $w(z) = \bar
F_s(z)/f_s(z)$. For RanSOM-E ($s \sim \mathrm{Exp}(1/\eta_t)$), $w(s) =
\eta_t$ is \emph{deterministic}. For RanSOM-B ($s \sim
\mathrm{Beta}(1, K_t)$ with $K_t = \eta_t^{-1} - 1$), $w(s) = (1-s)/K_t$
is \emph{random} but satisfies $0 \leq w(s) \leq 1/K_t$ a.s. In both
cases $\mathbb{E}[w(s)] = \mathbb{E}[\bar F_s(s)/f_s(s)] = \eta_t$, so
$\delta_{t+1} = w_t \cdot h_{t+1}$ is an unbiased estimator of the
gradient change in expectation over $s_t$ given $\mathcal{F}_t$
(cf.~\eqref{eq:app:stein-momentum}); it is not unbiased conditionally on
the realized $s_t$. For a general law with density $f_s$, the identity
$\mathbb{E}[g(s) - g(0)] = \mathbb{E}[w(s)g'(s)]$ requires $f_s > 0$
almost everywhere on $\{\bar F_s > 0\}$. Without this condition it can
fail: for $s$ uniform on $[1,2]$ and $g$ increasing from $0$ to $1$ on
$[0,1]$ and constant afterwards, $\mathbb{E}[g(s) - g(0)] = 1$ while
$\mathbb{E}[w(s)g'(s)] = 0$.
\end{remark}

\smallskip\noindent\textbf{Application to momentum bias.}
Apply Lemma~\ref{lem:app:stein} with $g(s) = \nabla f(x_t + s d_t)$, so
$g'(s) = \nabla^2 f(x_t + s d_t)\, d_t$. Writing $x_{t+1} = x_t + s_t
d_t$,
\begin{equation}
  \mathbb{E}_{s_t}\bigl[\nabla f(x_{t+1}) - \nabla f(x_t)\bigr]
  \;=\; \mathbb{E}_{s_t}\bigl[w_t\, \nabla^2 f(x_{t+1})\, d_t\bigr].
  \label{eq:app:stein-momentum}
\end{equation}
This is the defining property of the RanSOM correction: the left side is
exactly the momentum bias~(\ref{eq:bias_def}) we need to estimate, and
the right side is $\mathbb{E}[\delta_{t+1}]$ (the expectation of the
Hessian-vector-product correction computed by
Algorithms~\ref{alg:expsom}--\ref{alg:betasom}).

\subsection{Normalized Moment Constants}
\label{app:constants}

The convergence analysis depends on a small number of normalized moment
constants associated with the distributions of $s_t$ and $w_t$. We
define them once here; bounds for the specific distributions of
Algorithms~\ref{alg:expsom}--\ref{alg:betasom} are given at the end of
this subsection.

\begin{definition}[Normalized moments]
\label{def:app:moments}
For a given $q \in (1, 2]$, define
\begin{equation}
  M_w \;\triangleq\; \mathbb{E}\!\left[(|w_t|/\eta_t)^q\right],
  \qquad
  M_{ws} \;\triangleq\; \mathbb{E}\!\left[(|w_t|/\eta_t + s_t/\eta_t)^q\right],
  \qquad
  C_s \;\triangleq\; \mathbb{E}[s_t^2]/\eta_t^2.
  \label{eq:app:moments}
\end{equation}
For the unconstrained setting, let $u \triangleq L_1 \rho \eta_t$
satisfying~\eqref{eq:app:stepsize-mgf}, and define
\begin{equation}
  \widetilde C_{B,q} \;\triangleq\;
    \frac{2^{q-1}\bigl(1 + \Gamma(q+1)\bigr)}{(1 - qu)^{q+1}},
  \qquad
  \widetilde C_{A,q} \;\triangleq\;
    2^{q-1}\!\left[M_{ws} +
      u^q\cdot\frac{2^{q-1}\bigl(\Gamma(q+1) + \Gamma(2q+1)\bigr)}
        {(1 - qu)^{2q+1}}\right].
  \label{eq:app:CAB}
\end{equation}
\end{definition}

\smallskip\noindent\textbf{Bounds for RanSOM-E ($s_t \sim \mathrm{Exp}(1/\eta_t)$, $w_t = \eta_t$).}
Since $w_t$ is deterministic, $M_w = 1$. The ratio $s_t/\eta_t \sim
\mathrm{Exp}(1)$, so $\mathbb{E}[(s_t/\eta_t)^k] = \Gamma(k+1) = k!$ for
all $k > 0$. By $(a+b)^q \leq 2^{q-1}(a^q + b^q)$,
\[
  M_{ws} = \mathbb{E}[(1 + s_t/\eta_t)^q]
  \leq 2^{q-1}\bigl(1 + \Gamma(q+1)\bigr)
  \leq 2\cdot 3 = 6
  \qquad \text{(using } q \leq 2,\ \Gamma(q+1) \leq 2\text{).}
\]
Also $C_s = \mathbb{E}[(s_t/\eta_t)^2] = \Gamma(3) = 2$.

\smallskip\noindent\textbf{Bounds for RanSOM-B ($s_t \sim \mathrm{Beta}(1, K_t)$, $w_t = (1-s_t)/K_t$, $K_t = \eta_t^{-1} - 1$).}
Here $s_t \in [0,1]$ and $w_t \in [0, 1/K_t] = [0, \eta_t/(1-\eta_t)]$,
so both are bounded. Using $\mathbb{E}[s_t] = 1/(1+K_t) = \eta_t$ and
$\mathbb{E}[s_t^2] = 2/[(1+K_t)(2+K_t)]$,
\[
  C_s = \frac{\mathbb{E}[s_t^2]}{\eta_t^2}
  = \frac{2}{(1+K_t)(2+K_t)}\cdot \frac{1}{\eta_t^2}
  = \frac{2}{1 + \eta_t}
  \leq 2.
\]
For the Stein weight, $|w_t|/\eta_t = (1-s_t)/[\eta_t K_t] =
(1-s_t)/(1-\eta_t) \leq 1/(1-\eta_t) \leq 2$ for $\eta_t \leq 1/2$. Thus
\[
  M_w = \mathbb{E}\!\left[\bigl(\tfrac{1-s_t}{1-\eta_t}\bigr)^q\right]
  \leq \frac{1}{(1-\eta_t)^q} \leq 2^q \leq 4,
\]
using $\mathbb{E}[(1-s_t)^q] \leq \mathbb{E}[(1-s_t)] = K_t/(K_t+1) = 1 -
\eta_t \leq 1$. For $M_{ws}$, by $(a+b)^q \leq 2^{q-1}(a^q+b^q)$ and Jensen's inequality,
$\mathbb{E}[(s_t/\eta_t)^q] \le (\mathbb{E}[(s_t/\eta_t)^2])^{q/2} = (2/(1+\eta_t))^{q/2} \le 2^{q/2}$,
so
\[
  M_w \le \hat M_w \triangleq 2^q, \qquad
  M_{ws} \le \hat M_{ws} \triangleq 2^{q-1}\bigl(2^q + 2^{q/2}\bigr) \le 12,
  \qquad C_s = \frac{2}{1+\eta_t} \le 2 .
\]
The exact values of $M_w, M_{ws}, C_s$ depend on $\eta_t$ (for instance,
$M_w = 1/(1-\eta_t^2)$ at $q = 2$), but for $\eta_t \le 1/2$ they are
bounded by the absolute constants $\hat M_w, \hat M_{ws}$ and $2$, which do
not depend on $\eta_t$. 

\smallskip\noindent\textbf{Useful observation.}
Under the unconstrained stepsize condition~\eqref{eq:app:stepsize-mgf},
$qu \leq 1/2$, so
\[
  \frac{1}{(1-qu)^{q+1}} \leq 2^{q+1} \leq 8,
  \qquad
  \frac{1}{(1-qu)^{2q+1}} \leq 2^{2q+1} \leq 32,
\]
and both $\widetilde C_{B,q}$ and $\widetilde C_{A,q}$ are bounded by
absolute constants.


\subsection{Hessian-Noise Lemma (Unconstrained)}
\label{app:hessian-noise}

Define the centered bias error at step $t+1$ by
\begin{equation}
  \Psi_{t+1}
  \;\triangleq\;
  w_t\, H_\xi(x_{t+1})\, d_t
    - \bigl(\nabla f(x_{t+1}) - \nabla f(x_t)\bigr),
  \label{eq:app:Psi-def}
\end{equation}
the difference between the RanSOM correction $\delta_{t+1} = w_t
h_{t+1}$ and the true gradient change. By the Stein
identity~\eqref{eq:app:stein-momentum} and unbiasedness of the HVP
oracle, $\mathbb{E}[\Psi_{t+1} \mid \mathcal{F}_t] = 0$, so $\{\Psi_j\}$
is a conditional martingale difference sequence with respect to
$\{\mathcal{F}_{j-1}\}$.

\begin{lemma}[Hessian-Noise Bound, unconstrained]
\label{lem:app:hess-noise}
Consider RanSOM-E with $s_t \sim \mathrm{Exp}(1/\eta_t)$ and $w_t =
\eta_t$ deterministic. Under
Assumptions~\ref{ass:app:smooth}--\ref{ass:app:noise}, the stepsize
condition~\eqref{eq:app:stepsize-mgf}, and a stepsize cap $\eta_t \le
\bar\eta$,
\begin{equation}
  \mathbb{E}\bigl[\|\Psi_{t+1}\|_2^q \,\big|\, \mathcal{F}_t\bigr]
  \;\leq\;
  C_\delta^q\, \eta_t^q\,
    \bigl(\bar\sigma_h^q + \bar\alpha_h^q\, \|\nabla f(x_t)\|_*^q\bigr),
  \label{eq:app:hess-noise}
\end{equation}
where
\begin{align}
  \bar\sigma_h^q
  &\triangleq \rho^q\,\Bigl[\sigma_h^q
    + \kappa^q L_0^q\, \widetilde C_{A,q}
    + \alpha_h^q L_0^q \rho^q \bar\eta^q\,
      \tfrac{2^{q-1}\Gamma(q+1)}{(1-qu)^{q+1}}\Bigr],
  \label{eq:app:sigma-bar} \\
  \bar\alpha_h^q
  &\triangleq \rho^q\,\bigl[\tfrac{2^{q-1}\,\alpha_h^q}{1 - qu}
    + \kappa^q L_1^q\, \widetilde C_{B,q}\bigr],
  \label{eq:app:alpha-bar} \\
  C_\delta^q
  &\triangleq 2^{2(q-1)}\cdot \max\bigl(M_w,\; M_{ws}\cdot 2^{q-1}\bigr).
  \label{eq:app:C-delta}
\end{align}
The oracle-noise contributions ($\sigma_h$, $\alpha_h$) carry no factor
$\kappa$, because Assumption~\ref{ass:app:noise} bounds them directly in
$\|\cdot\|_2$; the path-residual contributions ($L_0$, $L_1$) carry
$\kappa^q$ from the conversion of the operator-norm bound to
$\|\cdot\|_2$. The third term in $\bar\sigma_h^q$ is the baseline
contribution of the affine Hessian noise evaluated at the random endpoint
$x_{t+1}$, transferred to the anchor $x_t$ by the Gr\"onwall estimate; it
is proportional to $\alpha_h^q$ and vanishes when $\alpha_h = 0$, but it
cannot be absorbed into a constant independent of $\alpha_h$.

At $L_1 = 0$ and $\alpha_h = 0$ the formula reduces to
$\bar\sigma_h^q \asymp \rho^q(\sigma_h^q + \kappa^q L_0^q)$, matching
$\bar\sigma_h = \rho(\sigma_h + \kappa L_0)$ of
Lemma~\ref{lemma:main_error} up to a numerical constant absorbed into
$C_\delta$.
\end{lemma}

\begin{remark}[Anchor at $x_t$, not $x_{t+1}$]
\label{rem:app:anchor}
The gradient-dependent term in~\eqref{eq:app:hess-noise} is anchored at
$x_t$, which is $\mathcal{F}_t$-measurable, because this lemma
conditions on $\mathcal{F}_t$. The issue is predictability, not
independence: the weighted-sum lemma actually used downstream
(Lemma~\ref{lem:app:affine-mds}) requires the affine factor $Y_j$ to be
measurable with respect to the $\sigma$-field at which the martingale
difference is centered. Under $\mathcal{F}_t$ that forces the anchor
$x_t$, and transporting the endpoint bound to $x_t$ is what produces the
$\alpha_h L_0$ baseline term above. Under the finer $\sigma$-field
$\mathcal{H}_t = \sigma(\mathcal{F}_t, s_t)$ the endpoint $x_{t+1}$ is
itself measurable, so the oracle part may be anchored there directly;
Lemma~\ref{lem:app:error} uses this to avoid the baseline term.
\end{remark}

\begin{proof}[Proof of Lemma~\ref{lem:app:hess-noise}]
We decompose $\Psi_{t+1}$ into a stochastic Hessian-noise term and a
deterministic pathwise approximation term, then bound each separately.

\smallskip\noindent\textbf{Decomposition.}
Writing $x_\tau = x_t + \tau s_t d_t$ for $\tau \in [0, 1]$ (so $x_0 =
x_t$ and $x_1 = x_{t+1}$), the fundamental theorem of calculus gives
\[
  \nabla f(x_{t+1}) - \nabla f(x_t)
  = \int_0^{s_t} \nabla^2 f(x_t + z d_t)\, d_t\, dz
  = s_t \int_0^1 \nabla^2 f(x_\tau)\, d_t\, d\tau.
\]
Substituting into~\eqref{eq:app:Psi-def},
\begin{align}
  \Psi_{t+1}
  &= \underbrace{w_t\bigl[H_\xi(x_{t+1}) - \nabla^2 f(x_{t+1})\bigr] d_t}_{T_1}
  + \underbrace{\Bigl[w_t\, \nabla^2 f(x_{t+1})\, d_t
      - s_t\!\int_0^1 \nabla^2 f(x_\tau)\, d_t\, d\tau\Bigr]}_{T_2}.
  \label{eq:app:Psi-decomp}
\end{align}
By $(a+b)^q \leq 2^{q-1}(a^q + b^q)$ (valid for $a, b \geq 0$ and $q \geq 1$),
\begin{equation}
  \mathbb{E}\bigl[\|\Psi_{t+1}\|_2^q \,\big|\, \mathcal{F}_t\bigr]
  \;\leq\; 2^{q-1}\Bigl(
    \mathbb{E}[\|T_1\|_2^q \mid \mathcal{F}_t]
    + \mathbb{E}[\|T_2\|_2^q \mid \mathcal{F}_t]\Bigr).
  \label{eq:app:Psi-split}
\end{equation}

\smallskip\noindent\textbf{Step 1 (Gradient transfer via Gr\"onwall).}
We first establish a pathwise bound that transfers the gradient norm at
any point $x_\tau$ on the interpolation segment back to the
$\mathcal{F}_t$-measurable anchor $x_t$, using pointwise
$(L_0, L_1)$-smoothness (Assumption~\ref{ass:app:smooth}).

Let $\varphi(\tau) \triangleq \|\nabla f(x_\tau)\|_*$. For $\tau_1 < \tau_2$,
\[
  |\varphi(\tau_2) - \varphi(\tau_1)|
  \le \|\nabla f(x_{\tau_2}) - \nabla f(x_{\tau_1})\|_*
  = \Bigl\|\int_{\tau_1}^{\tau_2} \nabla^2 f(x_\tau)(s_t d_t)\,d\tau\Bigr\|_*
  \le \int_{\tau_1}^{\tau_2} \|\nabla^2 f(x_\tau)\|_{\mathrm{op}}\, s_t\|d_t\|\,d\tau ,
\]
so $\varphi$ is absolutely continuous and, for almost every $\tau$,
$\varphi'(\tau) \le \|\nabla^2 f(x_\tau)\|_{\mathrm{op}}\, s_t\|d_t\|$. Using
$\|d_t\| \leq \rho$ and Assumption~\ref{ass:app:smooth} pointwise at
$x_\tau$,
\[
  \varphi'(\tau) \leq \rho s_t\, \bigl(L_0 + L_1 \varphi(\tau)\bigr).
\]
This is a linear differential inequality. By Gr\"onwall's lemma (e.g.
multiply both sides by $e^{-L_1\rho s_t \tau}$ and integrate), for all
$\tau \in [0, 1]$,
\begin{equation}
  \varphi(\tau)
  \leq e^{L_1 \rho s_t \tau}\bigl(\varphi(0) + L_0 \rho s_t \tau\bigr)
  \leq e^{L_1 \rho s_t}\bigl(\|\nabla f(x_t)\|_* + L_0 \rho s_t\bigr).
  \label{eq:app:gronwall}
\end{equation}
In particular, at $\tau = 1$,
\begin{equation}
  \|\nabla f(x_{t+1})\|_*
  \leq e^{L_1 \rho s_t}\bigl(\|\nabla f(x_t)\|_* + L_0 \rho s_t\bigr).
  \label{eq:app:gronwall-endpoint}
\end{equation}

\smallskip\noindent\textbf{Step 2 (Bounding $\mathbb{E}[\|T_1\|_2^q \mid \mathcal{F}_t]$).}
The Hessian noise is $T_1 = w_t [H_\xi(x_{t+1}) - \nabla^2 f(x_{t+1})]
d_t$. Condition on everything except the oracle noise:
Assumption~\ref{ass:app:noise} bounds the $\|\cdot\|_2^q$ moment
directly (no norm compatibility needed), with $\|d_t\| \leq \rho$:
\begin{equation}
  \mathbb{E}_\xi\bigl[\|T_1\|_2^q \,\big|\, x_{t+1}, s_t, w_t\bigr]
  \;\leq\; |w_t|^q\, \rho^q\,
    \bigl(\sigma_h^q + \alpha_h^q \|\nabla f(x_{t+1})\|_*^q\bigr).
  \label{eq:app:T1-cond}
\end{equation}
Raise~\eqref{eq:app:gronwall-endpoint} to the $q$-th power and apply
$(a+b)^q \leq 2^{q-1}(a^q + b^q)$:
\begin{equation}
  \|\nabla f(x_{t+1})\|_*^q
  \leq 2^{q-1}\, e^{q L_1 \rho s_t}\,
    \bigl(\|\nabla f(x_t)\|_*^q + L_0^q \rho^q s_t^q\bigr).
  \label{eq:app:grad-transfer-q}
\end{equation}
For RanSOM-E, $w_t = \eta_t$ is deterministic, so $|w_t|^q = \eta_t^q$.
Taking $\mathbb{E}[\,\cdot \mid \mathcal{F}_t]$ of~\eqref{eq:app:T1-cond}
and substituting~\eqref{eq:app:grad-transfer-q}:
\begin{align}
  \mathbb{E}[\|T_1\|_2^q \mid \mathcal{F}_t]
  &\leq \eta_t^q\rho^q\,\sigma_h^q
    + 2^{q-1}\eta_t^q \rho^q\alpha_h^q\,
      \mathbb{E}\!\left[e^{qL_1\rho s_t}\right] \|\nabla f(x_t)\|_*^q
  \notag \\
  &\quad + 2^{q-1}\eta_t^q \rho^q\alpha_h^q L_0^q \rho^q\,
      \mathbb{E}\!\left[s_t^q e^{qL_1\rho s_t}\right],
  \label{eq:app:T1-expanded}
\end{align}
where we used that $\|\nabla f(x_t)\|_*$ is $\mathcal{F}_t$-measurable.

The two exponential-type moments admit closed forms. For $s_t \sim
\mathrm{Exp}(1/\eta_t)$ and any $\lambda < 1/\eta_t$,
\begin{equation}
  \mathbb{E}\!\left[s_t^k e^{\lambda s_t}\right]
  = \frac{1}{\eta_t}\int_0^\infty z^k e^{-z(1/\eta_t - \lambda)}\,dz
  = \frac{\Gamma(k+1)\,\eta_t^k}{(1 - \lambda \eta_t)^{k+1}}.
  \label{eq:app:exp-moment}
\end{equation}
Setting $\lambda = qL_1\rho$ and $u = L_1\rho\eta_t$ (so $\lambda\eta_t =
qu \leq 1/2$ by~\eqref{eq:app:stepsize-mgf}):
\begin{equation}
  \mathbb{E}[e^{qL_1\rho s_t}] = \frac{1}{1 - qu}, \qquad
  \mathbb{E}[s_t^q e^{qL_1\rho s_t}]
    = \frac{\Gamma(q+1)\, \eta_t^q}{(1 - qu)^{q+1}}.
  \label{eq:app:exp-mgf}
\end{equation}
Substituting into~\eqref{eq:app:T1-expanded},
\begin{equation}
  \mathbb{E}[\|T_1\|_2^q \mid \mathcal{F}_t]
  \leq \eta_t^q \rho^q \!\left[\sigma_h^q
    + \frac{2^{q-1}\alpha_h^q}{1 - qu}\, \|\nabla f(x_t)\|_*^q
    + \alpha_h^q L_0^q \rho^q \eta_t^q\,
      \frac{2^{q-1}\Gamma(q+1)}{(1-qu)^{q+1}}\right].
  \label{eq:app:T1-final}
\end{equation}

\smallskip\noindent\textbf{Step 3 (Pathwise bound for $T_2$).}
By definition,
\[
  T_2 = w_t\, \nabla^2 f(x_{t+1})\, d_t
    - s_t \int_0^1 \nabla^2 f(x_\tau)\, d_t\, d\tau.
\]
Applying the triangle inequality and the induced-operator-norm
inequality $\|\nabla^2 f(x) u\|_* \leq \|\nabla^2 f(x)\|_{\mathrm{op}}
\|u\|$, then converting $\|\cdot\|_*$ to $\|\cdot\|_2$ via
Assumption~\ref{ass:app:kappa}:
\begin{equation}
  \|T_2\|_2 \leq \kappa\, \|T_2\|_*
  \leq \kappa\rho\Bigl(|w_t|\cdot \|\nabla^2 f(x_{t+1})\|_{\mathrm{op}}
    + s_t \int_0^1 \|\nabla^2 f(x_\tau)\|_{\mathrm{op}}\, d\tau\Bigr).
  \label{eq:app:T2-norm}
\end{equation}
Apply Assumption~\ref{ass:app:smooth} pointwise at each $x_\tau$ and
$x_{t+1}$, then use~\eqref{eq:app:gronwall}:
\[
  \|\nabla^2 f(x_\tau)\|_{\mathrm{op}}
  \leq L_0 + L_1 \varphi(\tau)
  \leq L_0 + L_1 e^{L_1\rho s_t}\bigl(\|\nabla f(x_t)\|_* + L_0\rho s_t\bigr),
\]
uniform in $\tau \in [0,1]$; the same bound holds at $\tau = 1$ for
$\|\nabla^2 f(x_{t+1})\|_{\mathrm{op}}$. Substituting
into~\eqref{eq:app:T2-norm},
\begin{equation}
  \|T_2\|_2
  \leq \kappa\rho\, (|w_t| + s_t)\,\Bigl[L_0
    + L_1 e^{L_1\rho s_t}\bigl(\|\nabla f(x_t)\|_* + L_0 \rho s_t\bigr)\Bigr].
  \label{eq:app:T2-pathwise}
\end{equation}

\smallskip\noindent\textbf{Step 4 (Splitting $T_2$ into baseline and gradient-linear parts).}
Distribute the factor $(|w_t| + s_t)$ over the bracket in
\eqref{eq:app:T2-pathwise} and collect by powers of $\|\nabla
f(x_t)\|_*$:
\[
  \|T_2\|_2 \leq \kappa\rho\,\bigl(\widetilde A
    + \widetilde B\, \|\nabla f(x_t)\|_*\bigr),
\]
where
\begin{align}
  \widetilde A &:= L_0\, (|w_t| + s_t)\,
    \bigl(1 + L_1\rho s_t e^{L_1\rho s_t}\bigr),
  \label{eq:app:A-tilde} \\
  \widetilde B &:= L_1\, (|w_t| + s_t)\, e^{L_1\rho s_t}.
  \label{eq:app:B-tilde}
\end{align}
Since $(s_t, w_t)$ are independent of $\mathcal{F}_t$ and $\|\nabla
f(x_t)\|_*$ is $\mathcal{F}_t$-measurable, $(a+b)^q \leq 2^{q-1}(a^q +
b^q)$ gives
\begin{equation}
  \mathbb{E}[\|T_2\|_2^q \mid \mathcal{F}_t]
  \leq \kappa^q \rho^q\, 2^{q-1}
    \bigl(\mathbb{E}[\widetilde A^q]
      + \mathbb{E}[\widetilde B^q]\, \|\nabla f(x_t)\|_*^q\bigr).
  \label{eq:app:T2-cond}
\end{equation}

\smallskip\noindent\textbf{Step 5 (Moments of $\widetilde B$).}
By $(a+b)^q \leq 2^{q-1}(a^q + b^q)$,
\begin{align*}
  \mathbb{E}[\widetilde B^q]
  &= L_1^q\, \mathbb{E}\!\left[(|w_t|+s_t)^q e^{qL_1\rho s_t}\right]
  \\
  &\leq L_1^q\, 2^{q-1}\!\left(|\eta_t|^q\, \mathbb{E}[e^{qL_1\rho s_t}]
      + \mathbb{E}[s_t^q e^{qL_1\rho s_t}]\right)
  \qquad (\text{using } w_t = \eta_t).
\end{align*}
Substituting~\eqref{eq:app:exp-mgf},
\[
  \mathbb{E}[\widetilde B^q]
  \leq L_1^q\, 2^{q-1}\!\left[\frac{\eta_t^q}{1-qu}
    + \frac{\Gamma(q+1)\, \eta_t^q}{(1-qu)^{q+1}}\right]
  \leq L_1^q\, \eta_t^q\cdot
    \frac{2^{q-1}\bigl(1 + \Gamma(q+1)\bigr)}{(1-qu)^{q+1}}
  = L_1^q\, \eta_t^q\, \widetilde C_{B,q},
\]
where we used $(1-qu)^{-1} \leq (1-qu)^{-(q+1)}$ since $0 \leq qu < 1$.

\smallskip\noindent\textbf{Step 6 (Moments of $\widetilde A$).}
Applying $(a+b)^q \leq 2^{q-1}(a^q + b^q)$ inside
\eqref{eq:app:A-tilde}:
\[
  \widetilde A^q
  \leq L_0^q\, 2^{q-1}\bigl[(|w_t|+s_t)^q
    + (L_1\rho)^q\, (|w_t|+s_t)^q\, s_t^q\, e^{qL_1\rho s_t}\bigr].
\]
The first term's expectation is at most $M_{ws}\eta_t^q$ by
Definition~\ref{def:app:moments}. For the second, apply $(a+b)^q \leq
2^{q-1}(a^q+b^q)$ once more to $(|w_t|+s_t)^q = (\eta_t + s_t)^q \leq
2^{q-1}(\eta_t^q + s_t^q)$:
\begin{align*}
  \mathbb{E}\!\left[(\eta_t + s_t)^q s_t^q e^{qL_1\rho s_t}\right]
  &\leq 2^{q-1}\bigl(\eta_t^q\, \mathbb{E}[s_t^q e^{qL_1\rho s_t}]
    + \mathbb{E}[s_t^{2q} e^{qL_1\rho s_t}]\bigr)
  \\
  &\leq 2^{q-1}\!\left(\frac{\Gamma(q+1)\,\eta_t^{2q}}{(1-qu)^{q+1}}
    + \frac{\Gamma(2q+1)\,\eta_t^{2q}}{(1-qu)^{2q+1}}\right)
  \\
  &\leq \eta_t^{2q}\cdot
    \frac{2^{q-1}\bigl(\Gamma(q+1) + \Gamma(2q+1)\bigr)}{(1-qu)^{2q+1}},
\end{align*}
using $(1-qu)^{-(q+1)} \leq (1-qu)^{-(2q+1)}$. Multiplying by $(L_1\rho)^q$
and noting $(L_1\rho\eta_t)^q = u^q$,
\[
  (L_1\rho)^q\, \mathbb{E}\!\left[(\eta_t + s_t)^q s_t^q e^{qL_1\rho s_t}\right]
  \leq \eta_t^q \cdot u^q \cdot
    \frac{2^{q-1}\bigl(\Gamma(q+1) + \Gamma(2q+1)\bigr)}{(1-qu)^{2q+1}}.
\]
Combining, $\mathbb{E}[\widetilde A^q] \leq L_0^q\, \eta_t^q\, \widetilde
C_{A,q}$ with $\widetilde C_{A,q}$ as defined in~\eqref{eq:app:CAB}. The
$u^q$ factor is bounded by $2^{-q}$ under~\eqref{eq:app:stepsize-mgf}, so
the second term in the bracket of $\widetilde C_{A,q}$ is a numerical
constant.

\smallskip\noindent\textbf{Step 7 (Assembly).}
Substituting the moments of $\widetilde A$ and $\widetilde B$
into~\eqref{eq:app:T2-cond},
\begin{equation}
  \mathbb{E}[\|T_2\|_2^q \mid \mathcal{F}_t]
  \leq \kappa^q \rho^q\, 2^{q-1}\, \eta_t^q\,
    \bigl(L_0^q\, \widetilde C_{A,q}
      + L_1^q\, \widetilde C_{B,q}\, \|\nabla f(x_t)\|_*^q\bigr).
  \label{eq:app:T2-final}
\end{equation}
Now combine~\eqref{eq:app:T1-final} and~\eqref{eq:app:T2-final}
via~\eqref{eq:app:Psi-split}. The baseline $\sigma_h$-term of $T_1$
gives $\eta_t^q\rho^q\sigma_h^q$, multiplied by the outer $2^{q-1}$
from~\eqref{eq:app:Psi-split}; the $\widetilde A$-term of $T_2$ gives
$\kappa^q\rho^q 2^{q-1}\cdot 2^{q-1} L_0^q\widetilde C_{A,q}\eta_t^q =
\kappa^q\rho^q 2^{2(q-1)} L_0^q \widetilde C_{A,q}\eta_t^q$. The
remaining baseline term of $T_1$ in~\eqref{eq:app:T1-final} is
$2^{q-1}\cdot\alpha_h^q L_0^q\rho^{2q}\eta_t^{2q}\cdot
2^{q-1}\Gamma(q+1)/(1-qu)^{q+1}$; writing $\eta_t^{2q} = \eta_t^q\cdot
\eta_t^q \le \eta_t^q\,\bar\eta^q$ yields the third term of
$\bar\sigma_h^q$ in~\eqref{eq:app:sigma-bar}. The three baseline terms
together, with the factor $M_w = 1$ for RanSOM-E, are bounded by
$C_\delta^q\eta_t^q\bar\sigma_h^q$ with $\bar\sigma_h$ as
in~\eqref{eq:app:sigma-bar} and $C_\delta^q$ as
in~\eqref{eq:app:C-delta}. The constant $\kappa$ multiplies only the
path-residual ($L_0$) term: the oracle term is bounded directly in
$\|\cdot\|_2$ by Assumption~\ref{ass:app:noise}.

The gradient-coefficient terms: from $T_1$, $2^{q-1}\alpha_h^q/(1-qu)$
times $\rho^q$; from $T_2$, $2^{2(q-1)} \kappa^q L_1^q\widetilde
C_{B,q}\rho^q$. Their sum is~\eqref{eq:app:alpha-bar} up to the factor
$C_\delta^q$, absorbed into $C_\delta^q\bar\alpha_h^q$.

This proves~\eqref{eq:app:hess-noise}. At $L_1 = 0$ and $\alpha_h = 0$,
$\widetilde C_{A,q}$ reduces to $2^{q-1} M_{ws}$ and the third term of
$\bar\sigma_h^q$ vanishes, so $\bar\sigma_h^q =
\rho^q[\sigma_h^q + 2^{q-1}M_{ws}\kappa^q L_0^q] \asymp
\rho^q(\sigma_h + \kappa L_0)^q$, matching $\bar\sigma_h = \rho(\sigma_h
+ \kappa L_0)$ of the main text's Lemma~\ref{lemma:main_error} up to a
numerical constant absorbed into $C_\delta$.
\end{proof}


\subsection{Hessian-Noise Lemma (Constrained)}
\label{app:hessian-noise-constrained}

In the constrained setting, optimization is over a compact convex set
$\mathcal{C} \subseteq \mathcal{E}$ with diameter $D \triangleq
\sup_{x, y \in \mathcal{C}}\|x - y\|$. Since $f$ is continuously
differentiable and $\mathcal{C}$ is compact, the gradient is uniformly
bounded on $\mathcal{C}$:
\begin{equation}
  G_\mathcal{C} \;\triangleq\; \sup_{x \in \mathcal{C}}\|\nabla f(x)\|_*
  \;<\; \infty.
  \label{eq:app:grad-bound}
\end{equation}
This uniform bound lets us absorb the affine growth $\|\nabla f\|_*$
into an effective constant, avoiding the need to relate $\|\nabla f\|_*$
to the Frank-Wolfe gap. Define the \emph{effective constants}
\begin{equation}
  L \triangleq L_0 + L_1 G_\mathcal{C}, \qquad
  \tilde\sigma_g^p \triangleq \sigma_g^p + \alpha_g^p G_\mathcal{C}^p,
  \qquad
  \tilde\sigma_h^q \triangleq \sigma_h^q + \alpha_h^q G_\mathcal{C}^q.
  \label{eq:app:eff-const}
\end{equation}
Then, on $\mathcal{C}$:
\[
  \|\nabla^2 f(x)\|_{\mathrm{op}} \leq L,
  \qquad
  \mathbb{E}_\xi\bigl[\|(H-\nabla^2 f)w\|_2^q \mid x, w\bigr]
  \leq \tilde\sigma_h^q\,\|w\|^q,
  \qquad
  \mathbb{E}_\xi\bigl[\|g - \nabla f\|_2^p \mid x\bigr] \leq \tilde\sigma_g^p.
\]
This reduces the analysis to the classical smoothness / bounded-noise
setting, with no Gr\"onwall argument needed. The only complication
relative to a deterministic-step analysis is that $w_t = (1-s_t)/K_t$
is now random.

\begin{lemma}[Hessian-Noise Bound, constrained]
\label{lem:app:hess-noise-b}
Consider RanSOM-B with $s_t \sim \mathrm{Beta}(1, K_t)$, $w_t =
(1-s_t)/K_t$, $K_t = \eta_t^{-1} - 1$, and update direction $d_t = v_t -
x_t$ satisfying $\|d_t\| \leq D$. Under
Assumptions~\ref{ass:app:smooth}--\ref{ass:app:noise}, and for any step
size $\eta_t \in (0, 1/2]$,
\begin{equation}
  \mathbb{E}\bigl[\|\Psi_{t+1}\|_2^q \,\big|\, \mathcal{F}_t\bigr]
  \;\leq\; C_{\delta,B}^q\, \eta_t^q\, \tilde\sigma_{h,\mathrm{eff}}^q,
  \label{eq:app:hess-noise-b}
\end{equation}
where
\begin{equation}
  \tilde\sigma_{h,\mathrm{eff}}^q
  \triangleq D^q\,\bigl(\tilde\sigma_h^q + \kappa^q L^q\bigr),
  \qquad
  C_{\delta,B}^q \triangleq 2^{2(q-1)}\cdot \max\!\bigl(M_w, M_{ws}\cdot 2^{q-1}\bigr),
  \label{eq:app:sigma-b}
\end{equation}
and $M_w, M_{ws}$ are the normalized moments of $(w_t, s_t)$ from
Definition~\ref{def:app:moments}.
\end{lemma}

\begin{proof}
Decompose $\Psi_{t+1} = T_1 + T_2$ as in~\eqref{eq:app:Psi-decomp}, and
split $\mathbb{E}[\|\Psi_{t+1}\|_2^q \mid \mathcal{F}_t] \leq
2^{q-1}(\mathbb{E}[\|T_1\|_2^q \mid \mathcal{F}_t] + \mathbb{E}[\|T_2\|_2^q
\mid \mathcal{F}_t])$.

\smallskip\noindent\textbf{Step 1 (Bounding $T_1$).}
Apply Assumption~\ref{ass:app:noise} with $\|d_t\| \leq D$, conditioning
on $(s_t, w_t)$:
\[
  \mathbb{E}_\xi\bigl[\|T_1\|_2^q \,\big|\, x_{t+1}, s_t, w_t\bigr]
  \leq |w_t|^q\, D^q\, \tilde\sigma_h^q,
\]
using the uniform bound $\|\nabla f(x_{t+1})\|_*^q \leq G_\mathcal{C}^q$
to absorb the affine term into $\tilde\sigma_h^q$. Take $\mathbb{E}[\,\cdot
\mid \mathcal{F}_t]$:
\begin{equation}
  \mathbb{E}[\|T_1\|_2^q \mid \mathcal{F}_t]
  \leq \mathbb{E}[|w_t|^q]\cdot D^q \tilde\sigma_h^q
  = M_w\, \eta_t^q\, D^q\, \tilde\sigma_h^q.
  \label{eq:app:T1-b}
\end{equation}

\smallskip\noindent\textbf{Step 2 (Bounding $T_2$).}
On $\mathcal{C}$, $\|\nabla^2 f(x)\|_{\mathrm{op}} \leq L$ uniformly, so
from~\eqref{eq:app:T2-norm} (with $\rho$ replaced by $D$ since $\|d_t\|
\leq D$):
\[
  \|T_2\|_2
  \leq \kappa D\,\Bigl(|w_t|\, L + s_t \int_0^1 L\, d\tau\Bigr)
  = \kappa D L\, (|w_t| + s_t).
\]
Raising to the $q$-th power and taking $\mathbb{E}[\,\cdot \mid
\mathcal{F}_t]$:
\begin{equation}
  \mathbb{E}[\|T_2\|_2^q \mid \mathcal{F}_t]
  \leq \kappa^q D^q L^q\, \mathbb{E}[(|w_t|+s_t)^q]
  = \kappa^q D^q L^q\, M_{ws}\, \eta_t^q.
  \label{eq:app:T2-b}
\end{equation}

\smallskip\noindent\textbf{Step 3 (Assembly).}
Combining,
\begin{align*}
  \mathbb{E}[\|\Psi_{t+1}\|_2^q \mid \mathcal{F}_t]
  &\leq 2^{q-1}\bigl(M_w\, \eta_t^q D^q \tilde\sigma_h^q
    + \kappa^q D^q L^q M_{ws}\, \eta_t^q\bigr)
  \\
  &\leq 2^{q-1}\, D^q\, \eta_t^q\, \max\bigl(M_w, M_{ws}\bigr)\,
    \bigl(\tilde\sigma_h^q + \kappa^q L^q\bigr)
  \\
  &\leq C_{\delta,B}^q\, \eta_t^q\, D^q\,
    \bigl(\tilde\sigma_h^q + \kappa^q L^q\bigr),
\end{align*}
using the definition of $C_{\delta,B}$ in~\eqref{eq:app:sigma-b}.
This is~\eqref{eq:app:hess-noise-b}.
\end{proof}

\begin{remark}
At $L_1 = \alpha_h = 0$ we have $L = L_0$ and $\tilde\sigma_h = \sigma_h$.
The oracle term is bounded directly in $\|\cdot\|_2$ and carries no
$\kappa$, while the path term carries $\kappa$ from the norm conversion;
the combined scale is therefore of order $D(\sigma_h + \kappa L_0) =
\bar\sigma_h$, the constant of Theorem~\ref{thm:betasom_convergence}.
The quantity $\kappa D(\sigma_h + L_0)$ is \emph{not} equivalent to
$D(\sigma_h + \kappa L_0)$ up to a constant independent of $\kappa$; the
placement of $\kappa$ on the path term only is the correct one.
\end{remark}


\subsection{Auxiliary Martingale Lemmas}
\label{app:martingale}

\begin{lemma}[{von Bahr-Esseen inequality; \citealp{vonbahr1965}; \citealp[Theorem~3.1]{pinelis2022multidimensional}}]
\label{lem:app:vbe}
Let $\{X_j\}_{j=1}^t$ be a martingale difference sequence in a Hilbert
space with respect to a filtration $\{\mathcal{H}_j\}$. For every
$r \in (1, 2]$, there is a constant $C_r \leq 2$ such that
\begin{equation}
  \mathbb{E}\!\left[\Bigl\|\textstyle\sum_{j=1}^t X_j\Bigr\|_2^r\right]
  \leq C_r \sum_{j=1}^t \mathbb{E}\bigl[\|X_j\|_2^r\bigr].
  \label{eq:app:vbe}
\end{equation}
\end{lemma}
\citet{vonbahr1965} prove~\eqref{eq:app:vbe} with $C_r = 2$ for real-valued random variables satisfying $\mathbb{E}[X_{j+1} \mid \sum_{i \le j} X_i] = 0$. \citet[Theorem~3.1]{pinelis2022multidimensional} extends the inequality, with the best constant $C_r \in [1, 2)$, to random vectors in a separable Hilbert space under the same condition, which holds for martingale differences with respect to any filtration, since $\sum_{i \le j} X_i$ is $\mathcal{H}_j$-measurable.

\begin{lemma}[Affine noise peeling]
\label{lem:app:peeling}
Let $\{c_j\}_{j=1}^t$ be non-negative deterministic scalars and
$\{\xi_j\}_{j=1}^t$ a martingale difference sequence satisfying
\[
  \mathbb{E}\bigl[\|\xi_j\|_2^p \,\big|\, \mathcal{H}_{j-1}\bigr]
  \leq \sigma^p + \alpha^p Y_j^p,
\]
where $Y_j \geq 0$ is $\mathcal{H}_{j-1}$-measurable and $p \in (1, 2]$.
Then
\begin{equation}
  \mathbb{E}\!\left[\Bigl(\textstyle\sum_{j=1}^t c_j^p \|\xi_j\|_2^p\Bigr)^{1/p}\right]
  \leq \sigma\,\Bigl(\textstyle\sum_{j=1}^t c_j^p\Bigr)^{1/p}
    + \alpha \sum_{j=1}^t c_j\, \mathbb{E}[Y_j].
  \label{eq:app:peeling}
\end{equation}
\end{lemma}

\begin{proof}
Let $S_k = \sum_{j=1}^k c_j^p \|\xi_j\|_2^p$. Apply Jensen's inequality
conditionally with the concave map $x \mapsto x^{1/p}$:
\begin{align*}
  \mathbb{E}[S_t^{1/p} \mid \mathcal{H}_{t-1}]
  &\leq \bigl(S_{t-1} + c_t^p\, \mathbb{E}[\|\xi_t\|_2^p \mid \mathcal{H}_{t-1}]\bigr)^{1/p}
  \leq \bigl(S_{t-1} + c_t^p\sigma^p + c_t^p\alpha^p Y_t^p\bigr)^{1/p}
  \\
  &\leq \bigl(S_{t-1} + c_t^p\sigma^p\bigr)^{1/p} + c_t\alpha Y_t,
\end{align*}
where the last step uses $(x+y)^{1/p} \leq x^{1/p} + y^{1/p}$ for $x, y
\geq 0$ and $p \geq 1$ (subadditivity of concave functions vanishing at
$0$). Take expectation and iterate from $j = t$ down to $j = 1$.
\end{proof}

The next lemma is the one actually used for the momentum error. It
bounds the first moment of a weighted martingale sum directly, without
passing a random affine factor through a moment root; the point is that
$(\mathbb{E} Y^p)^{1/p}$ is in general larger than $\mathbb{E}Y$, so a
bound on $\mathbb{E}\|\sum_j c_j\xi_j\|_2^p$ obtained from von
Bahr-Esseen cannot be converted into a bound linear in
$\mathbb{E}Y_j$ by taking a $p$-th root.

\begin{lemma}[Weighted martingale sum with affine conditional moments]
\label{lem:app:affine-mds}
Let $\{\mathcal{H}_j\}_{j\ge 0}$ be a filtration, $\{c_j\}_{j=1}^t$
non-negative deterministic scalars, and $\{\xi_j\}_{j=1}^t$
$\mathbb{R}^d$-valued, with $\xi_j$ $\mathcal{H}_j$-measurable,
$\mathbb{E}[\xi_j \mid \mathcal{H}_{j-1}] = 0$, and
\[
  \mathbb{E}\bigl[\|\xi_j\|_2^p \,\big|\, \mathcal{H}_{j-1}\bigr]
  \leq \sigma^p + \alpha^p Y_j^p,
\]
where $Y_j \geq 0$ is $\mathcal{H}_{j-1}$-measurable, $\sigma, \alpha
\ge 0$, and $p \in (1, 2]$. Then
\begin{equation}
  \mathbb{E}\Bigl\|\sum_{j=1}^t c_j \xi_j\Bigr\|_2
  \;\leq\; C_p^{1/p}\,\sigma\,\Bigl(\sum_{j=1}^t c_j^p\Bigr)^{1/p}
    + \alpha \sum_{j=1}^t c_j\, \mathbb{E}[Y_j],
  \label{eq:app:affine-mds}
\end{equation}
where $C_p \le 2$ is the von Bahr-Esseen constant of
Lemma~\ref{lem:app:vbe}.
\end{lemma}

\begin{proof}
Set $Z_j = \xi_j/(\sigma + \alpha Y_j)$ when $\sigma + \alpha Y_j > 0$,
and $Z_j = 0$ otherwise (in which case $\xi_j = 0$ a.s.). Since $\sigma +
\alpha Y_j$ is $\mathcal{H}_{j-1}$-measurable and positive on the event
where $Z_j \ne 0$, $\mathbb{E}[Z_j \mid \mathcal{H}_{j-1}] = 0$ and
\[
  \mathbb{E}\bigl[\|Z_j\|_2^p \,\big|\, \mathcal{H}_{j-1}\bigr]
  = \frac{\mathbb{E}[\|\xi_j\|_2^p \mid \mathcal{H}_{j-1}]}{(\sigma + \alpha Y_j)^p}
  \le \frac{\sigma^p + \alpha^p Y_j^p}{(\sigma + \alpha Y_j)^p} \le 1 .
\]
Decompose $\sum_j c_j\xi_j = \sigma\sum_j c_j Z_j + \alpha\sum_j c_j Y_j
Z_j$. For the first sum, $\{c_j Z_j\}$ is a martingale difference
sequence with $\mathbb{E}\|c_jZ_j\|_2^p \le c_j^p$, so by Jensen and
Lemma~\ref{lem:app:vbe},
\[
  \mathbb{E}\Bigl\|\sum_j c_j Z_j\Bigr\|_2
  \le \Bigl(\mathbb{E}\Bigl\|\sum_j c_j Z_j\Bigr\|_2^p\Bigr)^{1/p}
  \le C_p^{1/p}\Bigl(\sum_j c_j^p\Bigr)^{1/p}.
\]
For the second sum use the triangle inequality, then condition on
$\mathcal{H}_{j-1}$ (on which $Y_j$ is measurable) and apply the
conditional Jensen inequality $\mathbb{E}[\|Z_j\|_2 \mid \mathcal{H}_{j-1}]
\le (\mathbb{E}[\|Z_j\|_2^p \mid \mathcal{H}_{j-1}])^{1/p} \le 1$:
\[
  \mathbb{E}\Bigl\|\sum_j c_j Y_j Z_j\Bigr\|_2
  \le \sum_j c_j\, \mathbb{E}\bigl[Y_j\, \mathbb{E}[\|Z_j\|_2 \mid \mathcal{H}_{j-1}]\bigr]
  \le \sum_j c_j\, \mathbb{E}[Y_j].
\]
Combining the two bounds gives~\eqref{eq:app:affine-mds}.
\end{proof}

\begin{lemma}[Geometric sum bound]
\label{lem:app:geom}
For $\beta \in (0, 1/2]$ and $r \in (1, 2]$,
\[
  \sum_{k=0}^\infty (1-\beta)^{rk}
  = \frac{1}{1 - (1-\beta)^r}
  \leq \frac{1}{r\beta(1-\beta)^{r-1}}
  \leq \frac{2^{r-1}}{r\beta}
  \leq \frac{2}{r\beta}.
\]
Consequently, $\bigl(\sum_{k=0}^\infty (1-\beta)^{rk}\bigr)^{1/r} \leq
2^{1/r}(r\beta)^{-1/r} \leq 2(r\beta)^{-1/r}$.
\end{lemma}

\begin{proof}
The mean-value inequality applied to $h(x) = 1 - (1-x)^r$ gives $h(x) -
h(0) = h'(c) x$ for some $c \in [0, x]$, i.e. $1 - (1-\beta)^r =
r(1-c)^{r-1}\beta \geq r(1-\beta)^{r-1}\beta$ since $c \leq \beta$ and
$r-1 \geq 0$. For $\beta \leq 1/2$, $(1-\beta)^{r-1} \geq (1/2)^{r-1} =
2^{-(r-1)}$.
\end{proof}


\subsection{Descent Lemma (Unconstrained)}
\label{app:descent-e}

\begin{lemma}[Descent Inequality for RanSOM-E]
\label{lem:app:descent}
Under Assumption~\ref{ass:app:smooth}, the stepsize
condition~\eqref{eq:app:stepsize-mgf}, and
$\eta_t \leq 1/(2 C_s \rho L_1)$, the RanSOM-E update satisfies
\begin{equation}
  \mathbb{E}[f(x_{t+1}) \mid \mathcal{F}_t]
  \leq f(x_t)
    - \tfrac{\rho \eta_t}{2}\, \|\nabla f(x_t)\|_*
    + 2\rho \kappa \eta_t\, \|e_t\|_2
    + \tfrac{C_s \rho^2 L_0}{2}\, \eta_t^2\cdot C_{\mathrm{curv}},
  \label{eq:app:descent}
\end{equation}
where $C_{\mathrm{curv}} = 4/3$ under the stated stepsize condition
(the proof produces the factor $\tfrac{1}{1-u}$ with $u = L_1\rho\eta_t
\le 1/4$, which is bounded by $4/3$ with equality only at $u = 1/4$). It can be
absorbed into $C_s$, giving the main-text form $\tfrac{C_s\rho^2
L_0}{2}\eta_t^2$ with a redefined $C_s = 8/3$.
\end{lemma}

\begin{proof}
Write $a = \|\nabla f(x_t)\|_*$ and $u = L_1\rho\eta_t$. The stepsize
condition $\eta_t \le 1/(2C_s\rho L_1)$ with $C_s = 2$ reads $u \le 1/4$.
Parameterize the ray by $x(z) = x_t + z d_t$ for $z \in [0,s_t]$, so
that $x_{t+1} = x(s_t)$ ($z$ is arclength only when $\|d_t\| = 1$).

\smallskip\noindent\textbf{Step 1 (Exact Gr\"onwall bound along the ray).}
Let $\varphi(z) = \|\nabla f(x(z))\|_*$. As in the proof of
Lemma~\ref{lem:app:hess-noise}, $\varphi'(z) \le \rho\,(L_0 + L_1\varphi(z))$,
and solving this linear inequality exactly (rather than bounding
$L_0$ crudely) gives, for $L_1 > 0$,
\[
  \varphi(z) \;\le\; e^{L_1\rho z}\,a + \frac{L_0}{L_1}\bigl(e^{L_1\rho z} - 1\bigr).
\]
Hence, by Assumption~\ref{ass:app:smooth},
\[
  \|\nabla^2 f(x(z))\|_{\mathrm{op}}
  \;\le\; L_0 + L_1\varphi(z)
  \;\le\; (L_0 + L_1 a)\, e^{L_1\rho z}.
\]
This bound has no separate $L_0\rho z$ term: the $L_0$ contribution is
already carried by the exponential, which is what allows the remainder
to be integrated exactly below.

\smallskip\noindent\textbf{Step 2 (Taylor remainder and expectation over $s_t$).}
Taylor's theorem along the ray with integral remainder gives
\[
  f(x_{t+1}) = f(x_t) + s_t\langle\nabla f(x_t), d_t\rangle
    + \int_0^{s_t} (s_t - z)\,\langle d_t, \nabla^2 f(x(z))\, d_t\rangle\, dz.
\]
Using $|\langle d_t, \nabla^2 f(x(z)) d_t\rangle| \le \|d_t\|^2
\|\nabla^2 f(x(z))\|_{\mathrm{op}} \le \rho^2 (L_0 + L_1 a) e^{L_1\rho z}$,
\[
  f(x_{t+1}) \le f(x_t) + s_t\langle\nabla f(x_t), d_t\rangle
    + \rho^2 (L_0 + L_1 a) \int_0^{s_t} (s_t - z)\, e^{L_1\rho z}\, dz.
\]
Take $\mathbb{E}[\,\cdot \mid \mathcal{F}_t]$. The first-order term gives
$\eta_t\langle\nabla f(x_t), d_t\rangle$ since $\mathbb{E}[s_t] = \eta_t$
and $d_t$ is $\mathcal{F}_t$-measurable. For the remainder, write
$\int_0^{s_t}(s_t - z)e^{L_1\rho z}dz = \int_0^\infty (s_t - z)_+
e^{L_1\rho z} dz$ and exchange expectation and integral (Tonelli, the
integrand is non-negative). For $s_t \sim \mathrm{Exp}(1/\eta_t)$ the
memoryless property gives $\mathbb{E}[(s_t - z)_+] = \eta_t e^{-z/\eta_t}$,
so
\[
  \mathbb{E}\!\left[\int_0^{s_t}(s_t - z)e^{L_1\rho z}dz\right]
  = \eta_t\int_0^\infty e^{-(1/\eta_t - L_1\rho) z}\, dz
  = \frac{\eta_t}{1/\eta_t - L_1\rho}
  = \frac{\eta_t^2}{1-u},
\]
which is finite because $u < 1$. Therefore
\begin{equation}
  \mathbb{E}[f(x_{t+1}) \mid \mathcal{F}_t]
  \le f(x_t) + \eta_t\langle\nabla f(x_t), d_t\rangle
    + \frac{\rho^2\eta_t^2}{1-u}\,(L_0 + L_1 a).
  \label{eq:app:desc-step2}
\end{equation}
Every term on the right is exact or an upper bound; no positive term has
been discarded.

\smallskip\noindent\textbf{Step 3 (Relate $\langle\nabla f(x_t), d_t\rangle$ to the LMO).}
By the LMO property $d_t = \mathrm{argmin}_{v : \|v\| \leq \rho}\langle
m_t, v\rangle$ and the definition of $e_t = m_t - \nabla f(x_t)$,
\begin{align*}
  \langle \nabla f(x_t), d_t\rangle
  &= \langle m_t, d_t\rangle - \langle e_t, d_t\rangle
  \leq -\rho\|m_t\|_* + \|e_t\|_*\cdot\|d_t\|
  \\
  &\leq -\rho\|\nabla f(x_t)\|_* + \rho\|e_t\|_* + \rho\|e_t\|_*
  = -\rho\|\nabla f(x_t)\|_* + 2\rho\|e_t\|_*,
\end{align*}
where the second line uses $\|m_t\|_* \geq \|\nabla f(x_t)\|_* -
\|e_t\|_*$ (reverse triangle) and $\|d_t\| \leq \rho$. Applying
$\|e_t\|_* \leq \kappa\|e_t\|_2$ (Assumption~\ref{ass:app:kappa}),
\[
  \langle \nabla f(x_t), d_t\rangle
  \leq -\rho\|\nabla f(x_t)\|_* + 2\rho\kappa\|e_t\|_2.
\]

\smallskip\noindent\textbf{Step 4 (Assembly).}
Substitute the LMO bound into~\eqref{eq:app:desc-step2}, recalling $a =
\|\nabla f(x_t)\|_*$:
\begin{align*}
  \mathbb{E}[f(x_{t+1}) \mid \mathcal{F}_t]
  &\leq f(x_t) - \rho\eta_t\, a + 2\rho\kappa\eta_t\|e_t\|_2
    + \frac{\rho^2\eta_t^2 L_0}{1-u} + \frac{\rho^2\eta_t^2 L_1}{1-u}\, a .
\end{align*}
The last term is proportional to $a$ and is controlled by the negative
drift: since $\rho\eta_t L_1 = u$, its coefficient is $\rho\eta_t\cdot
u/(1-u)$, and $u \le 1/4$ gives $u/(1-u) \le 1/3 < 1/2$. Hence
$\frac{\rho^2\eta_t^2 L_1}{1-u}a \le \frac{\rho\eta_t}{2}a$, and
\[
  \mathbb{E}[f(x_{t+1}) \mid \mathcal{F}_t]
  \leq f(x_t) - \frac{\rho\eta_t}{2}\, a
    + 2\rho\kappa\eta_t\|e_t\|_2 + \frac{\rho^2 L_0}{1-u}\,\eta_t^2 .
\]
With $u \le 1/4$, $\frac{1}{1-u} \le \frac{4}{3}$, so the quadratic
coefficient is at most $\frac{4}{3}\rho^2 L_0\eta_t^2 = \frac{C_s\rho^2
L_0}{2}\eta_t^2\cdot C_{\mathrm{curv}}$ with $C_s = 2$ and
$C_{\mathrm{curv}} = 4/3$. This is~\eqref{eq:app:descent} under the
stated condition $\eta_t \le 1/(2C_s\rho L_1)$; no strengthening of the
stepsize condition is needed.

When $L_1 = 0$, taking $u = 0$ yields
\[
  \mathbb{E}[f(x_{t+1}) \mid \mathcal{F}_t]
  \le f(x_t) - \rho\eta_t\|\nabla f(x_t)\|_* + 2\rho\kappa\eta_t\|e_t\|_2 + \rho^2 L_0\eta_t^2 .
\]
This holds for every $\eta_t > 0$ and is precisely Lemma~\ref{lemma:main_descent}.
Accordingly, Theorem~\ref{thm:expsom_convergence} imposes no upper cap on $\eta$.
\end{proof}


\subsection{Descent Lemma (Constrained)}
\label{app:descent-b}

\begin{lemma}[Descent Inequality for RanSOM-B]
\label{lem:app:descent-b}
Under Assumption~\ref{ass:app:smooth} with effective
smoothness~\eqref{eq:app:eff-const}, and $s_t \sim \mathrm{Beta}(1, K_t)$
with $K_t = \eta_t^{-1} - 1$,
\begin{equation}
  \mathbb{E}[f(x_{t+1}) \mid \mathcal{F}_t]
  \leq f(x_t) - \eta_t\, \mathcal{G}(x_t)
    + 2\kappa D\eta_t\, \|e_t\|_2 + \tfrac{C_s L D^2}{2}\, \eta_t^2,
  \label{eq:app:descent-b}
\end{equation}
where $\mathcal{G}(x) = \max_{v \in \mathcal{C}}\langle \nabla f(x), x -
v\rangle$ is the Frank-Wolfe gap, $C_s = \mathbb{E}[s_t^2]/\eta_t^2 \leq
2$, $L = L_0 + L_1 G_\mathcal{C}$, and $\kappa$ is the norm compatibility
constant of Assumption~\ref{ass:app:kappa}. The factor $\kappa$ is not
absorbed into $D$: $D$ is the diameter in the primal norm, and for a
rescaled norm the bound without $\kappa$ is false.
\end{lemma}

\begin{proof}
On $\mathcal{C}$, $\|\nabla^2 f(x)\|_{\mathrm{op}} \leq L$ uniformly, so
$f$ is $L$-smooth on $\mathcal{C}$ in the usual (Lipschitz-gradient)
sense. The standard descent inequality
\[
  f(x_{t+1}) \leq f(x_t) + \langle \nabla f(x_t), x_{t+1} - x_t\rangle
    + \tfrac{L}{2}\|x_{t+1} - x_t\|^2
\]
follows. With $x_{t+1} - x_t = s_t d_t$ and $\|d_t\| = \|v_t - x_t\|
\leq D$:
\[
  f(x_{t+1}) \leq f(x_t) + s_t\langle\nabla f(x_t), d_t\rangle
    + \tfrac{L D^2}{2}\, s_t^2.
\]
Taking $\mathbb{E}[\,\cdot \mid \mathcal{F}_t]$ with $\mathbb{E}[s_t] =
\eta_t$, $\mathbb{E}[s_t^2] \leq C_s\eta_t^2$:
\[
  \mathbb{E}[f(x_{t+1}) \mid \mathcal{F}_t]
  \leq f(x_t) + \eta_t\langle\nabla f(x_t), d_t\rangle
    + \tfrac{C_s L D^2}{2}\eta_t^2.
\]

\smallskip\noindent\textbf{Relating $\langle\nabla f(x_t), d_t\rangle$ to the FW gap.}
Let $v^\star(x_t) \in \mathrm{argmin}_{v \in \mathcal{C}}\langle\nabla
f(x_t), v\rangle$ and $v_t \in \mathrm{argmin}_{v \in
\mathcal{C}}\langle m_t, v\rangle$. Then $d_t = v_t - x_t$, so
\begin{align*}
  \langle\nabla f(x_t), d_t\rangle
  &= \langle\nabla f(x_t), v_t - x_t\rangle
  = \langle m_t, v_t - x_t\rangle + \langle e_t, x_t - v_t\rangle
  \\
  &\leq \langle m_t, v^\star(x_t) - x_t\rangle + \|e_t\|_*\cdot \|x_t - v_t\|
  \qquad \text{(LMO optimality of } v_t\text{)}
  \\
  &= \langle\nabla f(x_t), v^\star(x_t) - x_t\rangle
    + \langle e_t, v^\star(x_t) - x_t\rangle + \|e_t\|_*\cdot D
  \\
  &\leq -\mathcal{G}(x_t) + \|e_t\|_*\cdot D + \|e_t\|_*\cdot D
  = -\mathcal{G}(x_t) + 2D\|e_t\|_*,
\end{align*}
where the second-to-last step uses $\langle m_t, v^\star - x_t\rangle =
\langle\nabla f(x_t), v^\star - x_t\rangle + \langle e_t, v^\star -
x_t\rangle$ (since $m_t = \nabla f(x_t) + e_t$) and bounds the $e_t$ term
by $\|e_t\|_*\|v^\star - x_t\| \le \|e_t\|_* D$. Converting the dual norm
to the Euclidean norm with $\|e_t\|_* \leq \kappa\|e_t\|_2$
(Assumption~\ref{ass:app:kappa}),
\[
  \langle\nabla f(x_t), d_t\rangle \leq -\mathcal{G}(x_t) + 2\kappa D\|e_t\|_2.
\]
Substituting gives~\eqref{eq:app:descent-b}.
\end{proof}


\subsection{Momentum Error Recursion}
\label{app:error-recursion}

We now bound the momentum error $e_t = m_t - \nabla f(x_t)$ for both
algorithms. The derivation is identical up to the specific form of the
Hessian-noise bound used (Lemma~\ref{lem:app:hess-noise} for RanSOM-E,
Lemma~\ref{lem:app:hess-noise-b} for RanSOM-B); we present it for
RanSOM-E and indicate the constrained variant at the end.

\smallskip\noindent\textbf{Unrolling.}
The momentum update is $m_{t+1} = (1-\beta)(m_t + \delta_{t+1}) + \beta
g_{t+1}$. Subtract $\nabla f(x_{t+1})$ from both sides and use the
definitions $e_{t+1} = m_{t+1} - \nabla f(x_{t+1})$,
$\bar\xi_{t+1} = g_{t+1} - \nabla f(x_{t+1})$, and
$\bar\Psi_{t+1} = \delta_{t+1} - (\nabla f(x_{t+1}) - \nabla f(x_t))$:
\begin{align*}
  e_{t+1}
  &= (1-\beta)(m_t + \delta_{t+1}) + \beta g_{t+1} - \nabla f(x_{t+1})
  \\
  &= (1-\beta)\bigl(m_t + (\nabla f(x_{t+1}) - \nabla f(x_t)) + \bar\Psi_{t+1}\bigr)
    + \beta\bigl(\nabla f(x_{t+1}) + \bar\xi_{t+1}\bigr) - \nabla f(x_{t+1})
  \\
  &= (1-\beta)\bigl(m_t - \nabla f(x_t)\bigr) + (1-\beta)\bar\Psi_{t+1}
    + \beta \bar\xi_{t+1}
  \\
  &= (1-\beta) e_t + \beta \bar\xi_{t+1} + (1-\beta)\bar\Psi_{t+1}.
\end{align*}
Unrolling from step $0$,
\begin{equation}
  e_{t+1} = (1-\beta)^{t+1} e_0
    + \underbrace{\sum_{j=1}^{t+1} \beta(1-\beta)^{t+1-j}\, \bar\xi_j}_{M_{\mathrm{grad}}}
    + \underbrace{\sum_{j=1}^{t+1} (1-\beta)^{t+2-j}\, \bar\Psi_j}_{M_{\mathrm{corr}}}.
  \label{eq:app:e-unrolled}
\end{equation}

\begin{lemma}[Error Bound, RanSOM-E]
\label{lem:app:error}
Under Assumptions~\ref{ass:app:smooth}--\ref{ass:app:noise}, the stepsize
condition~\eqref{eq:app:stepsize-mgf}, a stepsize cap $\eta \le \bar\eta$,
constant batch size $B \ge 1$ and $\beta \in (0, 1/2]$, the RanSOM-E
momentum error satisfies
\begin{align}
  \mathbb{E}\|e_{t+1}\|_2
  &\leq (1-\beta)^{t+1}\mathbb{E}\|e_0\|_2
    + \frac{8\beta^{1-1/p}}{p^{1/p} B^{(p-1)/p}}\sigma_g
    + \frac{2\alpha_g}{B^{(p-1)/p}}
      \sum_{j=1}^{t+1}\beta(1-\beta)^{t+1-j}\mathbb{E}\|\nabla f(x_j)\|_*
  \notag \\
  &\quad + \frac{8\,\eta\beta^{-1/q}}{q^{1/q}}\Bigl(\frac{\rho\,\sigma_h}{B^{(q-1)/q}} + \bar\sigma_{\mathrm{path}}\Bigr)
    + \frac{2\eta\rho\alpha_h}{B^{(q-1)/q}}\sum_{j=1}^{t+1}(1-\beta)^{t+2-j}\mathbb{E}\|\nabla f(x_j)\|_*
  \notag \\
  &\quad + \eta\,\bar\alpha_{\mathrm{path}}\sum_{j=1}^{t+1}(1-\beta)^{t+2-j}\,\mathbb{E}\|\nabla f(x_{j-1})\|_*,
  \label{eq:app:error-bound}
\end{align}
where the path-residual constants are
\begin{equation}
  \bar\sigma_{\mathrm{path}}^q \triangleq \kappa^q\rho^q\, 2^{q-1}\, L_0^q\,\widetilde C_{A,q},
  \qquad
  \bar\alpha_{\mathrm{path}}^q \triangleq \kappa^q\rho^q\, 2^{q-1}\, L_1^q\,\widetilde C_{B,q},
  \label{eq:app:path-consts}
\end{equation}
with $\widetilde C_{A,q}, \widetilde C_{B,q}$ as in
Definition~\ref{def:app:moments}. These involve only the true Hessian
through $L_0, L_1$ and contain no $\alpha_h$: the oracle affine factor is
handled at its own anchor $x_j$ in the oracle term, so the
$\alpha_h$-dependent baseline term of Lemma~\ref{lem:app:hess-noise}
(which arises there only from transferring that factor to $x_{j-1}$) does
not appear. The oracle-noise terms ($\sigma_g$,
$\alpha_g$, $\sigma_h$, $\alpha_h$) contract with the batch size; the
path-residual terms ($\bar\sigma_{\mathrm{path}}$,
$\bar\alpha_{\mathrm{path}}$) do not, because the random step $s_{j-1}$
that produces them is shared by every sample in the batch.
\end{lemma}

\begin{proof}
Take $\|\cdot\|_2$ of~\eqref{eq:app:e-unrolled} and apply the triangle
inequality:
\[
  \mathbb{E}\|e_{t+1}\|_2 \leq (1-\beta)^{t+1}\,\mathbb{E}\|e_0\|_2
    + \mathbb{E}\|M_{\mathrm{grad}}\|_2 + \mathbb{E}\|M_{\mathrm{corr}}\|_2.
\]
We keep the initialization under an expectation, since $e_0$ is random
when $m_0$ is a stochastic estimate. Throughout, $\mathcal{H}_{j-1} =
\sigma(\mathcal{F}_{j-1}, s_{j-1})$ is the intermediate $\sigma$-field
of Section~\ref{app:notation}; recall that $x_j$ is
$\mathcal{H}_{j-1}$-measurable but not $\mathcal{F}_{j-1}$-measurable.

\smallskip\noindent\textbf{Bounding $\mathbb{E}\|M_{\mathrm{grad}}\|_2$.}
The batched gradient noise $\bar\xi_j = \tfrac1B\sum_{i\in\mathcal{B}_j}
(\nabla f_\xi^{(i)}(x_j) - \nabla f(x_j))$ is centered conditional on
$\mathcal{H}_{j-1}$ (the oracle noise is centered given $x_j$, which is
$\mathcal{H}_{j-1}$-measurable). Since $\mathcal{H}_{j-1} \subseteq
\mathcal{F}_j \subseteq \mathcal{H}_j$, the sequence $\{\bar\xi_j\}$ is a
martingale difference sequence with respect to $\{\mathcal{H}_{j}\}$.
Conditional on $\mathcal{H}_{j-1}$ the $B$ summands are i.i.d.\ and
centered, so by von Bahr-Esseen (Lemma~\ref{lem:app:vbe}, applied to the
within-batch sum)
\[
  \mathbb{E}[\|\bar\xi_j\|_2^p \mid \mathcal{H}_{j-1}]
  \leq \frac{C_p}{B^{p-1}}\bigl(\sigma_g^p + \alpha_g^p\|\nabla f(x_j)\|_*^p\bigr),
  \qquad C_p \le 2 .
\]
Apply Lemma~\ref{lem:app:affine-mds} with filtration $\mathcal{H}$,
coefficients $c_j = \beta(1-\beta)^{t+1-j}$, $\sigma^p =
C_p\sigma_g^p/B^{p-1}$, $\alpha^p = C_p\alpha_g^p/B^{p-1}$ and $Y_j =
\|\nabla f(x_j)\|_*$:
\[
  \mathbb{E}\|M_{\mathrm{grad}}\|_2
  \leq \frac{C_p^{2/p}\,\sigma_g}{B^{(p-1)/p}}\Bigl(\sum_j c_j^p\Bigr)^{1/p}
    + \frac{C_p^{1/p}\,\alpha_g}{B^{(p-1)/p}}\sum_j c_j\, \mathbb{E}\|\nabla f(x_j)\|_*.
\]
By Lemma~\ref{lem:app:geom} with $r = p$, $\sum_j c_j^p \leq
\beta^p\sum_{k\ge0}(1-\beta)^{pk} \leq 2\beta^{p-1}/p$, so
$(\sum_j c_j^p)^{1/p} \leq 2^{1/p}\beta^{(p-1)/p}p^{-1/p}$. Since $C_p
\le 2$ and $p \in (1,2]$, $C_p^{2/p}\cdot 2^{1/p} \le 2^{3/p} \le 8$ and
$C_p^{1/p} \le 2^{1/p} \le 2$ (the bound $2^{2/p} \le 2$ used previously
is false for $p < 2$; the uniform bounds are $2^{2/p} \le 4$ and
$2^{3/p} \le 8$). Hence
\[
  \mathbb{E}\|M_{\mathrm{grad}}\|_2
  \leq \frac{8\,\beta^{1-1/p}}{p^{1/p} B^{(p-1)/p}}\, \sigma_g
    + \frac{2\,\alpha_g}{B^{(p-1)/p}}\sum_j \beta(1-\beta)^{t+1-j}\,\mathbb{E}\|\nabla f(x_j)\|_*.
\]
We retain the factor $\beta$ in $c_j$: its
convolution sum $\sum_j \beta(1-\beta)^{t+1-j} \le 1$ is what keeps the
affine gradient term $\mathcal{O}(\alpha_g B^{-(p-1)/p} G_{\mathrm{avg}})$
after averaging in Lemma~\ref{lem:app:avg-error}.

\smallskip\noindent\textbf{Bounding $\mathbb{E}\|M_{\mathrm{corr}}\|_2$.}
Decompose the batched correction error as in~\eqref{eq:app:Psi-split},
$\bar\Psi_j = \bar T_{1,j} + T_{2,j}$, where
\[
  \bar T_{1,j} = w_{j-1}\Bigl[\tfrac1B\sum_{i\in\mathcal{B}_j} H^{(i)}_\xi(x_j) - \nabla^2 f(x_j)\Bigr] d_{j-1}
\]
is the batched \emph{oracle} noise and $T_{2,j} = w_{j-1}\nabla^2 f(x_j)
d_{j-1} - (\nabla f(x_j) - \nabla f(x_{j-1}))$ is the \emph{path
residual}. The path residual does not depend on the oracle samples at
all: it is a function of $(x_{j-1}, d_{j-1}, s_{j-1})$ alone, so it is
identical for every sample in the batch and is not reduced by averaging.
Accordingly we split $M_{\mathrm{corr}} = M_{\mathrm{orc}} +
M_{\mathrm{path}}$ with weights $d_j = (1-\beta)^{t+2-j}$ and bound the
two sums at different conditioning levels.

\emph{Oracle part.} Conditional on $\mathcal{H}_{j-1}$ (which fixes
$x_j$, $d_{j-1}$, $w_{j-1}$), the $B$ Hessian-vector noises are i.i.d.\
and centered, so by Assumption~\ref{ass:app:noise} and von Bahr-Esseen
within the batch, with $\|d_{j-1}\| \le \rho$ and $w_{j-1} = \eta$,
\[
  \mathbb{E}[\|\bar T_{1,j}\|_2^q \mid \mathcal{H}_{j-1}]
  \leq \frac{C_q\,\eta^q\rho^q}{B^{q-1}}\bigl(\sigma_h^q + \alpha_h^q\|\nabla f(x_j)\|_*^q\bigr).
\]
The affine factor is anchored at $x_j$, which is
$\mathcal{H}_{j-1}$-measurable, so no Gr\"onwall transfer is needed
here. Apply Lemma~\ref{lem:app:affine-mds} with filtration $\mathcal{H}$,
$Y_j = \|\nabla f(x_j)\|_*$, $\sigma^q = C_q\eta^q\rho^q\sigma_h^q/B^{q-1}$,
$\alpha^q = C_q\eta^q\rho^q\alpha_h^q/B^{q-1}$:
\[
  \mathbb{E}\|M_{\mathrm{orc}}\|_2
  \leq \frac{C_q^{2/q}\,\eta\rho\,\sigma_h}{B^{(q-1)/q}}\Bigl(\sum_j d_j^q\Bigr)^{1/q}
    + \frac{C_q^{1/q}\,\eta\rho\,\alpha_h}{B^{(q-1)/q}}\sum_j d_j\,\mathbb{E}\|\nabla f(x_j)\|_*.
\]

\emph{Path part.} The path residual $T_{2,j}$ is centered conditional
on $\mathcal{F}_{j-1}$ by the Stein identity~\eqref{eq:app:stein-momentum}
(this is precisely the statement $\mathbb{E}_{s_{j-1}}[w_{j-1}\nabla^2
f(x_j)d_{j-1}] = \mathbb{E}_{s_{j-1}}[\nabla f(x_j) - \nabla f(x_{j-1})]$),
and it is $\mathcal{F}_j$-measurable, so $\{T_{2,j}\}$ is a martingale
difference sequence with respect to $\{\mathcal{F}_j\}$. Its conditional
moment is exactly the $T_2$ bound~\eqref{eq:app:T2-final} of
Lemma~\ref{lem:app:hess-noise}, anchored at $x_{j-1}$ via the Gr\"onwall
estimate:
\[
  \mathbb{E}[\|T_{2,j}\|_2^q \mid \mathcal{F}_{j-1}]
  \leq \eta^q\bigl(\bar\sigma_{\mathrm{path}}^q + \bar\alpha_{\mathrm{path}}^q\,\|\nabla f(x_{j-1})\|_*^q\bigr),
\]
with $\bar\sigma_{\mathrm{path}}, \bar\alpha_{\mathrm{path}}$ as
in~\eqref{eq:app:path-consts}. Two features of this bound matter. There is
no factor $B^{-(q-1)}$, since $T_{2,j}$ is common to the batch. And there
is no $\alpha_h$: the path residual involves only the true Hessian. The
$\alpha_h L_0$ baseline term of Lemma~\ref{lem:app:hess-noise} is an
artifact of bounding oracle and path randomness at a single anchor; the
split used here avoids it.
Apply Lemma~\ref{lem:app:affine-mds} with filtration $\mathcal{F}$ and
$Y_j = \|\nabla f(x_{j-1})\|_*$ (which is $\mathcal{F}_{j-1}$-measurable):
\[
  \mathbb{E}\|M_{\mathrm{path}}\|_2
  \leq C_q^{1/q}\,\eta\,\bar\sigma_{\mathrm{path}}\Bigl(\sum_j d_j^q\Bigr)^{1/q}
    + \eta\,\bar\alpha_{\mathrm{path}}\sum_j d_j\,\mathbb{E}\|\nabla f(x_{j-1})\|_*.
\]

\emph{Constant sums.} By Lemma~\ref{lem:app:geom}, $\sum_j d_j^q =
(1-\beta)^q\sum_j (1-\beta)^{q(t+1-j)} \le 2/(q\beta)$, so $(\sum_j
d_j^q)^{1/q} \le 2^{1/q}(q\beta)^{-1/q}$. Combining the oracle and path
parts, and using $C_q^{2/q}2^{1/q} \le 2^{3/q} \le 8$ and
$C_q^{1/q} \le 2^{1/q} \le 2$,
\begin{align*}
  \mathbb{E}\|M_{\mathrm{corr}}\|_2
  &\leq \frac{8\,\eta\beta^{-1/q}}{q^{1/q}}\Bigl(\frac{\rho\sigma_h}{B^{(q-1)/q}} + \bar\sigma_{\mathrm{path}}\Bigr)
    + \frac{2\,\eta\rho\,\alpha_h}{B^{(q-1)/q}}\sum_j d_j\,\mathbb{E}\|\nabla f(x_j)\|_*
  \\ &\quad
    + \eta\,\bar\alpha_{\mathrm{path}}\sum_j d_j\,\mathbb{E}\|\nabla f(x_{j-1})\|_*.
\end{align*}
Combining with the gradient bound gives~\eqref{eq:app:error-bound}.
\end{proof}

\begin{remark}[Why the path residual is not batched]
\label{rem:app:no-batch}
The previous version of this lemma applied the batch contraction
$B^{-(q-1)}$ to the whole of $\bar\Psi_j$. This is incorrect: on the
deterministic quadratic $f(x) = x^2/2$ with exact oracles the path
residual equals $(\eta - s_{j-1})d_{j-1}$ for RanSOM-E, whose second
moment is $\eta^2$ for every $B$, so no bound that vanishes as $B \to
\infty$ can hold. Only the oracle noise averages over the batch. As a
consequence, for large batches the correction error is dominated by the
curvature term $\eta\rho\kappa L_0\beta^{-1/q}$ rather than by
$\sigma_h$, and increasing $B$ cannot remove it. Giving each sample its
own independent step would change the algorithm and its endpoint
structure; it is not a fix to the present analysis.
\end{remark}

\begin{lemma}[Averaged Error Bound]
\label{lem:app:avg-error}
Let $\eta, \beta$ be constant over $t = 0, \ldots, T-1$, and define
$e_{\mathrm{avg}} = \tfrac1T\sum_{t=0}^{T-1}\mathbb{E}\|e_t\|_2$ and
$G_{\mathrm{avg}} = \tfrac1T\sum_{t=0}^{T-1}\mathbb{E}\|\nabla f(x_t)\|_*$,
both over the $T$ iterates $x_0, \ldots, x_{T-1}$ that enter the
telescoped descent bound. Under the hypotheses of
Lemma~\ref{lem:app:error},
\begin{align}
  e_{\mathrm{avg}}
  &\leq \mathbb{E}\|e_0\|_2\,\min\Bigl\{1,\frac{1}{\beta T}\Bigr\}
    + \frac{8\beta^{1-1/p}\sigma_g}{p^{1/p} B^{(p-1)/p}}
    + \frac{2\alpha_g}{B^{(p-1)/p}}\, G_{\mathrm{avg}}
  \notag \\
  &\quad + \frac{8\eta\beta^{-1/q}}{q^{1/q}}\Bigl(\frac{\rho\sigma_h}{B^{(q-1)/q}} + \bar\sigma_{\mathrm{path}}\Bigr)
    + \frac{2\eta\rho\alpha_h}{\beta\, B^{(q-1)/q}}\, G_{\mathrm{avg}}
    + \frac{\eta\,\bar\alpha_{\mathrm{path}}}{\beta}\, G_{\mathrm{avg}}.
  \label{eq:app:avg-error}
\end{align}
Every gradient that appears on the right is one of $x_0, \ldots,
x_{T-1}$; no terminal iterate $x_T$ enters.
\end{lemma}

\begin{proof}
Write $\sum_{t=0}^{T-1}\mathbb{E}\|e_t\|_2 = \mathbb{E}\|e_0\|_2 +
\sum_{t=0}^{T-2}\mathbb{E}\|e_{t+1}\|_2$ and
apply~\eqref{eq:app:error-bound} to each $e_{t+1}$ with $t = 0, \ldots,
T-2$. In these bounds the oracle anchors are $x_j$ with $1 \le j \le t+1
\le T-1$ and the path anchors are $x_{j-1}$ with $0 \le j-1 \le T-2$, so
all anchors lie among $x_0, \ldots, x_{T-1}$. (Applying the bound also
for $t = T-1$ would introduce $x_T$, which is not controlled by the
descent telescope over $T$ steps; we do not do so.) We treat the terms in
turn and divide by $T$.

\emph{Initialization.} $\mathbb{E}\|e_0\|_2 +
\sum_{t=0}^{T-2}(1-\beta)^{t+1}\mathbb{E}\|e_0\|_2 =
\mathbb{E}\|e_0\|_2\sum_{k=0}^{T-1}(1-\beta)^k \le
\mathbb{E}\|e_0\|_2\min\{T, 1/\beta\}$, since the geometric sum is at
most $T$ termwise and at most $1/\beta$ in total. Dividing by $T$ gives
the first term. The factor $\min\{1, 1/(\beta T)\}$ matters when $\beta T
< 1$: the initialization error is then not amplified, only carried.

\emph{Constant noise terms} are $t$-independent and pass through the
average unchanged.

\emph{Gradient affine sum.} With the factor $\beta$ retained and the
range $0 \le t \le T-2$, $1 \le j \le t+1$,
\begin{multline*}
  \sum_{t=0}^{T-2}\sum_{j=1}^{t+1}\beta(1-\beta)^{t+1-j}\,\mathbb{E}\|\nabla f(x_j)\|_*
  = \sum_{j=1}^{T-1}\mathbb{E}\|\nabla f(x_j)\|_*\;\beta\!\!\sum_{t=j-1}^{T-2}(1-\beta)^{t+1-j}
  \\
  \le \sum_{j=1}^{T-1}\mathbb{E}\|\nabla f(x_j)\|_*
  \le T\,G_{\mathrm{avg}},
\end{multline*}
since $\beta\sum_{k\ge0}(1-\beta)^k = 1$ and the sum runs over a subset
of $x_0, \ldots, x_{T-1}$. Dividing by $T$ and multiplying by
$2\alpha_g/B^{(p-1)/p}$ gives the third term. (Had the $\beta$ been
dropped from the weights, the inner sum would be $\le 1/\beta$ and the
term would carry an extra $1/\beta$.)

\emph{Hessian oracle affine sum.} The weights $d_j = (1-\beta)^{t+2-j}$
carry no factor $\beta$, so the same exchange gives $\sum_t\sum_j
d_j\,\mathbb{E}\|\nabla f(x_j)\|_* \le \tfrac1\beta\sum_{j=1}^{T-1}
\mathbb{E}\|\nabla f(x_j)\|_* \le \tfrac{T}{\beta}G_{\mathrm{avg}}$;
after dividing by $T$ and multiplying by $2\eta\rho\alpha_h/B^{(q-1)/q}$
this is the fifth term.

\emph{Path affine sum.} Identical with anchors $x_{j-1}$, $0 \le j-1 \le
T-2$: $\sum_t\sum_j d_j\,\mathbb{E}\|\nabla f(x_{j-1})\|_* \le
\tfrac1\beta\sum_{k=0}^{T-2}\mathbb{E}\|\nabla f(x_k)\|_* \le
\tfrac{T}{\beta}G_{\mathrm{avg}}$, giving the last term.
\end{proof}

\begin{remark}[Constrained analogue]
\label{rem:app:avg-error-b}
For RanSOM-B, Lemmas~\ref{lem:app:error} and~\ref{lem:app:avg-error}
hold with the following replacements. The gradient noise uses
$\tilde\sigma_g$ and $\alpha_g \to 0$. For the correction error, write
the batched oracle residual as
\[
  \zeta_j \;\triangleq\; w_{j-1}\Bigl[\tfrac1B\sum_{i\in\mathcal{B}_j} H^{(i)}_\xi(x_j) - \nabla^2 f(x_j)\Bigr] d_{j-1},
\]
and the Beta weight $w_{j-1} = (1-s_{j-1})/K$ is a function of
$s_{j-1}$, hence $\mathcal{H}_{j-1}$-measurable. Conditioning on
$\mathcal{H}_{j-1}$ therefore leaves $|w_{j-1}|^q$ as a \emph{realized}
random variable, which cannot be replaced by its mean $M_w\eta^q$. We
instead condition on $\mathcal{F}_{j-1}$: by the tower property
$\mathbb{E}[\zeta_j \mid \mathcal{F}_{j-1}] = \mathbb{E}[\mathbb{E}[\zeta_j
\mid \mathcal{H}_{j-1}] \mid \mathcal{F}_{j-1}] = 0$, so $\{\zeta_j\}$ is
a martingale difference sequence for $\{\mathcal{F}_j\}$. On the compact
set the Hessian-noise bound is uniform in the endpoint, $\mathbb{E}_\xi
\|[\tfrac1B\sum_i H^{(i)}_\xi(x) - \nabla^2 f(x)]d\|_2^q \le C_q
D^q\tilde\sigma_h^q/B^{q-1}$ for every $x \in \mathcal{C}$, so
integrating $|w_{j-1}|^q$ over the step (Definition~\ref{def:app:moments})
gives the deterministic conditional bound
\[
  \mathbb{E}[\|\zeta_j\|_2^q \mid \mathcal{F}_{j-1}]
  \le \frac{C_q\, M_w\,\eta^q D^q\tilde\sigma_h^q}{B^{q-1}} .
\]
Lemma~\ref{lem:app:affine-mds} then applies with $\alpha = 0$ and the
$\mathcal{F}$-filtration. The unbatched path residual is bounded as in
Lemma~\ref{lem:app:hess-noise-b}, Step~2, by $\kappa^q D^q L^q
M_{ws}\eta^q$ given $\mathcal{F}_{j-1}$; it is also centered given
$\mathcal{F}_{j-1}$ by the Stein identity. Both affine coefficients
vanish because the gradient is bounded by $G_\mathcal{C}$ on
$\mathcal{C}$ and this bound is absorbed into $\tilde\sigma_g,
\tilde\sigma_h$ via~\eqref{eq:app:eff-const}. The path residual is again
not reduced by the batch size. (Alternatively one may keep the
$\mathcal{H}$-conditioning and use the pointwise bound $|w_{j-1}| \le
\eta/(1-\eta) \le 2\eta$ for $\eta \le 1/2$, replacing $M_w$ by $2^q$;
the exponents are unchanged.)
\end{remark}


\subsection{Comparison with the Lower Bound}
\label{app:lower-bound}

This section details the optimality statement after
Corollary~\ref{cor:oracle_complexity}. Theorem and definition numbers for
\citet{fradin2025tight} refer to arXiv:2512.18713v2. To avoid a clash
with our $L_0, L_1$, we write $\mathcal{L}_1$ and $\mathcal{L}_2$ for
the Lipschitz constants of the gradient and of the Hessian in their
function class.

\paragraph{Oracle model.} The lower bound of \citet[Theorem~3.3]{fradin2025tight}
is stated for stochastic oracles in the sense of \citet{arjevani2020second}:
a query at $x$ returns unbiased estimators of $\nabla F(x)$ and
$\nabla^2 F(x)$, which are not required to be derivatives of a common
sample function. Their hard instance uses a row-masked Hessian estimator
(their Definition~B.8), which is not symmetric and is therefore not the
Hessian of any function. Our analysis needs only unbiasedness and the
moment bounds of Assumption~\ref{ass:app:noise} (see the oracle-model
paragraph after that assumption), so all our upper bounds hold in their
oracle model, and the comparison below is made in that model. An
algorithm with access to a Hessian estimator $H_\xi(x)$ can form the
products $H_\xi(x)d_t$ used by RanSOM-E, and their moment assumption on
$\|H_\xi(x) - \nabla^2 F(x)\|_{\mathrm{op}}$ implies our directional bound,
since $\|(H_\xi(x) - \nabla^2 F(x))w\|_2 \le \|H_\xi(x) - \nabla^2 F(x)\|_{\mathrm{op}}\|w\|_2$.
Hence their instances satisfy our assumptions in the Euclidean geometry
($\kappa = 1$) with $L_0 = \mathcal{L}_1$, $L_1 = 0$, $\sigma_g = \sigma_1$,
$\sigma_h = \sigma_2$ and $\alpha_g = \alpha_h = 0$, for every value of
$\mathcal{L}_2$.

\paragraph{Which branch applies.} For
$0 < \varepsilon \le c_1\min\{\sqrt{\mathcal{L}_1\Delta},\, \mathcal{L}_2^{1/3}\Delta^{2/3}\}$,
their Theorem~3.3 gives a lower bound of order $\min\{B_1, B_2, B_3\}$ for
zero-respecting algorithms, with
\begin{align*}
  B_1 &= \frac{\mathcal{L}_1\Delta}{\varepsilon^2}\Bigl[1 + \Bigl(\frac{\sigma_1}{\varepsilon}\Bigr)^{\frac{p}{p-1}}\Bigr], \qquad
  B_2 = \frac{\mathcal{L}_2^{1/2}\Delta}{\varepsilon^{3/2}}\Bigl[1 + \Bigl(\frac{\sigma_1}{\varepsilon}\Bigr)^{\frac{p}{p-1}}\Bigr], \\
  B_3 &= \min\Bigl\{\frac{\mathcal{L}_1\Delta}{\varepsilon^2}, \frac{\mathcal{L}_2^{1/2}\Delta}{\varepsilon^{3/2}}\Bigr\}
  + \frac{\sigma_2\Delta}{\varepsilon^2} + \frac{\sigma_2\Delta}{\varepsilon^2}\Bigl(\frac{\sigma_1}{\varepsilon}\Bigr)^{\frac{p}{q(p-1)}}
  + \Bigl(\frac{\sigma_1}{\varepsilon}\Bigr)^{\frac{p}{p-1}} .
\end{align*}
Our class places no bound on the Hessian Lipschitz constant, so
for each $\varepsilon$ the hard instance may be taken with
$\mathcal{L}_2 \ge \mathcal{L}_1^2/\varepsilon$, which also satisfies the
condition $\varepsilon \le c_1\mathcal{L}_2^{1/3}\Delta^{2/3}$ once
$\varepsilon$ is small. Then $B_2 \ge B_1$, the inner minimum in $B_3$
equals $\mathcal{L}_1\Delta/\varepsilon^2$, and the bound becomes
$\min\{B_1, B_3\}$. For fixed $\sigma_1, \sigma_2 > 0$, $B_1 \ge B_3$ for
all sufficiently small $\varepsilon$: since $q > 1$, the exponent
$p/(q(p-1))$ is smaller than $p/(p-1)$, so the term
$(\mathcal{L}_1\Delta/\varepsilon^2)(\sigma_1/\varepsilon)^{p/(p-1)}$ of $B_1$
grows faster than every term of $B_3$ as $\varepsilon \to 0$ (the term
$(\sigma_1/\varepsilon)^{p/(p-1)}$ is dominated because
$\mathcal{L}_1\Delta/\varepsilon^2 \ge c_1^{-2}$). Hence, for small
$\varepsilon$, the lower bound on our class is of order
\[
  \Bigl(\frac{\sigma_1}{\varepsilon}\Bigr)^{\frac{p}{p-1}}
  + \frac{(\mathcal{L}_1 + \sigma_2)\Delta}{\varepsilon^2}
  + \frac{\sigma_2\Delta}{\varepsilon^2}\Bigl(\frac{\sigma_1}{\varepsilon}\Bigr)^{\frac{p}{q(p-1)}} ,
\]
a target-dependent worst case over Hessian-Lipschitz constants, not a
statement for a fixed $\mathcal{L}_2$.

\paragraph{Exponents.} In the Euclidean geometry,
Corollary~\ref{cor:oracle_complexity} gives
$N = \mathcal{O}\bigl(1 + (\sigma_g/\varepsilon)^{p/(p-1)}
+ \Delta_0(L_0 + h)\varepsilon^{-2}[1 + (\sigma_g/\varepsilon)^{p/(q(p-1))}]\bigr)$
with $h = \sigma_h + L_0$. Every term of the lower bound appears with the
same exponent of $\varepsilon$, and the total exponent
$\max\{p/(p-1),\, 2 + p/(q(p-1))\}$ is the same for both, so RanSOM-E
attains the optimal exponents for every $(p,q) \in (1,2]^2$, including
$p \ne q$. The constants differ: our mixed term contains
$\Delta_0 L_0\varepsilon^{-2}(\sigma_g/\varepsilon)^{p/(q(p-1))}$ in addition
to the $\sigma_h$ term, because $h$ includes the fluctuation of the random
step. The comparison is therefore between exponents for fixed positive
noise scales; exact Hessian oracles ($\sigma_2 = 0$) are not covered. For
$p = q$ the leading exponent is that of \citet{sadiev2025second}; for
$p = q = 2$ it is the $\varepsilon^{-3}$ bound of \citet{arjevani2020second}.

\paragraph{Zero-respecting membership.} A zero-respecting algorithm
started at $x_0 = 0$ queries only points whose support lies in the union
of the supports of the oracle outputs observed so far (their
Definition~B.4); for $x_0 \neq 0$ the definition applies to $x - x_0$.
RanSOM-E with the $\ell_2$ geometry has this property: $m_t$ is a linear
combination of oracle outputs, $d_t = -\rho\, m_t/\|m_t\|_2$ has the
support of $m_t$, and $d_t = 0$ when $m_t = 0$. The same holds for the
$\ell_\infty$ geometry with $\mathrm{sign}(0) = 0$. The random step
lengths are internal randomness, which zero-respecting algorithms may
use. The spectral LMO does not preserve supports: for
$M = \bigl(\begin{smallmatrix} 1 & 1\\ 1 & 0\end{smallmatrix}\bigr)$ the polar
factor is $\frac{1}{\sqrt5}\bigl(\begin{smallmatrix} 1 & 2\\ 2 & -1\end{smallmatrix}\bigr)$,
which is nonzero where $M$ is zero. In that geometry the upper-bound
exponents still hold, but the lower bound does not transfer directly.

\paragraph{Upper bound of \citet{fradin2025tight}.} Their NSGD-Hess (their
Theorem~5.1) evaluates the HVP at an auxiliary uniform point of
$[x_{t-1}, x_t]$. Its complexity contains, in addition to the three terms
above, a term that depends on the Hessian Lipschitz constant, and they
describe it as near-optimal. RanSOM-E attains the exponents of the lower
bound without a Hessian-Lipschitz assumption and with the HVP at the
endpoint.

The comparison concerns the unconstrained problem. For the constrained
problem we are not aware of a matching lower bound in the Frank-Wolfe gap
criterion, so Theorem~\ref{thm:betasom_convergence} is stated as an upper
bound. The first-order lower bound of \citet{arjevani2023lower} does not
by itself cover algorithms with an HVP oracle; the relevant second-order
results are \citet{arjevani2020second} for $p = q = 2$ and
\citet{fradin2025tight} for general $(p,q)$, both of which grant the
algorithm stochastic Hessian-vector products.

\subsection{Convergence of RanSOM-E}
\label{app:conv-e}

\begin{theorem}[Convergence of RanSOM-E, general]
\label{thm:app:conv-e}
Let Assumptions~\ref{ass:app:wd}--\ref{ass:app:noise} hold, let $A =
(p-1)/p$ and $K = 1/q$, and let $\eta, \beta$ be constant with
\begin{equation}
  \beta \in (0, \tfrac12], \qquad
  \eta \le \bar\eta \triangleq \min\!\Bigl(\tfrac{1}{4\rho L_1},\, \tfrac{1}{\rho L_0}\Bigr),
  \label{eq:app:stepsize-constraints}
\end{equation}
where the first cap enforces the MGF condition~\eqref{eq:app:stepsize-mgf}
with $q \le 2$ and the descent condition of Lemma~\ref{lem:app:descent};
under it, all $u$-dependent constants of Definition~\ref{def:app:moments}
are replaced by their uniform upper bounds at $u = 1/4$. Define the
noise scales
\begin{gather*}
  \Sigma_g \triangleq \frac{8\,\sigma_g}{p^{1/p} B^{A}}, \qquad
  \Sigma_h \triangleq \frac{8}{q^{1/q}}\Bigl(\frac{\rho\sigma_h}{B^{(q-1)/q}} + \bar\sigma_{\mathrm{path}}\Bigr), \\
  \Lambda_g \triangleq \frac{2\alpha_g}{B^{A}}, \qquad
  \Lambda_h \triangleq \frac{2\rho\alpha_h}{B^{(q-1)/q}} + \bar\alpha_{\mathrm{path}},
\end{gather*}
with $\bar\sigma_{\mathrm{path}}, \bar\alpha_{\mathrm{path}}$ as
in~\eqref{eq:app:path-consts}, and let $G_{\mathrm{avg}} =
\tfrac1{T}\sum_{t=0}^{T-1}\mathbb{E}\|\nabla f(x_t)\|_*$. Assume the
self-coupling conditions
\begin{equation}
  8\kappa\,\Lambda_g \le \tfrac14
  \qquad\text{and}\qquad
  \frac{\eta}{\beta} \le \frac{1}{32\kappa\Lambda_h}
  \quad(\text{vacuous if } \Lambda_h = 0).
  \label{eq:app:self-coupling}
\end{equation}
Then
\begin{equation}
  G_{\mathrm{avg}}
  \;\le\; \frac{4\Delta_0}{\rho\eta T}
    + 8\kappa\,\mathbb{E}\|e_0\|_2\,\min\Bigl\{1,\frac{1}{\beta T}\Bigr\}
    + 8\kappa\Sigma_g\,\beta^{A}
    + 8\kappa\Sigma_h\,\eta\,\beta^{-K}
    + \frac{16}{3}\rho L_0\,\eta .
  \label{eq:app:master-e}
\end{equation}
\end{theorem}

\begin{proof}
Telescope~\eqref{eq:app:descent} over $t = 0,\dots,T-1$ (with
$C_{\mathrm{curv}} \le 4/3$, $C_s = 2$), take total expectations, and use
$f(x_T) \ge f_*$:
\[
  \frac{\rho\eta}{2}\sum_{t=0}^{T-1}\mathbb{E}\|\nabla f(x_t)\|_*
  \le \Delta_0 + 2\rho\kappa\eta\sum_{t=0}^{T-1}\mathbb{E}\|e_t\|_2
    + T\cdot\frac{4}{3}\rho^2 L_0\eta^2 .
\]
Dividing by $\rho\eta T/2$,
\[
  G_{\mathrm{avg}} \le \frac{2\Delta_0}{\rho\eta T} + 4\kappa\, e_{\mathrm{avg}} + \frac{8}{3}\rho L_0\eta ,
\]
where both averages run over the same $T$ iterates $x_0, \ldots,
x_{T-1}$. Substitute Lemma~\ref{lem:app:avg-error}, which is stated for
exactly this average and introduces no terminal iterate:
\[
  4\kappa\,e_{\mathrm{avg}} \le 4\kappa\,\mathbb{E}\|e_0\|_2\min\{1,\tfrac{1}{\beta T}\}
    + 4\kappa\Sigma_g\beta^A + 4\kappa\Sigma_h\eta\beta^{-K}
    + \Bigl(4\kappa\Lambda_g + \frac{4\kappa\eta\Lambda_h}{\beta}\Bigr)G_{\mathrm{avg}} .
\]
Under~\eqref{eq:app:self-coupling} each of the two coefficients of
$G_{\mathrm{avg}}$ is at most $1/8$, so their sum is at most $1/4$ and
$G_{\mathrm{avg}}(1 - 1/4)$ is bounded by the remaining terms. Multiplying
by $4/3 \le 2$ and applying the factor $2$ to \emph{every} term
gives~\eqref{eq:app:master-e}; in particular the curvature term becomes
$\tfrac{16}{3}\rho L_0\eta$.
\end{proof}

\begin{lemma}[Explicit tuning]
\label{lem:app:tuning}
Let $A, K \in (0,1)$, $D_0 = 2A + K$, $a, b, c > 0$, $d, e \ge 0$ and
$\bar\eta \in (0, \infty]$, and suppose that for all $\eta \in (0, \bar\eta]$
and $\beta \in (0, \frac12]$,
\begin{equation}
  G \le \frac{a}{\eta T} + b\beta^A + c\,\eta\beta^{-K} + d\eta
  + e\min\Bigl\{1, \frac{1}{\beta T}\Bigr\}.
  \label{eq:app:tuning-form}
\end{equation}
Set $z = (ac/(b^2T))^{1/D_0}$, $\beta_T = \min\{\frac12, z\}$,
$\eta^\star = (b/c)\,z^{A+K}$ and
$\eta_T = \min\{\bar\eta,\, \eta^\star,\, \sqrt{a/((2^K c + d)T)}\}$. Then
\[
  G \le 3S + 2Q + C + e\min\Bigl\{1, \frac{1}{\beta_T T}\Bigr\},
  \quad
  S = \bigl(a^A b^K c^A\bigr)^{1/D_0} T^{-A/D_0},\;
  Q = \sqrt{\frac{a(2^K c + d)}{T}},\;
  C = \frac{a}{\bar\eta T},
\]
with $C = 0$ if $\bar\eta = \infty$. If $b = 0$, the choice $\beta_T = \frac12$,
$\eta_T = \min\{\bar\eta, \sqrt{a/((2^Kc + d)T)}\}$ gives $G \le C + 2Q + e\min\{1, 2/T\}$.
\end{lemma}

\begin{proof}
From $z^{D_0} = ac/(b^2T)$ one checks $bz^A = c\,\eta^\star z^{-K} =
a/(\eta^\star T) = S$. Since $\eta_T$ is the minimum of three terms,
$a/(\eta_T T) = \max\{C, S, Q\} \le C + S + Q$. The gradient-noise term
satisfies $b\beta_T^A \le bz^A = S$. If $z \le \frac12$, then $\beta_T = z$,
the correction term is $c\,\eta_T z^{-K} \le c\,\eta^\star z^{-K} = S$, and
$d\eta_T \le d\sqrt{a/((2^Kc + d)T)} \le Q$. If $z > \frac12$, then
$\beta_T = \frac12$ and the correction and curvature terms sum to
$(2^K c + d)\eta_T \le Q$. In both cases they are at most $S + Q$, and
adding the four contributions gives the claim. The case $b = 0$ follows in
the same way.
\end{proof}

Since $\beta_T T = \min\{T/2,\, c_\beta T^{\theta}\}$ with
$c_\beta = (ac/b^2)^{1/D_0}$ and $\theta$ as in~\eqref{eq:init-exponent},
and $T \ge T^\theta$ for $T \ge 1$,
\begin{equation}
  \beta_T T \ge \min\{\tfrac12, c_\beta\}\, T^{\theta}.
  \label{eq:app:betaT}
\end{equation}
Since $c_\beta$ depends on the problem parameters, the tuned bounds below
state the transient as $\min\{1, 1/(\beta_T T)\}$.

\begin{remark}[Tuning, and why $q > p$ needs a separate regime]
\label{rem:app:tuning-e}
The tunable part of~\eqref{eq:app:master-e} has the form $a/(\eta T) +
b\beta^A + c\eta\beta^{-K} + d\eta$ with $a = 4\Delta_0/\rho$, $b =
8\kappa\Sigma_g$, $c = 8\kappa\Sigma_h$, $d = \tfrac{16}{3}\rho L_0$, to
be minimized over the feasible set $\beta \le 1/2$, $\eta \le \bar\eta$,
$\eta \le \beta/(32\kappa\Lambda_h)$. The unconstrained minimizer of the
two noise terms is $\beta^\star(\eta) = (K c\eta/(A b))^{1/(A+K)}$, at
which their sum is at most $2(b^Kc^A)^{1/(A+K)}\eta^{A/(A+K)}$: at
$\beta^\star$ the gradient term and the correction term have ratio $K/A$,
and $(K/A)^{A/(A+K)} + (A/K)^{K/(A+K)} \le 2$. Choosing $\beta$ to equalize
the two terms gives the same bound $2(b^Kc^A)^{1/(A+K)}\eta^{A/(A+K)}$. Balancing against $a/(\eta T)$ gives the interior
choices
\[
  \eta_T \propto T^{-\frac{A+K}{2A+K}}, \qquad \beta_T \propto T^{-\frac{1}{2A+K}},
  \qquad \frac{\eta_T}{\beta_T} \propto T^{-\frac{A+K-1}{2A+K}} = T^{\frac{q-p}{pq(2A+K)}} ,
\]
where $\propto$ hides factors that do not depend on $T$. Whether this is admissible depends on the sign of $A + K - 1 = (p-q)/(pq)$.
The cap $\beta \le 1/2$ decreases $\beta$ when the interior value exceeds
$1/2$; the clipped choice and its transient are given explicitly by
Lemma~\ref{lem:app:tuning} and~\eqref{eq:app:betaT}.

\emph{Case $\Lambda_h = 0$} ($\alpha_h = 0$ and $L_1 = 0$, the
main-text setting). The coupling condition is vacuous. Lemma~\ref{lem:app:tuning} applies directly, with
$e = 8\kappa\,\mathbb{E}\|e_0\|_2$, and gives for every $T \ge 1$ and every
$(p,q)$ the uniform finite-horizon bound
\begin{multline*}
  G_{\mathrm{avg}} \le \mathcal{O}\Bigl(
    \bigl(\tfrac{\Delta_0}{\rho}\bigr)^{\frac{A}{2A+K}}(\kappa\Sigma_g)^{\frac{K}{2A+K}}(\kappa\Sigma_h)^{\frac{A}{2A+K}}\,T^{-\frac{A}{2A+K}}
    + \sqrt{\tfrac{\Delta_0(\kappa\Sigma_h/\rho + L_0)}{T}}
    \\
    + \frac{\Delta_0}{\rho\bar\eta\,T}
    + 8\kappa\,\mathbb{E}\|e_0\|_2\min\Bigl\{1,\,\frac{1}{\beta_T T}\Bigr\}\Bigr),
\end{multline*}
in which the noise-dependent $T^{-1/2}$ term, which comes from capping
$\beta$ at $1/2$, is displayed rather than absorbed. (The $\rho$ cancels in the first and
second terms because $\Sigma_h \propto \rho$.)

\emph{Case $\Lambda_h > 0$, $q < p$.} Then $\eta_T/\beta_T \to 0$, so the
coupling condition holds for $T$ large, and the interior rate
$T^{-A/(2A+K)}$ is attained with the same finite-horizon
qualifications.

\emph{Case $\Lambda_h > 0$, $q = p$.} Then $\eta_T/\beta_T$ is constant
in $T$; choose the multiplicative constant in $\eta_T$ small enough that
$\eta_T/\beta_T \le 1/(32\kappa\Lambda_h)$. Large $T$ alone does not
suffice. The rate is $T^{-A/(1+A)} = T^{-(p-1)/(2p-1)}$, which coincides
with the interior exponent since $2A + K = 1 + A$ when $p = q$.

\emph{Case $\Lambda_h > 0$, $q > p$.} Then $\eta_T/\beta_T \to \infty$
and the interior tuning is \emph{not} admissible for large $T$. Set $\eta
= c_0\beta$ with a fixed $c_0 \le 1/(32\kappa\Lambda_h)$. Since $q > p$
gives $1 - K > A$, the correction term $c\,c_0\beta^{1-K}$ and the
curvature term $d\,c_0\beta$ are lower order than $b\beta^A$, and
balancing $a/(c_0\beta T)$ against $b\beta^A$ gives $\beta_T \propto
(a/(bc_0 T))^{1/(1+A)}$ and
\[
  G_{\mathrm{avg}} \le \mathcal{O}\!\Bigl(
    \bigl(\tfrac{\Delta_0}{\rho c_0}\bigr)^{\frac{A}{1+A}}(\kappa\Sigma_g)^{\frac{1}{1+A}}\,T^{-\frac{A}{1+A}}\Bigr)
  = \mathcal{O}\bigl(T^{-\frac{p-1}{2p-1}}\bigr),
\]
plus the transient and lower-order terms. This exponent is smaller than
the interior $A/(2A+K)$ when $q > p$; the gap reflects the coupling
restriction of this proof, not a lower bound. In summary, for fixed
$\Lambda_h > 0$ the rate established here is
$T^{-\min\{A/(2A+K),\,A/(1+A)\}}$; for $\Lambda_h = 0$ it is
$T^{-A/(2A+K)}$ for every $(p,q)$.

The first self-coupling condition is a batch-size requirement, $B^A
\ge 64\kappa\alpha_g$, vacuous when $\alpha_g = 0$. In the boundary
regime, $\beta_T T \propto T^{A/(1+A)}$, so an initial batch of
order $1/\beta_T$ makes the transient no larger than the $\beta_T^A$ term
at a cost of $\mathcal{O}(T)$ samples; the initialization and coupling
repairs are compatible. In the degenerate case $\sigma_g = 0$ ($b = 0$), take $\beta_T = \frac12$
and apply the last statement of Lemma~\ref{lem:app:tuning} with $\bar\eta$
replaced by $\min\{\bar\eta, 1/(64\kappa\Lambda_h)\}$ (no change if
$\Lambda_h = 0$): this enforces the coupling condition
$\eta/\beta \le 1/(32\kappa\Lambda_h)$ at $\beta = \frac12$ and adds a term
$64\kappa\Lambda_h a/T$; the batch-size condition remains in force. If $\Delta_0 = 0$, then $x_0$ is a
global minimizer and hence stationary, and we exclude this trivial case.
\end{remark}

\begin{corollary}[Simplified setting, $B = 1$]
\label{cor:app:conv-e-simple}
Under $\alpha_g = \alpha_h = 0$, $L_1 = 0$ and $B = 1$ we have $\Lambda_g =
\Lambda_h = 0$, $\Sigma_g = 8\sigma_g/p^{1/p}$
and, from~\eqref{eq:app:path-consts} with $\widetilde C_{A,q} = 2^{q-1}M_{ws}$
at $L_1 = 0$,
\[
  \bar\sigma_{\mathrm{path}} = \kappa\rho L_0\, 2^{\,2 - 2/q}\, M_{ws}^{1/q},
  \qquad
  \Sigma_h = \frac{8}{q^{1/q}}\,\rho\bigl(\sigma_h + 2^{\,2-2/q}M_{ws}^{1/q}\,\kappa L_0\bigr)
  \;\asymp\; \rho(\sigma_h + \kappa L_0) = \bar\sigma_h .
\]
The interior coefficient in Remark~\ref{rem:app:tuning-e} becomes
$(\Delta_0/\rho)^{\frac{A}{2A+K}}(\kappa\sigma_g)^{\frac{K}{2A+K}}(\kappa\bar\sigma_h)^{\frac{A}{2A+K}}
= \kappa^{\frac{A+K}{2A+K}}\,(\Delta_0\bar\sigma_h/\rho)^{\frac{A}{2A+K}}\sigma_g^{\frac{K}{2A+K}}$,
and since $\kappa \ge 1$ and $(A+K)/(2A+K) < 1$, the outer factor is at
most $\kappa$; this is the conservative prefactor $\kappa$ displayed in
Theorem~\ref{thm:expsom_convergence} (a further $\kappa$ remains inside
$\bar\sigma_h$). In this setting no cap on $\eta$ is needed: the descent inequality holds for every $\eta > 0$ (last paragraph of the proof of Lemma~\ref{lem:app:descent}), the Stein integrability condition reads $\int_0^\infty L_0\rho\,\bar F_s = L_0\rho\eta < \infty$, and the normalized step moments of Section~\ref{app:constants}, and the resulting constants at $L_1 = 0$, do not depend on $\eta$ (the raw moments $\mathbb{E}[s_t^k] = \Gamma(k+1)\eta^k$ do). Hence~\eqref{eq:app:master-e} holds for every $\eta > 0$, and Lemma~\ref{lem:app:tuning} applies with $\bar\eta = \infty$, so $C = 0$. The uniform finite-horizon bound of the remark then reads
\begin{multline*}
  G_{\mathrm{avg}}
  \le \mathcal{O}\Bigl(\kappa\,\bigl(\Delta_0\,\tfrac{\bar\sigma_h}{\rho}\bigr)^{\frac{A}{2A+K}}
    \sigma_g^{\frac{K}{2A+K}}\, T^{-\frac{q(p-1)}{2q(p-1)+p}}
    + \sqrt{\tfrac{\Delta_0(L_0 + \kappa\bar\sigma_h/\rho)}{T}}
    \\
    + \kappa\,\mathbb{E}\|e_0\|_2\min\Bigl\{1,\,\frac{1}{\beta_T T}\Bigr\}\Bigr),
\end{multline*}
with $\eta_T, \beta_T$ from Lemma~\ref{lem:app:tuning}, which is
Theorem~\ref{thm:expsom_convergence}. For $p = q = 2$ the leading
exponent is $1/3$. This corollary uses only the $B = 1$ argument; none of
the batch or affine machinery of Remark~\ref{rem:app:tuning-e} beyond the
$\Lambda_h = 0$ case is needed.
\end{corollary}


\subsection{Convergence of RanSOM-B}
\label{app:conv-b}

\begin{theorem}[Convergence of RanSOM-B, general]
\label{thm:app:conv-b}
Consider RanSOM-B under Assumptions~\ref{ass:app:wd}--\ref{ass:app:noise},
optimizing $f$ over a compact convex set $\mathcal{C}$ of diameter $D$
with $G_\mathcal{C} = \sup_{\mathcal{C}}\|\nabla f\|_* < \infty$ and the
effective constants~\eqref{eq:app:eff-const}. Let $\eta, \beta$ be
constant with
\begin{equation}
  \beta \in (0,\tfrac12], \qquad \eta \in (0,\tfrac12],
  \label{eq:app:stepsize-constraints-b}
\end{equation}
so that $K_t = 1/\eta - 1 \ge 1$ and the normalized moments $M_w, M_{ws},
C_s$ of the Beta step are bounded by absolute constants
(Section~\ref{app:constants}); at $\eta = 1$ the law $\mathrm{Beta}(1,0)$
is undefined and $\mathbb{E}[(w_t/\eta)^q] = K^{1-q}(K+1)^q/(K+q)$
diverges as $K \to 0$ for $q > 1$. Define
\begin{gather*}
  \Sigma_g^{\mathcal{C}} \triangleq \frac{8\,\tilde\sigma_g}{p^{1/p} B^{A}}, \qquad
  \Sigma_h^{\mathcal{C}} \triangleq \frac{8}{q^{1/q}}\Bigl(\frac{D\,\tilde\sigma_h\, \hat M_w^{1/q}}{B^{(q-1)/q}} + \kappa D L\, \hat M_{ws}^{1/q}\Bigr),
\end{gather*}
with $L = L_0 + L_1 G_\mathcal{C}$, $\tilde\sigma_g, \tilde\sigma_h$
as in~\eqref{eq:app:eff-const}, and the $\eta$-independent envelopes
$\hat M_w \ge M_w$, $\hat M_{ws} \ge M_{ws}$ of Section~\ref{app:constants}
in place of the exact moments, so that $\Sigma_h^{\mathcal{C}}$ does not
depend on $\eta$, and $\mathrm{Gap}_{\mathrm{avg}} =
\tfrac1{T}\sum_{t=0}^{T-1}\mathbb{E}[\mathcal{G}(x_t)]$. Then
\begin{equation}
  \mathrm{Gap}_{\mathrm{avg}}
  \;\le\; \frac{\Delta_0}{\eta T}
    + 2\kappa D\,\mathbb{E}\|e_0\|_2\,\min\Bigl\{1,\frac{1}{\beta T}\Bigr\}
    + 2\kappa D\,\Sigma_g^{\mathcal{C}}\,\beta^{A}
    + 2\kappa D\,\Sigma_h^{\mathcal{C}}\,\eta\,\beta^{-K}
    + L D^2\,\eta .
  \label{eq:app:master-b}
\end{equation}
No self-coupling terms appear, so no absorption step (and no doubling)
is needed: on the compact set the affine growth is absorbed into
$\tilde\sigma_g, \tilde\sigma_h$ through $G_\mathcal{C}$.
\end{theorem}

\begin{proof}
Telescope the constrained descent inequality~\eqref{eq:app:descent-b}
(with its factor $\kappa$) over $t = 0,\dots,T-1$, take total
expectations, use $f(x_T) \ge f_*$ and $C_s \le 2$, and divide by $\eta
T$:
\[
  \mathrm{Gap}_{\mathrm{avg}} \le \frac{\Delta_0}{\eta T} + 2\kappa D\, e_{\mathrm{avg}} + L D^2\eta ,
\]
with both averages over $x_0, \ldots, x_{T-1}$. For $e_{\mathrm{avg}}$
use the constrained form of Lemmas~\ref{lem:app:error}--\ref{lem:app:avg-error}
given in Remark~\ref{rem:app:avg-error-b}: the batched oracle residual
$\zeta_j$ is a martingale difference for the $\mathcal{F}$-filtration
with deterministic conditional moment $C_q M_w\eta^q D^q\tilde\sigma_h^q
/B^{q-1}$ (obtained by integrating the Beta weight over the step after
conditioning on $\mathcal{F}_{j-1}$, which is legitimate because
compactness makes the endpoint dependence uniform), and the unbatched
path residual has conditional moment $\kappa^q D^q L^q M_{ws}\eta^q$
given $\mathcal{F}_{j-1}$. Lemma~\ref{lem:app:affine-mds} with $\alpha =
0$ and Lemma~\ref{lem:app:geom} give, exactly as in
Lemma~\ref{lem:app:error} and with the $T$-point averaging of
Lemma~\ref{lem:app:avg-error},
\[
  e_{\mathrm{avg}} \le \mathbb{E}\|e_0\|_2\min\{1,\tfrac{1}{\beta T}\}
    + \Sigma_g^{\mathcal{C}}\beta^A + \Sigma_h^{\mathcal{C}}\eta\beta^{-K}.
\]
Substituting gives~\eqref{eq:app:master-b}.
\end{proof}

\begin{remark}[Tuning, constrained]
\label{rem:app:tuning-b}
The tunable part of~\eqref{eq:app:master-b} is $a/(\eta T) + b\beta^A +
c\eta\beta^{-K} + d\eta$ with $a = \Delta_0$, $b = 2\kappa
D\Sigma_g^{\mathcal{C}}$, $c = 2\kappa D\Sigma_h^{\mathcal{C}}$, $d =
LD^2$, over the feasible set $\beta, \eta \le 1/2$. There is no coupling
condition, so Lemma~\ref{lem:app:tuning} applies with $\bar\eta = 1/2$ and
$e = 2\kappa D\,\mathbb{E}\|e_0\|_2$ for every $(p,q)$, and gives the uniform
finite-horizon bound
\begin{multline*}
  \mathrm{Gap}_{\mathrm{avg}} \le \mathcal{O}\Bigl(
    \Delta_0^{\frac{A}{2A+K}}(\kappa D\Sigma_g^{\mathcal{C}})^{\frac{K}{2A+K}}(\kappa D\Sigma_h^{\mathcal{C}})^{\frac{A}{2A+K}}\,T^{-\frac{A}{2A+K}}
    + \sqrt{\tfrac{\Delta_0(\kappa D\Sigma_h^{\mathcal{C}} + LD^2)}{T}}
    \\
    + \frac{\Delta_0}{T}
    + \kappa D\,\mathbb{E}\|e_0\|_2\min\Bigl\{1,\,\frac{1}{\beta_T T}\Bigr\}\Bigr),
\end{multline*}
where the $\Delta_0/T$ term comes from the cap $\eta \le 1/2$. The
exponent is $A/(2A+K) = q(p-1)/(2q(p-1)+p)$ for all $(p,q)$; the $q > p$
obstruction of the unconstrained general case does not arise because
compactness removes the affine self-coupling.
\end{remark}

\begin{corollary}[Simplified setting, constrained, $B = 1$]
\label{cor:app:conv-b-simple}
Under $\alpha_g = \alpha_h = 0$, $L_1 = 0$ and $B = 1$ we have $L = L_0$,
$\tilde\sigma_g = \sigma_g$, $\tilde\sigma_h = \sigma_h$, and
$\Sigma_h^{\mathcal{C}} \asymp D(\sigma_h + \kappa L_0) = \bar\sigma_h$,
the constant of Theorem~\ref{thm:betasom_convergence}. The interior
coefficient becomes $\kappa^{\frac{A+K}{2A+K}}\Delta_0^{\frac{A}{2A+K}}
(D\sigma_g)^{\frac{K}{2A+K}}(D\bar\sigma_h)^{\frac{A}{2A+K}}$, with outer
factor at most $\kappa$, and the uniform bound reads
\begin{multline*}
  \mathrm{Gap}_{\mathrm{avg}}
  \le \mathcal{O}\Bigl(\kappa\,\Delta_0^{\frac{A}{2A+K}}\,(D\sigma_g)^{\frac{K}{2A+K}}\,(D\bar\sigma_h)^{\frac{A}{2A+K}}\,
    T^{-\frac{q(p-1)}{2q(p-1)+p}}
    + D\sqrt{\tfrac{\Delta_0(L_0 + \kappa\bar\sigma_h/D)}{T}}
    \\
    + \frac{\Delta_0}{T}
    + \kappa D\,\mathbb{E}\|e_0\|_2\min\Bigl\{1,\,\frac{1}{\beta_T T}\Bigr\}\Bigr),
\end{multline*}
which is Theorem~\ref{thm:betasom_convergence}. The two factors of $D$ in
the leading coefficient come from the descent inequality (which
multiplies every momentum-error contribution by $D$) and from the
length $\|d_t\| \le D$ of the Hessian-vector direction inside
$\bar\sigma_h$. For $p = q = 2$ the exponent is $1/3$.
\end{corollary}

\subsection{Unknown Horizon}
\label{app:anytime}

Theorems~\ref{thm:expsom_convergence} and~\ref{thm:betasom_convergence} choose constant parameters $\eta$ and $\beta$ that depend on the horizon $T$. This section shows that the horizon does not need to be known in advance: a restart scheme attains the same bound up to a factor $4$ in the number of steps. We state it for RanSOM-E in the main-text setting. The constrained case is identical, with $\Delta_0/\rho$ replaced by $\Delta_0$ and the cap $\eta \le \frac12$.

By Lemma~\ref{lem:app:tuning} and~\eqref{eq:app:betaT}, the tuned bound of Theorem~\ref{thm:expsom_convergence} implies $G_T \le \Phi(T)$ with
\[
  \Phi(T) = C_1 T^{-r} + C_2 T^{-1/2} + C_3 T^{-1} + C_4\,\psi(T),
  \qquad
  \psi(T) =
  \begin{cases}
    \min\bigl\{1,\, \max\{2T^{-1},\, c_\beta^{-1}T^{-\theta}\}\bigr\} & \text{if } \sigma_g > 0 \text{ and } \theta > 0,\\
    1 & \text{if } \sigma_g > 0 \text{ and } \theta \le 0,\\
    \min\{1, 2T^{-1}\} & \text{if } \sigma_g = 0,
  \end{cases}
\]
where $C_1, \ldots, C_4$ and $c_\beta$ do not depend on $T$, and $r$ and $\theta$ are as in Section~\ref{sec:theory}. Indeed, for $\sigma_g > 0$, $1/(\beta_T T) = \max\{2T^{-1}, c_\beta^{-1}T^{-\theta}\}$, and when $\theta \le 0$ we bound $\min\{1, 1/(\beta_T T)\}$ by $1$; for $\sigma_g = 0$, $\beta_T = \frac12$ and $c_\beta$ is not used. Each term of $\Phi$ is non-increasing in $T$ and of the form $C\,T^{-a}$ with $a \in [0,1]$, possibly inside $\min\{1, \cdot\}$. Since $(T/4)^{-a} = 4^aT^{-a} \le 4T^{-a}$ for $a \in [0,1]$ and $\min\{1, 4x\} \le 4\min\{1, x\}$,
\begin{equation}
  \Phi(T/4) \le 4\,\Phi(T) \qquad \text{for all } T \ge 1 .
  \label{eq:app:phi-doubling}
\end{equation}

\paragraph{Restarts.} Fix $T_0 \ge 1$ and let $T_k = 2^k T_0$. Epoch $k = 0, 1, 2, \ldots$ runs RanSOM-E from $x_0$ for $T_k$ steps, with a fresh initial batch of size $B_{\mathrm{init}}$ and with $\eta$ and $\beta$ tuned for the horizon $T_k$ by Lemma~\ref{lem:app:tuning}, and returns $x^{(k)}_R$ with $R$ uniform on $\{0, \ldots, T_k - 1\}$. When the algorithm is stopped after a total of $T \ge T_0$ steps, it returns the output of the last completed epoch.

\begin{proposition}[Restarts]
\label{prop:app:restart}
Let $K$ be the index of the last epoch completed within $T \ge T_0$ steps. Then $T_K > T/4$, and the returned point satisfies
\[
  \E\|\nabla f(x^{(K)}_R)\|_* \le \Phi(T_K) \le 4\,\Phi(T).
\]
The run uses $T$ steps, each with one gradient and one HVP, and at most $(K+2)B_{\mathrm{init}}$ initial gradients, with $K + 2 \le \log_2(T/T_0) + 2$.
\end{proposition}

\begin{proof}
Epoch $K+1$ was not completed, so $T \le \sum_{k=0}^{K+1} T_k = (2^{K+2} - 1)T_0 < 4 \cdot 2^K T_0 = 4T_K$. Epoch $K$ is a run of RanSOM-E from $x_0$ with horizon $T_K$ and parameters tuned for $T_K$, so Theorem~\ref{thm:expsom_convergence} gives $\E\|\nabla f(x^{(K)}_R)\|_* \le \Phi(T_K)$. Since $\Phi$ is non-increasing and $T_K > T/4$, $\Phi(T_K) \le \Phi(T/4) \le 4\Phi(T)$ by~\eqref{eq:app:phi-doubling}. Each started epoch uses one initial batch, and at most $K+2$ epochs are started.
\end{proof}

The factor $4$ refers to steps; the initial batches add $\mathcal{O}(B_{\mathrm{init}}\log(T/T_0))$ gradients. Restarting from $x_0$ discards the progress of earlier epochs, which costs only the constant factor above. In practice one warm-starts each epoch from the last iterate of the previous one; its analysis must also bound $\E f(x^{(k)}_0) - f_*$ at the start of each epoch, and we do not pursue it here.

\paragraph{Time-varying parameters.} A common alternative replaces $T$ by $t+1$ in the schedules, $\eta_t = \eta_0(t+1)^{-(A+K)/(2A+K)}$ and $\beta_t = \min\{\frac12, \beta_0(t+1)^{-1/(2A+K)}\}$. We do not analyze this schedule, and its behavior depends on $(p,q)$. When $2A + K < 1$ (so $\theta < 0$), the exponent $1/(2A+K)$ exceeds $1$ and the weights $\beta_t$ are summable, so the noise injected at an early step keeps a non-vanishing weight in the momentum: for $f = 0$ and $f_\xi(x) = \xi x$ with Rademacher $\xi$, the momentum error converges to a random variable with $\E|e_\infty| > 0$, while the bound $\sigma_g\beta_t^A$ that a direct transfer of Lemma~\ref{lemma:main_error} would suggest tends to zero. The boundary case $2A + K = 1$ needs a separate argument. When $2A + K > 1$, the products $\eta_t\beta_t^A$ and $\eta_t^2\beta_t^{-K}$ that enter the summed descent inequality are proportional to $1/(t+1)$, which suggests the rate of Theorem~\ref{thm:expsom_convergence} up to a factor $\log T$; a proof requires a momentum-error bound for the non-geometric weights $\prod_{i=j}^{t-1}(1-\beta_i)$, which we do not provide. The restart scheme above is the construction we analyze.

\end{document}